\documentclass[11pt,a4paper]{amsart}

\usepackage[margin=2.4cm]{geometry}                
\usepackage{graphicx}
\usepackage{amssymb}
\usepackage{color}
\usepackage{bbm, dsfont}
\usepackage[hidelinks]{hyperref}
\usepackage{upgreek}

\hypersetup{
	colorlinks=false,
	pdfborder={0 0 0},
	pdfborderstyle={/S/U/W 0},
}

\numberwithin{equation}{section}

\usepackage{verbatim}

\usepackage[dvipsnames]{xcolor}

\usepackage{ulem}

\newtheorem{theorem}{Theorem}[section]
\newtheorem{lemma}[theorem]{Lemma}
\newtheorem{corollary}[theorem]{Corollary}
\newtheorem{proposition}[theorem]{Proposition}

\theoremstyle{definition}
\newtheorem{definition}[theorem]{Definition}

\newtheoremstyle{myremark}
{}{}
{\itshape}
{}
{\itshape}
{.}
{ }
{}
\theoremstyle{myremark}
\newtheorem{remark}[theorem]{Remark}
\theoremstyle{definition}
\newtheorem*{remark*}{Note}
\newtheorem*{observation}{Observation}

\numberwithin{equation}{section}

\usepackage{amsmath}
\usepackage{mathrsfs}

\usepackage{verbatim}
\usepackage{amsmath,amssymb,amsthm}        
\usepackage{amssymb}      
\usepackage{amsfonts}     
\usepackage{mathrsfs}     
\usepackage{amsthm}
\usepackage{algorithm}  
\usepackage{algorithmicx}  
\usepackage{algpseudocode}  
\usepackage{mathtools}
\usepackage{commath}
\usepackage{physics}
\usepackage{bm}
\usepackage{graphicx}

\usepackage{float}
\usepackage{listings}
\usepackage{subfigure}
\usepackage{multirow}
\usepackage{color}
\usepackage{bbm}
\usepackage[export]{adjustbox}
\usepackage{enumerate}
\usepackage{bbm}

\usepackage{amsmath}
\usepackage{mathrsfs}

\usepackage{fancyhdr}
\begin{document}	
	
	\title{Asymptotics for the percolation threshold of finitary random interlacements in four and higher dimensions}

\author{Yijie Bi, Zhenhao Cai, Xinyi Li, Bal\'azs R\'ath, Yuan Zhang }

\address[Yijie Bi]{School of Mathematical Sciences, Peking University}
		\email{2200010925@stu.pku.edu.cn}

\address[Zhenhao Cai]{Faculty of Mathematics and Computer Science, Weizmann Institute of Science}
		\email{zhenhao.cai@weizmann.ac.il}
	
\address[Xinyi Li]{Beijing International Center for Mathematical Research, Peking University}
		\email{xinyili@bicmr.pku.edu.cn}

\address[Bal\'azs R\'ath]{Department of Stochastics, Institute of Mathematics, Budapest University of Technology and Economics; HUN-REN Alfr\'ed R\'enyi Institute of Mathematics}
\email{rathb@math.bme.hu}

\address[Yuan Zhang]{Center for Applied Statistics and School of Statistics, Renmin University of China} \email{zhang\_probab@ruc.edu.cn}
        
\begin{abstract}
    We establish sharp asymptotic bounds for the critical intensity of the Finitary Random Interlacements (FRI) model in four and higher dimensions with general trajectory length distributions. Our proof reveals that the construction of near-critical FRI clusters in four and higher dimensions is essentially analogous to a Galton-Watson process, whose expected number of offspring corresponds to the capacity of a random walk killed at the given length. 
\end{abstract}
	\maketitle
\section{Introduction}

Finitary Random Interlacements (FRI), introduced by Bowen in \cite{bowen2019finitary}, is a Poisson cloud of random walk trajectories on $\mathbb Z^d$ where the trajectory lengths follow a specified distribution. As the expected length varies from zero to infinity, the model interpolates between the renowned Bernoulli percolation and Random Interlacements (RI). The latter was introduced by Sznitman in \cite{sznitman2010vacant} and has been extensively studied over the last decade (e.g. \cite{teixeira2011fragmentation}, \cite{10.1214/10-AOP545}, \cite{sznitman2017disconnection}, \cite{sznitman2012decoupling}, \cite{10.1214/09-AOP450}, \cite{10.1214/11-AOP683}; see \cite{drewitz2014introduction} for a self-contained introduction).

The FRI model, denoted by $\mathcal I^{u,\rho}$ (where $u>0$ is a fixed number and $\rho$ is a distribution supported on non-negative integers), can be constructed as follows. Independently for each vertex $x\in\mathbb Z^d$($d\ge3$), we initiate from $x$ a Poisson$\left(\frac u{\mu_1(\rho)+1}\right)$-distributed number of independent simple random walks $(X_n)_{n\ge0}$ killed at time $T\sim\rho$, where $\mu_i(\rho)$, $i\in\mathbb N$ is the $i$-th moment of $\rho$. The process $\mathcal I^{u,\rho}$ is then constructed as the subgraph of $\mathbb Z^d$ whose edges are those traversed by the aforementioned random walk trajectories. The precise definition(s) of FRI can be found in Section 2 of this paper. The parameter $u$ is the analogue of the intensity parameter in RI, which corresponds to the expected total occupation time at each vertex by the random walk trajectories. The length distribution $\rho$ is typically taken to be geometric in the literature, i.e., one considers the model $\mathcal {FI}^{u,T}\coloneq\mathcal I^{u,{\rm Geo}(\frac{1}{T+1})}$ for $T\ge0$, where ${\rm Geo}(\lambda)$ represents the geometric distribution with success probability $\lambda$; also, the model $\mathcal I^{u,\delta_n}$, where $\delta_n$ is the Dirac measure at $n$, has been studied as a finite-range approximation of RI (see \cite{duminil2023finite} for an example). In this paper, we consider general length distributions $\rho$ on $\mathbb N$ satisfying a mild restriction, which encompass two common families of distributions, namely the geometric distributions and the Dirac mass.

For fixed $u$ and varying $\rho$, the behavior of $\mathcal I^{u,\rho}$ constrained to any finite subset of $\mathbb Z^d$ resembles that of RI with intensity $u$, denoted by $\mathcal I^u$, when $\mu_1(\rho)\to\infty$. In the geometric setup, this property was discussed in \cite{bowen2019finitary} and made rigorous in \cite{cai2021non}. Moreover, couplings between FRI and RI have proven instrumental in the study of the latter. For example, in \cite{10.1214/23-EJP950}, it was established that RI is a factor of i.i.d.\ using the convergence in law of a sequence of similarly defined finite-range models to RI. Moreover, recent breakthroughs concerning the sharpness of phase transition of the vacant set of RI \cite{duminil2024characterization}, \cite{duminil2023finite} and \cite{duminil2023phase} rely crucially on couplings both between FRI and RI and between FRIs of various lengths to overcome obstacles imposed by the long-range dependence of RI.

In contrast to RI, where $\mathcal I^u$ almost surely percolates for all $u>0$ and $d\ge3$, FRI undergoes phase transitions with respect to both parameters as a bond percolation model on $\mathbb Z^d$. On the one hand, for fixed $u>0$, the percolation properties of $\mathcal {FI}^{u,T}$ as a function of $T$ are studied in several papers. Bowen proved the existence of infinite clusters (connected components) in $\mathcal {FI}^{u,T}$ for any fixed $u>0$ and sufficiently large $T$ in \cite{bowen2019finitary}. It was shown in \cite{procaccia2021percolation} that $\mathcal {FI}^{u,T}$ has a non-trivial phase transition with respect to $T$ for all $u>0$ and $d\ge3$. However, it was proved in \cite{cai2021non} that $\mathcal{FI}^{u,T}$ is non-monotone for $T\in(0,\infty)$ due to the fact that the expected number of trajectories from each vertex decreases as $T$ increases. As a result, the existence and uniqueness of a critical length parameter $T_c$ remain open.

On the other hand, if we fix the length distribution $\rho$ and study the percolation as $u$ varies, the phase transition is unique, because $\mathcal I^{u,\rho}$ is increasing in $u$. For FRI with geometric length distributions, this phase transition has been studied by several papers. In \cite{cai2021some}, it was proved that for each fixed $T>0$, there exists $u_*(d,T)\in (0,\infty)$ such that the probability of the existence of an infinite cluster in $\mathcal{FI}^{u,T}$ is zero for $u\in(0,u_*(d,T))$ and positive for $u\in(u_*(d,T),\infty)$. For the asymptotic of $u_*(d,T)$, it was shown in \cite{cai2021some} that $u_*(d,T)\cdot T\to-\log(1-p_d^c)/2$ as $T\to 0$, where $p_d^c$ is the critical parameter for Bernoulli bond percolation on $\mathbb{Z}^d$. Moreover, it was proved in \cite{cai2023exact} that for $d\ge 3$, there exist constants $C_{\dagger}(d)>c_{\dagger}(d)>0$ such that for all sufficiently large $T$, 
\begin{equation}\label{eq:previous_bound}
    c_{\dagger}f_d(T)\le u_*(d,T)\le C_{\dagger}f_d(T),
\end{equation}
where $f_d(T)$ is given by
\begin{equation*}
    f_d(T):=\left\{
        \begin{aligned}
            T^{-1/2},\text{ }&d=3;\\
            T^{-1}\log T,\text{ }&d=4;\\
            T^{-1},\text{ }&d\geq 5.
        \end{aligned}
        \right.
\end{equation*}
Interestingly, this implies that FRI achieves supercriticality in percolation with vanishing density.
It was conjectured in \cite{cai2023exact} that $u_*(d,T)$ admits an asymptotic limit, i.e., the coefficients $C_\dagger$ and $c_\dagger$ can be arbitrarily close. Estimates similar to (\ref{eq:previous_bound}) are established for the percolation of Wiener sausages with fixed length in \cite{erhard2015asymptotics}, and the analogous conjecture for the critical threshold of Wiener sausages was posed in Section 1.4 of that paper.

In this paper, we confirm the conjecture for FRI when $d\geq4$, and further generalize it to FRI with general length distributions. 

Let $\{ 0\stackrel{\mathcal I^{u,\rho}}{\longleftrightarrow}\infty\}$ be the event that the origin is in an infinite cluster of $\mathcal I^{u,\rho}$. For any length distribution $\rho$, we define the critical intensity of the FRI model $\mathcal I^{u,\rho}$ by
$$u_*(\rho):=\inf\{u\geq0:\mathbb{P}[\bm{0}\stackrel{\mathcal I^{u,\rho}}{\longleftrightarrow}\infty]>0\}.$$ 
We aim to establish the asymptotic of $u_*$ for a family of length distributions $(\rho_n)_{n\in \mathbb{N}}$ with expectations tending to infinity. For technical reasons, we need certain restriction on the fluctuation of each $\rho_n$. Specifically,
\begin{definition}\label{def:appropriate_family}
        We say a family $(\rho_n)_{n\in\mathbb{N}}$ of distributions on $\mathbb{N}$ is appropriate if $\mu_1(\rho_n)\to\infty$ as $n\to\infty$, and 
        for all $\epsilon>0$, there exists $C=C(\epsilon,(\rho_n)_{n\in\mathbb{N}})<\infty$ satisfying
        \begin{equation}\label{eq:appropriate}
            \vartheta(\rho_n,C)\coloneq\sum_{L=C\mu_1(\rho_n)}^\infty\frac{\rho_n(L)\cdot L^2}{\mu_2(\rho_n)}\leq\epsilon,\forall n\in \mathbb{N},
        \end{equation} 
        where $\mu_i$ means the $i$-th moment.
\end{definition}
\begin{remark}\label{rmk:geodirac}
    For any random variable $\xi$ supported on $(0,\infty)$ with finite second moment and any sequence of positive numbers $l(n)\to\infty$, if we define $\rho_n$ to be the law of $\lfloor l(n)\cdot\xi\rfloor$, then (\ref{eq:appropriate}) holds. In particular, $(\rho_n)_{n\in\mathbb N}=(\delta_n)_{n\in\mathbb N}$ corresponds to $\xi\equiv1$ and $l(n)=n$; $(\rho_n)_{n\in\mathbb N}=({\rm Geo}(\frac1{n+1}))_{n\in\mathbb N}$ corresponds to $\xi\sim{\rm Exp}(1)$ (the exponential distribution with killing rate $1$, see Section 2 for definition) and $l(n)=\log^{-1}(1+\frac1n)$. Thus, $(\delta_n)_{n\in\mathbb N}$ and $({\rm Geo}(\frac1{n+1}))_{n\in\mathbb N}$ are both appropriate families.
\end{remark}
\begin{remark}
    This restriction may not be optimal in the 4-dimensional case. In fact, it can be relaxed to \begin{equation}\label{eq:appropriate'}              \sum_{L=C\mu_1(\rho_n)}^\infty\frac{\rho_n(L)\cdot L^2}{\mu_2(\rho_n)}\cdot\frac{\log \mu_1(\rho_n)}{\log L}\leq\epsilon.
    \end{equation}                                  
\end{remark}
The main result is the following theorem.
\begin{theorem}\label{thm:main}
    For any $d\ge4$ and any appropriate family of distributions $(\rho_n)_{n\in\mathbb{N}}$,
        \begin{equation}
            \label{eq:main5}\lim_{n\to\infty}u_*(\rho_n)\cdot\frac{\mu_2(\rho_n)}{\mu_1(\rho_n)(1+\log\mu_1(\rho_n)\cdot\mathbbm{1}_{d=4})}\cdot \varepsilon_d=1.
        \end{equation}
    Here, $\varepsilon_d$ is the coefficient in the asymptotic capacity of random walks, i.e.,
    $$\varepsilon_d=\lim_{T\to\infty}\frac{\mathbb E[{\rm cap}(X[0,T])]}{T(1+\log T\cdot\mathbbm 1_{d=4})^{-1}},$$
    where ${\rm cap}(A)$ for $A\subset\mathbb Z^d$ represents the capacity of the set $A$ (see Section 2 for the definition). (The existence of $\varepsilon_d\in(0,\infty)$ when $d\ge5$ was proved in \cite{jain1968range}; while one has $\varepsilon_4=\frac{\pi^2}8$ according to \cite[Corollary 1.4]{10.1214/18-AOP1288}.)
\end{theorem}

In particular, we obtain the following sharp asymptotics of $u_*(d,n)$, the threshold under geometric distribution (see above \eqref{eq:previous_bound} for notation).
\begin{corollary} The critical intensity of FRI of geometric length admits an asymptotic limit:
    $$\lim_{n\to\infty}u_*(d,n)\cdot2n\cdot \varepsilon_d=1$$
    for all $d\geq5$, and
    $$\lim_{n\to\infty}u_*(4,n)\cdot n\log^{-1}n\cdot\frac {\pi^2}{4}=1.$$
\end{corollary}

We now give a brief overview of the proof of Theorem \ref{thm:main}. Let 
\begin{equation}\label{eq:def_un}
    u_n:=\frac{\mu_1(\rho_n)(1+\log\mu_1(\rho_n)\mathbbm{1}_{d=4})}{\varepsilon_d\mu_2(\rho_n)}
\end{equation}denote the order of $u_*(\rho_n)$, and for a fixed $\epsilon>0$, let 
\begin{equation}\label{eq:upm}
    u_n^{\pm}:=(1\pm \epsilon)u_n.
\end{equation}
To derive Theorem \ref{thm:main}, it suffices to show that for all large $n$,
\begin{equation}\label{new1.4}
 u_n^-  \le  u_*(\rho_n)\le u_n^+. 
\end{equation}
In other words, we need to show that $\mathcal I^{u_n^+,\rho_n}$ (resp. $\mathcal I^{u_n^-,\rho_n}$) is supercritical (resp. subcritical). 

To prove the subcriticality of $\mathcal I^{u_n^-,\rho_n}$, we explore the FRI cluster containing the origin layer by layer. Precisely, we call the origin the $0$-th layer; given the $i$-th layer, the $(i+1)$-th layer is given by the union of trajectories hitting the $i$-th layer but not the previous ones. In particular, the typical number of trajectories in one layer is approximately the intensity times the capacity of the previous layer. Therefore, when taking $u=u_n^-$, the number of discovered trajectories in the aforementioned exploration process can be dominated by a subcritical Galton-Watson process, which dies out in finite time almost surely. 

To verify that $\mathcal I^{u_n^+,\rho_n}$ is supercritical, we construct the infinite cluster through a classic coarse-graining argument inspired by \cite{penrose1993spread}, where a similar asymptotic upper bound for the critical intensity of spread-out percolation was proved. In our setup, some additional delicate treatment is needed to handle the correlation generated by random walk trajectories. Specifically, we introduce a truncated version of $\mathcal I^{u_n^+,\rho_n}$ which ignores certain types of trajectories (e.g., those with atypically large diameters or long durations), and run an exploration process to show the supercriticality on this truncated version instead. We divide the space $\mathbb{Z}^d$ into disjoint boxes of side length $R_n\coloneq\lfloor\sqrt{\mu_1(\rho_n)}\rfloor$ (the correlation length of a random walk killed at time $T\sim \rho_n$, see also (\ref{eq:intro_Rn})). Next, we sample the cluster containing $\bm{0}$ (more precisely, the center point of the box around $\bm 0$, see (\ref{eq:intro_zn})) with a specific part of trajectories within the $R_n$-box around $\bm{0}$, and require it to satisfy several conditions including a lower bound on the capacity and some restriction on the relative position of any two trajectories (we call this cluster a ``seed''). We then try to extend this seed to all of the neighboring boxes, and call the box around $\bm0$ a ``good'' box if a seed is sampled in each of the neighboring boxes. Next, we resume the exploration from one of the neighboring boxes, and call it good if new seeds are sampled for all its unexplored neighbors. When a bad box appears, to ensure independence we give up the exploration of all unexplored boxes within a neighborhood of radius $2\gamma R_n$ concentric with this bad box that are neighbors of good ones where $\gamma$ is an integer defined in (\ref{eq:intro_gamma}) so large that trajectories in subsequent exploration from boxes that are not ruined do not intersect those explored from the current bad box (thanks to the truncation).

Note that in the exploration algorithm described above, we are able to choose parameters properly so that the conditional probability of each step (i.e., extension to neighboring boxes) is sufficiently close to $1$.  Consequently, the exploration, after proper rescaling, may stochastically dominate a finitely dependent percolation configuration with sufficiently high density, which is supercritical by an application of the Liggett--Schonmann--Stacey domination. 

We now briefly comment on the truncation. Note that in the exploration algorithm, it would be easier to obtain the desired independence if we ignored more trajectories. Meanwhile, such modification is not allowed to increase the critical intensity by a constant factor (otherwise, it will violate (\ref{new1.4})). These requirements exhibit opposite directions of monotonicity, which entails a delicate truncation scheme, which stands as one of the main technical novelties of this paper.

Let us also quickly remark on the $d=3$ case. Since it is much more likely for two random walk trajectories to intersect in three dimensions, the behavior of near-critical FRI clusters can no longer be characterized by Galton-Watson processes, and the corresponding asymptotic threshold is actually determined by a different mechanism. In fact, in a paper in progress \cite{3D}, the pre-factor of the asymptotic threshold is upper and lower bounded by two (plausibly equal) critical thresholds of finitary Brownian interlacements (FBI), a counterpart of FRI in the continuum. A model similar to FBI in which trajectories are of fixed length is studied in \cite{erhard2015asymptotics}, where the existence of a non-trivial phase transition is proved (c.f.\ Theorem 1.3 therein).

The remaining sections are organized as follows. In Section 2, we introduce the notations and some preliminary results. In Sections 3, we prove the lower bound of (\ref{new1.4}). Sections 4 and 5 are dedicated to the corresponding upper bound. In Section 4, we introduce the exploration algorithm and reduce the supercriticality of $\mathcal I^{u_n^+,\rho}$ to the fact that the conditional probability of the extension to a neighboring box in the algorithm can be arbitrarily close to $1$ (c.f. Proposition \ref{uniform_low}). Section 5 proves Proposition \ref{uniform_low} with various coupling techniques.

\section{Notations and Preliminaries}
    We write $\mathbb{N}=\{0,1,2,...\}$ for the set of natural numbers. For a real number $a$, write $\lfloor a\rfloor$ for its integer part, and for real numbers $b,c$, we write $b\wedge c$ and $b\vee c$ for the minimum and maximum of $b$ and $c$ respectively. Denote the $\ell^\infty$ norm on $\mathbb{Z}^d$ by $|\cdot|$ and the $\ell^1$ norm by $|\cdot|_1$. For two non-empty subsets $A,B\subset\mathbb{Z}^d$, define the distance between them by ${\rm d}(A,B):=\min_{x\in A,y\in B}|x-y|$, and write ${\rm d}(\{x\},B)$ as ${\rm d}(x,B)$ for all $x\in\mathbb{Z}^d$. The diameter of a finite nonempty set $A\subset\mathbb Z^d$ is defined by
    ${\rm diam}(A):=\max_{y,z\in A}{\rm d}(y,z).$ For all $x\in\mathbb{Z}^d$ and $r\in(0,\infty)$, we denote the subset of $\mathbb Z^d$ inside the box of side length $2r$ centered at $x$ by $B(x,r)=B_x(r)\coloneq\{y\in\mathbb{Z}^d:|x-y|\leq r\}$; for all $A\subset\mathbb Z^d$, let $B(A,r)\coloneq \bigcup_{x\in A}B(x,r)$ be the collection of vertices in the $r$-neighborhood of $A$. For all $n\in\mathbb{N}$ and $A\subset\mathbb{Z}^d$, let $nA=\{nx:x\in A\}$. In particular, we will consider the set $n\mathbb Z^d$ which will be referred to as a lattice. For all $x,y\in n\mathbb Z^d$, they are called nearest neighbors if $|x-y|_1=n$. A nearest-neighbor path on $n\mathbb Z^d$ is a sequence of vertices on $n\mathbb Z^d$ such that the adjacent vertices are nearest neighbors.
    For $\lambda\in(0,1]$, define the geometric distribution ${\rm Geo}(\lambda)$ by ${\rm Geo}(\lambda)(n)=\lambda(1-\lambda)^n$, $n\in\mathbb N$. For $\lambda\in(0,\infty)$, define the exponential distribution ${\rm Exp}(\lambda)$ by the density function $f(x)=\lambda e^{-\lambda x}\mathbbm1_{x>0}$.
    
    From now on, we always assume that the dimension is at least $4$, and often omit $d$ in the notation when defining objects such as random walks and quantities including constants and parameters that do not explicitly depend on $d$. All the results stated in this paper are supposed to hold for all $d\geq4$ unless stated otherwise. The constants $C$ and $C'$ appearing in the paper may implicitly depend on $d$, and may differ from place to place. However, numbered constants such as $C_1,C_2,...$ refer to their first appearance in the text.
    
    \hspace{\fill}
    
    \textbf{Measures on countable spaces}: For an arbitrary countable space $\mathscr X$ equipped with the canonical $\sigma$-field, we denote by $\mathcal{M}(\mathscr X)$ (resp. $\mathcal M^*(\mathscr X)$) the set of all probability measures (resp. finite measures) on $\mathscr X$, and abbreviate $\mathcal M=\mathcal M(\mathbb N)$. For two measures $\nu_1,\nu_2\in\mathcal M(\mathscr X)$, we write $\nu_1\le\nu_2$ if 
    $\nu_1(A)\le \nu_2(A)$ for all $A\subset\mathscr X$.
    For all $\nu\in\mathcal M^*(\mathscr X)$, we define the corresponding probability measure $\nu^0\in\mathcal M(\mathscr X)$ by
    $$\nu^0(A)=\frac{\nu(A)}{\nu(\mathscr X)},\;\forall A\subset\mathscr X.$$
    
    For $\rho\in\mathcal M(\mathscr X),x\in\mathscr X$, let $\rho(x)$ be the mass assigned to $x$. We call $\rho$ a point measure if $\rho(x)\in\mathbb N$ for all $x\in\mathscr X$. Let $\mathcal S$ be a multi-subset of $\mathscr X$. We sometimes identify $\mathcal S$ with its corresponding point measure on $\mathscr X$, i.e., we also use $\mathcal S$ to represent the point measure
    $\sum_{x\in\mathscr X}n_x(\mathcal S)\cdot \delta_x$,
    where $n_x(\mathcal S)$ is the multiplicity of $x$ in $\mathcal S$. Write
    $|\mathcal S|=\sum_{x\in\mathcal S}n_x(\mathcal S)$.
    For a point measure $\mathcal X\in\mathcal M(\mathscr X)$ and $A\subset\mathscr X$, write
    $\mathcal X\cap A=\sum_{x\in A}\mathcal X(x) \cdot\delta_x$.
    
    Let $k$ be a positive integer. For all $\rho\in\mathcal M$, denote the $k$-th moment of $\rho$ by
    \begin{equation}
    \mu_k(\rho):=\displaystyle\sum_{l=0}^\infty l^k\rho(l).
    \end{equation}
    Also, we define the ``tail contribution'' to the $k$-th moment by
    $$\mu_k^{(m)}(\rho)\coloneq\sum_{l=m}^\infty l^k\rho(l),$$
    where $m\in\mathbb N$.

    \hspace{\fill}
    
        \textbf{Trajectories:} We will refer to nearest-neighbor paths on $\mathbb Z^d$ as trajectories. Let $W^{[0,\infty)}$ denote the set of all finite trajectories. For all $\eta\in W^{[0,\infty)}$, write $\eta=(\eta(s))_{0\leq s\leq T}$. Here, $T(\eta):=T$ denotes the length of $\eta$. In addition, we define the range of $\eta$ by ${\rm range}(\eta):=\{\eta(s):0\leq s\leq T\}$. Sometimes, $\eta$ is identified with its range when no confusion arises (e.g. $B(\eta,r)=B({\rm range}(\eta),r)$). For $0\leq a\leq b\leq T$, define the sub-path $\eta[a,b]:=(\eta(a+s))_{0\leq s\leq b-a}$.

    For a sigma-finite point measure $\omega$ on the space $W^{[0,\infty)}$, we write $\eta\in\omega$ if $\omega(\{\eta\})>0$, and define the set of sites visited by the trajectories of $\omega$ by $V(\omega):=\bigcup_{\eta\in\omega}{\rm range}(\eta)$, and let $G(\omega)$ be the subgraph of $\mathbb{Z}^d$ where an edge is open iff it connects two adjacent vertices appearing in one of the trajectories in $\omega$. Sometimes, the notation is slightly abused and the point measure $\omega$ is identified with the corresponding (multi-)subset of $W^{[0,\infty)}$.

    Let $\eta_i\in W^{[0,\infty)}$ be a trajectory with length $T_i$, $i=1,2$. Define $\eta_1*\eta_2$, the concatenation of $\eta_1$ and $\eta_2$, by
    $$\eta_1*\eta_2(s)\coloneq\begin{cases}\eta_1(s),0\leq s\leq T_1;\\\eta'_2(s-T_1),T_1<s\leq T_1+T_2,\end{cases}$$
    where $\eta_2'=(\eta_2(s)-\eta_2(0)+\eta_1(T_1))_{0\leq s\leq T_2}$ is $\eta_2$ translated to the endpoint of $\eta_1$. We write $\eta_1\simeq\eta_2$ if they are a spatial translations of each other.
    If $\eta_1$ and $\eta_2$ are both (one-sided) infinite trajectories, we write $\eta_1\simeq\eta_2$ if $\eta_1[0,T]\simeq\eta_2[0,T]$ for every $T\in\mathbb N$.
    
    For an arbitrary trajectory $\eta$ (finite, infinite) and $A\subset\mathbb{Z}^d$, let 
    \begin{equation}\label{eq:def_hitting}
        \tau_A=\tau_A^\eta\coloneq\inf\{t\in\mathbb N:\eta(t)\in A\}.
    \end{equation}For any $B\subset A$, we say $\eta$ hits $A$ in $B$ if $\tau_A<\infty$ and $\eta(\tau_A)\in B$. Define the set of all finite trajectories hitting $A$ in $B$ by
    \begin{equation}\label{eq:def_cH}
        \mathcal{H}[A;B]:=\{\eta\in W^{[0,\infty)}:\eta\mbox{ hits }A\mbox{ in }B\}.
    \end{equation}
    
    \hspace{\fill}
    
    \textbf{Simple random walk on }$\mathbb{Z}^d$\textbf{ and potential theory:} For $x\in\mathbb Z^d$, we denote by $P^x$ the law of a doubly infinite simple random walk starting from $x$, and denote the walk by $X$.  For all $y,z\in\mathbb Z^d$, let $P^{y,z}$ be the joint law of two independent simple random walks $X^1\sim P^y,X^2\sim P^z$.

    For all $A\subset\mathbb Z^d$, we define the first hitting/entrance time of $A$ by $X$ by
    \begin{align*}\
        H_A&\coloneq\inf\{t\ge0:X_t\in A\},
        \\\widetilde H_A&\coloneq\inf\{t>0:X_t\in A\}.
    \end{align*}
    
    To lighten the notation, for $m,l\in\mathbb N$,  let
    \begin{equation}\label{eq:def_muxml}
        \mu(x,m,l;A)(\cdot)=P^x[X[-m,l]\in\cdot\big|X[-m,-1]\cap A=\emptyset]
    \end{equation}
    be a probability measure on $W^{[0,\infty)}$, and we omit ``$A$'' when $A=\emptyset$ (i.e., no conditioning is imposed). For all $R>0$, let $W_{R}\subset W^{[0,\infty)}$ denote the subset of trajectories with diameter at most $R$. We have the following estimate on the diameter of a trajectory following the distribution $\mu(x,m,l)$ whose proof follows from a simple application of the Donsker's invariance principle.
    \begin{lemma}\label{lemma2.1}
        There exists a decreasing function $f:(0,\infty)\to(0,\infty)$ such that for all $\kappa>0$, $R>C(\kappa)$ large and $m+l<\kappa R^2$,
        $$\mu(x,m,l)(W_R)\geq f(\kappa).$$
    \end{lemma}
    For a finite subset $A\subset\mathbb{Z}^d$, define its equilibrium measure by $e_A(x):=P_x[\widetilde{H}_A=\infty]\cdot\mathbf{1}_{x\in A}$, and define the capacity of $A$ by ${\rm cap}(A):=\sum_{x\in A}e_A(x)$. The capacity is monotone and sub-additive, namely
    $$\begin{aligned}{\rm cap}(A)&\leq{\rm cap}(D),\\{\rm cap}(A\cup B)&\leq{\rm cap}(A)+{\rm cap}(B)\end{aligned}$$
    for all finite $A,B,D\subset\mathbb Z^d,A\subset D$.
    By (2.16) in \cite{trove.nla.gov.au/work/15980712}, there are constants $0<C<C'<\infty$, such that for all $D\in\mathbb N$,
    \begin{equation}\label{eq:cap_ball}
        Cr^{d-2}\leq{\rm cap}(B(0,r))\leq C'r^{d-2}.
    \end{equation}
    For all $x,y\in\mathbb Z^d$, define the Green's function on $\mathbb Z^d$ by 
    $$g(x,y)=E^x\left[\sum_{n=0}^\infty\mathbf 1_{X_n=y}\right].$$ By Theorem 1.5.4 of \cite{trove.nla.gov.au/work/15980712}, there exist constants $0<C<C'<\infty$ such that
    \begin{equation}\label{eq:green's_estimate}
        C(1\vee|x-y|)^{2-d}\le g(x,y)\leq C'(1\vee|x-y|)^{2-d}
    \end{equation}
    for all $x,y\in\mathbb Z^d$.
    For all $A\subset\mathbb{Z}^d$ and $x\in\mathbb{Z}^d$, define the hitting probability by $h(x,A):=P^x[H_A<\infty]$. When $A$ is finite, by the last exit decomposition (see Proposition 2.4.1 of \cite{trove.nla.gov.au/work/15980712}), we have
    \begin{equation}\label{eq:hit_decompose}
        h(x,A)=\sum_{y\in A}g(x,y)e_A(y).
    \end{equation}
    Combining (\ref{eq:green's_estimate}) and (\ref{eq:hit_decompose}) gives the following: there exist constants $0<C<C'<\infty$ such that for all finite $A\subset\mathbb Z^d$ and $x\in\mathbb Z^d$,
    \begin{equation}\label{eq:hitting}
        {\rm cap}(A) \cdot C(1\vee d(x,A))^{2-d}\le h(x,A)\leq{\rm cap}(A) \cdot C'(1\vee d(x,A))^{2-d}.
    \end{equation}

    By \cite[Theorem 1.1]{asselah2018capacity}, \cite[Theorem B]{schapira2020capacity} and \cite[Theorem 1.2]{10.1214/18-AOP1288} (corresponding to $d\ge6$, $d=5$ and $d=4$ respectively), we have laws of large numbers and concentration results for the capacity of random walk trace in dimension 4 and above.

    \begin{proposition}\label{trace_cap_origin_new}
    As $T\in\mathbb N$ tends to infinity, we have
    \begin{equation}\label{eq:cap_ex}
        E^{\bm0}[{\rm cap}(X[0,T])]=        (1+o(1))\cdot\varepsilon_dT(1+\log T\cdot\mathbbm 1_{d=4})^{-1}.
    \end{equation}
    \begin{equation}\label{eq:cap_lln}
        {\rm cap}(X[0,T])\cdot T^{-1}(1+\log T\cdot\mathbbm1_{d=4})\stackrel{P}\longrightarrow\varepsilon_d.
    \end{equation}
    Moreover, there exists some $C=C(d)$ such that
    \begin{equation}\label{eq:cap_var}
        {\rm var}({\rm cap}(X[0,T]))<\begin{cases}CT^2\log^{-4}T,&d=4;\\CT\log T,&d\ge5.\end{cases}
    \end{equation}
    \begin{equation}\label{eq:cap_con}
        P^{\bm0}\left[\left|{\rm cap}(X[0,T])-E^{\bm0}[{\rm cap}(X[0,T])]\right|>\delta E^{\bm0}[{\rm cap(X[0,T])}]\right]<\begin{cases}C\delta^{-2} \log^{-2}T,&d=4;\\C\delta^{-2}T^{-1}\log T,&d\geq5.\end{cases}
    \end{equation}
\end{proposition} 

To analyze hitting of a given set by finite paths, we consider the following quantity that formally resembles the capacity of a set.
For all finite $A\subset \mathbb Z^d$, define ${\rm cap}(A,T)$ by
\begin{equation}\label{eq:def_capT}
    {\rm cap}(A,T)=\sum_{x\in A}P^x[\widetilde{H}_A>T\log^{-10}T].
\end{equation}
(Here, the choice of $-10$ in the exponent is rather arbitrary. In fact, any real number smaller than $-2$ is sufficient for Proposition \ref{trace_cap_truncated_new} below.) We can prove estimates on this quantity analogous to Proposition \ref{trace_cap_origin_new} by directly comparing ${\rm cap}(\cdot)$ with ${\rm cap}(\cdot,T)$.
\begin{proposition}\label{trace_cap_truncated_new}
    For all $r>0$, the following estimates hold.
    \begin{equation}\label{eq:Tcap_ex}
        E^{\bm0}[{\rm cap}(X[0,T],rT)]=        (1+o(1))\cdot\varepsilon_dT(1+\log T\cdot\mathbbm 1_{d=4})^{-1}.
    \end{equation}
    Moreover, for all $\delta>0$, there exists some $C=C(d,\delta)>0$ such that
    \begin{equation}\label{eq:Tcap_con}
        P^{\bm0}\left[\left|{\rm cap}(X[0,T],rT)-E^{\bm0}[{\rm cap}(X[0,T])]\right|>\delta E^{\bm0}[{\rm cap(X[0,T])}]\right]<\begin{cases}C \log^{-1/2}T,&d=4;\\CT^{-1/4},&d\geq5\end{cases}
    \end{equation}
    holds for all large $n\in\mathbb N$.
\end{proposition}

\begin{proof}
    We only present the proof of the 4-dimensional case since the higher dimensional case is simpler and can be treated similarly. Assume that $d=4$, then for all $2T\log^{-10}T<t<T-2T\log^{-10}T$, we have
    \begin{equation}\label{eq:desc}
        \begin{aligned}
            &E^{\bm0}[P^{X_t}[\widetilde H_{X[0,T]}>T\log^{-10}T]-e_{X[0,T]}(X_t)]\\=\;& P^{0,0}\left[X^1[-t,T-t]\cap X^2(0,t']=\emptyset,\ X^1[-t,T-t]\cap[t',\infty)\neq\emptyset\right],
        \end{aligned}
    \end{equation}
    where $t'=\lfloor T\log^{-10}T\rfloor$. Let $E_1,E_2,E_3$ denote the events
    $$E_1\coloneq\{{\rm d}(X^2_{t'},X^1[-t',t'])>(t')^{1/2}\log^{-1/4}T\},$$
    $$E_2\coloneq\{X^1[-t',t']\cap X^2(0,t']=\emptyset\},$$
    $$E_3\coloneq\{{\rm cap}(X^1[-t',t'])\le 4\varepsilon_4t'\log^{-1}T\}.$$
    By \cite[Theorem 3.5.1]{trove.nla.gov.au/work/15980712},
    \begin{equation}\label{eq:E2}
        P^{0,0}[E_2]\le C\log^{-1}T.
    \end{equation}
    By (\ref{eq:cap_con}),
    \begin{equation}\label{eq:E3}
        P^{0.0}[E_3^c]\le C\log^{-2}T.
    \end{equation}
    Now we aim to bound the probability of $E_1$. To this end, for $X\sim P^{\bm0}$, let $\tau_0=0$, and define for all $k\ge1$
    $$\tau_k\coloneq\inf\{t>\tau_{k-1}:|X_t-X_{\tau_{k-1}}|\ge\lfloor (t')^{1/2}\log^{-1/4}T\rfloor\},$$
    $$\sigma_k\coloneq\tau_k-\tau_{k-1}.$$
    Apparently $\sigma_k, k\ge1$ are i.i.d.,\ and we have the estimate $E^{\bm0}[e^{-\sigma_1}]\le e^{-Ct'\log^{-1/2}T}$ by Gaussian estimates.
    Thus, we have 
    \begin{equation}\label{eq:vol}
        E^{\bm0}[|B(X[0,t'],(t')^{1/2}\log^{-1/4}T|]\le C(t')^2\log^{-1/2}T,
    \end{equation}
    which can be deduced from the large deviation estimate
    \begin{align*}
        &P^{\bm0}\left[|B(X[0,t'],(t')^{1/2}\log^{-1/4}T|\ge M\cdot(t')^2\log^{-1/2}T\right]\le P^{\bm0}[\tau_{\lceil M\log^{1/2}T\rceil}\le t']\\
        =\;&P^{\bm0}[e^{-t'}\le\exp(-\tau_{\lceil M\log^{1/2}T\rceil})]\le e^{t'}\left(E^{\bm0}[e^{-\sigma_1}]\right)^{M\log^{1/2}T}\le \exp(-(CM-1)t')
    \end{align*}
    for all $M\in\mathbb N^*$.
    By (\ref{eq:vol}) and the LCLT \cite[Theorem 1.2.1]{trove.nla.gov.au/work/15980712},
    \begin{equation}\label{eq:E1}
        P^{0,0}[E_1^c]\le C\cdot (t')^{-2}E^{\bm0}[|B(X[-t',t'],(t')^{1/2}\log^{-1/4}T)|]\le C\log^{-1/2}T.
    \end{equation}
    By (\ref{eq:E2}), (\ref{eq:E3}) and (\ref{eq:E1}), \begin{equation}\label{eq:dintersect}\begin{aligned}
        &P^{0,0}[X^1[-t,T-t]\cap X^2(0,t']=\emptyset,\ X^1[-t,T-t]\cap X^2[t',\infty)\neq\emptyset]\\
        \le\;&P^{0,0}[E_1^c]+P^{0.0}[E_3^c]+E^{0,0}\left[P^{0,0}[X^1[-t,T-t]\cap X^2[t',\infty)\neq\emptyset|X^1[-t',t'],X^2_{t'}]\mathbbm1_{E_1\cap E_2\cap E_3}\right]\\
        \le\;& C\log^{-4}T+C\log^{-1/2}T\cdot P^{0,0}\left[E_1\cap E_2\cap E_3 \right]\le C\log^{-3/2}T,
    \end{aligned}\end{equation}
    where the second inequality holds because on the event $E_1$,
    $$P^{0,0}[(X^1[-t,-t']\cup X^1[t',T-t])\cap X^2[t',\infty)\neq\emptyset|X^1[-t',t'],X^2_{t'}]\le \frac{C\log\log T}{\log T}$$
    according to \cite[Theorem 4.3.3]{trove.nla.gov.au/work/15980712}, and on $E_1\cap E_3$, 
    \begin{align*}
        &P^{0,0}[X^1[-t',t']\cap X^2[t',\infty)\neq\emptyset|X^1[-t',t'],X^2_{t'}]\\
        \le\;&{\rm cap}(X^1[-t',t'])\cdot C(1\vee {\rm d}(X_{t'}^2,X^1[-t',t']))^{-2}
        \le Ct'\log^{-1}T\cdot(t')^{-1}\log^{1/2}T\le C\log^{-1/2}T
    \end{align*}
    by (\ref{eq:hitting}).
By (\ref{eq:desc}) and (\ref{eq:dintersect}),
\begin{equation}\label{eq:dexpectation}\begin{aligned}
    &E^{\bm0}[{\rm cap}(X[0,T],rT)-{\rm cap}(X[0,T])]\\
    \le\;& 4T\log^{-10}T+\sum_{t=\lfloor2T\log^{-10}T\rfloor+1}^{T-\lfloor 2T\log^{-10}T\rfloor-1}E^{\bm0}[P^{X_t}[\widetilde H_{X[0,T]}>T\log^{-10}T]-e_{X[0,T]}(X_t)]\\
    \le\;&4T\log^{-10}T+CT\log^{-3/2}T\le CT\log^{-3/2}T,
\end{aligned}\end{equation}
which implies
\begin{equation}\label{eq:dexpectation'}
    P^{\bm0}[{\rm cap}(X[0,T],rT)-{\rm cap}(X[0,T])>\delta/2\cdot E^{\bm0}[{\rm cap}(X[0,T])]]\le C\delta^{-1}\log^{-1/2}T
\end{equation}
by (\ref{eq:cap_ex}) and the Markov inequality.
It is easy to verify that (\ref{eq:Tcap_ex}) follows from (\ref{eq:cap_ex}) and (\ref{eq:dexpectation}). Moreover, by (\ref{eq:dexpectation'}) and (\ref{eq:cap_con}),
\begin{align*}
    &P^{\bm0}\left[\left|{\rm cap}(X[0,T],rT)-E^{\bm0}[{\rm cap}(X[0,T])]\right|>\delta E^{\bm0}[{\rm cap(X[0,T])}]\right]\\
    \le\;& P^{\bm0}[{\rm cap}(X[0,T],rT)-{\rm cap}(X[0,T])>\delta/2\cdot E^{\bm0}[{\rm cap}(X[0,T])]]\\
    &+P^{\bm0}\left[\left|{\rm cap}(X[0,T])-E^{\bm0}[{\rm cap}(X[0,T])]\right|>\delta E^{\bm0}[{\rm cap(X[0,T])}]\right]\\
    \le\;&C\delta^{-1}\log^{-1/2}T+C\delta^{-2}\log^{-2}T
    \le C\delta^{-1}\log^{-1/2}T,
\end{align*}
proving (\ref{eq:Tcap_con}).

When $d\ge5$, one can prove 
\begin{equation}\label{eq:dexpectation''}
    E^{\bm0}[{\rm cap}(X[0,T],rT)-{\rm cap}(X[0,T])]\le CT^{3/4}
\end{equation}parallel to (\ref{eq:dexpectation}) with similar techniques, and deduce (\ref{eq:Tcap_ex}) and (\ref{eq:Tcap_con}) from this estimate as above.
\end{proof}
    \hspace{\fill}
        
    \textbf{Finitary random interlacements (FRI)}: We now present the definition of the finitary random interlacements (FRI) point process. For any $u>0$ and any $\rho\in\mathcal{M}$, we define a measure by    $$\pi=\sum_{x\in\mathbb{Z}^d}\sum_{l\in\mathbb{N}}\frac{\rho(l)}{\mu_1(\rho)+1}\cdot\mu(x,0,l).$$
    \begin{definition}
        The finitary random interlacements point process with intensity $u$ and length distribution $\rho$, denoted by $\mathcal X^{u,\rho}$, is a Poisson point process (PPP) on the space $W^{[0,\infty)}$ with density measure $u\cdot\pi$, and the corresponding random subgraph of $\mathbb Z^d$ is defined by
        $\mathcal I^{u,\rho}=G(\mathcal X^{u,\rho})$.
    \end{definition}
    This definition generalizes the one commonly used in the literature, and takes into consideration clouds of trajectories with arbitrary length distributions instead of just geometric ones. The factor of $\frac{1}{\mu_1(\rho)+1}$ in the definition of $\pi$ guarantees that the expected total local time at a vertex is exactly $u$ in $\mathcal X^{u,\rho}$. We refer to \cite{cai2021some} for a more comprehensive introduction of the model.
    
    We are concerned about the distribution of trajectories within $\mathcal X^{u,\rho}$ originating from or traversing a certain subset of $\mathbb{Z}^d$. To this end, we further introduce some notations. Let $\mathcal{X}$ be any point measure on $W^{[0,\infty)}$. For all $A\subset\mathbb{Z}^d$, let $\mathcal X[A]$ denote the point process comprising trajectories within $\mathcal X$ intersecting $A$. For $A,B,D\subset\mathbb{Z}^d,D\subset A$, let 
    \begin{equation}\label{eq:defXAB}
        \mathcal X[A;B]=\mathcal X[A]-\mathcal X[B]
    \end{equation}
    and let 
    \begin{equation}\label{eq:defXADB}
        \mathcal X[A,D;B]=\mathcal X[A;B]\cap\mathcal{H}[A;D],
    \end{equation}where $\mathcal H[A;D]$ was defined in (\ref{eq:def_cH}).

    For all $A\subset\mathbb{Z}^d$ finite, we can find the distribution of $\mathcal X^{u,\rho}[A]$ by a technique called ``rerooting'' of the trajectories (c.f. \cite[Lemma 3.1]{duminil2023finite}).
    \begin{lemma}\label{local}
        Let $A\subset\mathbb{Z}^d$ be finite, then 
        $\mathcal X^{u,\rho}[A]$ is a PPP with intensity $u\cdot\pi_A$,
        where $$\pi_A\coloneq\sum_{x\in A,m,l\in\mathbb{N}}\frac{\rho(m+l)}{\mu_1(\rho)+1}P^x[\widetilde H_A>m]\cdot\mu(x,m,l;A).$$
    \end{lemma}
    Given this result, it is easy to compute the expected number of trajectories within $\mathcal X^{u,\rho}$ hitting $A$ at $x$, which equals $ue_A^{(\rho)}(x)$, where
    \begin{equation}
        e_A^{(\rho)}(x)=\sum_{m=0}^\infty\frac{\mu_0^{(m)}(\rho)}{\mu_1(\rho)+1}P^x[\widetilde H_A>m].
    \end{equation}
    Thus, it is natural to define a substitute for the usual capacity of subsets of $\mathbb{Z}^d$
    \begin{equation}\label{eq:rho_cap}
        {\rm cap}^{(\rho)}(A):=\displaystyle\sum_{x\in A}e_A^{(\rho)}(x),
    \end{equation}
    which we call the $\rho$-capacity.
    Note that $e_A^{(\rho)}(x)\geq e_A(x)$ and thus ${\rm cap}^{(\rho)}(A)\geq{\rm cap}(A)$. 
    
    We will need control over the $\rho$-capacity. It is easy to see that for arbitrary $\rho$ and $d\ge4$,
    \begin{equation}\label{eq:rho_bound5}{\rm cap}^{(\rho)}(A)\leq|A|.
    \end{equation} 
    Moreover, letting $N=\lfloor \mu_1(\rho)\log^{-\lambda}\mu_1(\rho)\rfloor$, then we have
    $$e^{(\rho)}_A(x)\leq\sum_{m=0}^N\frac{m_0^{(m)}(\rho)}{\mu_1(\rho)+1}\cdot 1+\sum_{m=N+1}^\infty\frac{m_0^{(m)}(\rho)}{\mu_1(\rho)+1}P^x[\widetilde H_A>N]\leq\frac{N}{\mu_1(\rho)+1}+P^x[\widetilde H_A>N].$$
    Summing over $x\in A$, we get
    \begin{equation}\label{eq:rho_bound4}
        {\rm cap}^{(\rho)}(A)\leq|A|\log^{-10}\mu_1(\rho)+{\rm cap}(A,\mu_1(\rho)),
    \end{equation}
    where ${\rm cap}(A,\mu_1(\rho))$ is defined by (\ref{eq:def_capT}).
    
    \hspace{\fill}

    \textbf{Branching Processes with agents in countable sets}: Let $\mathscr X$ be an arbitrary countable set. For $\mathcal S$ a multi-subset of $\mathscr X$ and all $x\in\mathcal X$, recall that $n_x(\mathcal S)$ is the multiplicity of $x $ in $\mathcal S$. For all $x\in\mathscr X$, let $\nu_x$ be a distribution on the space of finite point measures on $\mathscr X$. Let $\mathcal S$ be a multi-subset of $\mathscr X$. The branching process with agents in $\mathscr X$ with offspring distribution $(\nu_x)_{x\in\mathscr X}$ starting from $\mathcal S$, denoted by $(Y_n)_{n\geq0}$, is then recursively defined as follows. For all $x\in\mathscr X$ and $i,j\in\mathbb N$, let $Y_{x,i,j}\sim\nu_x$ be mutually independent. First, let $Y_0=\sum_{x\in\mathscr X}n_x(\mathcal S)\cdot \delta_x$; then, assuming that $Y_i$ has been sampled, let $$Y_{i+1}=\sum_{x\in\mathscr X}\sum_{0\le j<n_x(Y_i)}Y_{x,i,j}.$$ 
    In words, every agent at $x$ in the $i$-th generation produces offsprings according to the distribution $\nu_x$. For random variables $\xi_x\sim\nu_x,x\in\mathscr X$, we also say that $(Y_n)_{n\geq0}$ has offspring distribution $(\xi_x)_{x\in\mathscr X}$.

\section{Proof of the lower bound of (\ref{new1.4})}
    In this section, we prove the lower bound of (\ref{new1.4}) which will be deduced from the subcriticality of $\mathcal {I}^{u_n^-,\rho_n}$ (recalling the definition of $u_n^-$ from (\ref{eq:upm})). Throughout the section, let $(\rho_n)_{n\in\mathbb{N}}$ be a fixed appropriate family of distributions (recalling Definition \ref{def:appropriate_family}), and fix $\epsilon>0$.

    The idea is to explore the cluster in $\mathcal I^{u_n^-,\rho_n}$ containing $\bm 0$ layer by layer, and prove that the cluster is almost surely finite (i.e., $\mathcal I^{u_n^-,\rho_n}$ is subcritical) by showing that the exploration stops within finitely many steps. To see this, we compare the exploration process with a Galton-Watson tree that dies out almost surely. In \cite{cai2023exact}, a similar process was employed to prove the upper bound of (\ref{eq:previous_bound}). We now carry out more precise analysis in order to prove the subcriticality of $\mathcal I^{u_n^-,\rho_n}$.
    
    We start with the construction of the exploration process $(\mathcal L_k,L_k)_{k\ge0}$ for $\mathcal I^{u,\rho}$ with arbitrary parameters $u$ and $\rho$. Let $L_0=\{\bm0\}$ and $\mathcal{L}_0$ be the zero measure. Assume that we have defined $L_k$ ($k\in\mathbb{N}$), then let $$\mathcal{L}_{k+1}=\mathcal X^{u,\rho}\left[L_k;\bigcup_{i=0}^{k-1}L_i\right]$$ and $L_{k+1}=V(\mathcal{L}_{k+1})$.

    Let $\mathcal{L}=\cup_{i=1}^\infty\mathcal{L}_i$. Note that $\mathcal{L}$ is exactly the collection of trajectories comprising the cluster in $\mathcal I^{u,\rho}$ that includes $0$. Due to the almost sure finiteness of $\mathcal{L}_i$, $i>0$ and $\mathcal X^{u,\rho}[B(0,N)]$, $N>0$, the cluster of $\mathcal I^{u,\rho}$ containing the origin is infinite iff $\mathcal{L}_i$ is nonempty for all $i>0$.
    
    To simplify our computation, we introduce a modified sequence $(\mathcal{L}_i')_{i\geq 1}$ that stochastically dominates $(\mathcal{L}_i)_{i\geq1}$. Let $\mathcal{J}_i$, $i>0$ be independent PPPs on $W^{[0,\infty)}$ with the same distribution as $\mathcal X^{u,\rho}$. We now generate the $i$-th layer of trajectories with $\mathcal{J}_i$, so that the source of randomness for each layer is independent. Precisely, let $L_0'=\{{\bf0}\}$. Assume that $L_k'$ is defined, then let $\mathcal{L}_{k+1}'=\mathcal{J}_{k+1}[L_k']$.
    \begin{lemma}\label{le:domination_lower}
        For all $u>0$ and $\rho\in\mathcal{M}$, there exists a coupling $\mathbb{Q}$ of $(\mathcal{L}_i)_{i>0}$ and $(\mathcal{L}_i')_{i>0}$ such that for all $i$,
        $$\mathbb{Q}[\mathcal{L}_i\subset\mathcal{L}_i']=1.$$
    \end{lemma}
    \begin{proof}
        Note the following fact parallel to Lemma 4.3 of \cite{cai2023exact}.
        \begin{equation}
        \begin{aligned}
            &\mbox{For any }k\in\mathbb N\mbox{, given }L_0,...,L_k\mbox{, then the conditional distribution}\\ &\mbox{of }\mathcal L_{k+1}\mbox{ is the same as that of }\mathcal X^{u,\rho}\left[L_k;\bigcup_{i=0}^{k-1}L_i\right].
        \end{aligned}
        \end{equation}
        Thus, we can construct $(\mathcal L_k)_{k>0}$ by
        $$\mathcal L_{k+1}=\mathcal J_{k+1}[L_k;\bigcup_{i=0}^{k-1}L_i].$$
        Assuming $\mathcal L_i\subset\mathcal L_i'$ for all $0\le i\le k$, we have
        $$\mathcal L_{k+1}\subset\mathcal J_{k+1}[L_k]\subset\mathcal J_{k+1}[L_k']=\mathcal L_{k+1}'.$$
        The conclusion then follows by induction.
    \end{proof}

    By Lemma \ref{le:domination_lower}, it suffices to prove that $(\mathcal L_k')_{k=0}^\infty$ dies out almost surely. To this end, we introduce a quantity
    $$\kappa^{(\rho)}(A):=\sum_{x\in A}\sum_{m=0}^\infty\frac{\mu_1^{(m)}(\rho)}{\mu_2(\rho)}P^x[\widetilde H_A>m],$$
    where $A\subset\mathbb Z^d$ is finite, and show that the expectation $W_k\coloneq\mathbb E[\kappa^{(\rho)}(L_k')]$ decreases to zero exponentially fast, which matches our intuition of a sub-critical Galton-Watson tree. Technically, note that $k^{(\rho)}$ is sub-additive, i.e.
    $\kappa^{(\rho)}(A_1\cup A_2)\le\sum_{i=1,2}\kappa^{(\rho)}(A_i)$ (since $P^x[\widetilde H_{A_1\cup A_2}>m]\le\sum_{i=1,2}P^x[\widetilde H_{A_i}>m]$). Therefore, given $L_k'$ and the hitting points at $L_k'$ of trajectories in $\mathcal{L}_{k+1}'$, the expectation of $k^{(\rho)}(L_{k+1}')$ from above by those of $\kappa^{(\rho)}(\eta)$, $\eta\in\mathcal{L}_{k+1}'$ whose conditional distribution can be obtained from Lemma \ref{local}. Moreover, when we take $(u,\rho)=(u_n^-,\rho_n)$, the conditional expectation of the sum $\sum_{\eta\in\mathcal L_{k+1}'}\kappa^{(\rho_n)}(\eta)$ is approximately $(1-\epsilon)\kappa^{(\rho_n)}(L_{k}')$, implying that $\mathbb E[\kappa^{(\rho_n)}(L_k')]$ decays exponentially in $k$.
    
    Formally, let 
    $$V_k:=\begin{cases}\mathbb{E}[|L'_k|],&d\geq5;\\\mathbb{E}[|L_k'|]\cdot\log^{-1}\mu_1(\rho),&d=4.\end{cases}$$
    \begin{lemma}\label{lemma3.2}
        For every $\delta>0$, the following holds for large $n\in\mathbb{N}$. For $(u,\rho)=(u_n^-,\rho_n)$, we have
        \begin{align}\label{eq:ineq_C}
                W_{k+1}\leq&(1-\epsilon/2)W_k+\delta V_k,\\
                \label{eq:ineq_V}V_{k+1}\leq& C\cdot W_k,
        \end{align}
        for all  $k\in\mathbb{N}$. Here $C<\infty$ is a constant only dependent on $d$.
    \end{lemma}
    We defer the proof of Lemma \ref{lemma3.2} for a few lines, and first finish the proof of the lower bound of (\ref{new1.4}) using it.
    \begin{proof}[Proof of the lower bound of (\ref{new1.4}) assuming Lemma \ref{lemma3.2}.]
        In (\ref{eq:ineq_V}), set $\delta=\epsilon/(4C)$, then for large $n\in\mathbb{N}$ and $(u,\rho)=(u_n^-,\rho_n)$, we have
        $$W_{k+1}\leq(1-\epsilon/2)W_k+\epsilon/4\cdot W_{k-1}$$ for $k>0$. Take upper limits on both sides and we get $\lim_{k\to\infty}W_k$=0. (It can be deduced from (\ref{eq:ineq_C}) that the sequence $(W_k)_{k\in\mathbb{N}}$ is bounded.) Thus, $\lim_{k\to\infty}V_k=0$, indicating the existence of some (random) $k\in\mathbb N$ such that $L_j'=\emptyset$ for all $j>k$. Denote by $\mathcal C $ the cluster in $\mathcal I^{u_n^-,\rho_n}$ containing $0$, then by Lemma \ref{le:domination_lower}, 
        $$\mathcal C\subset\bigcup_{j=0}^k L_j\subset\bigcup_{j=0}^k L_j'.$$
        Thus, since $L_j$ is almost surely finite for all $j\in\mathbb N$, $\mathcal C$ is almost surely finite, which yields $u_*(\rho_n)\ge u_n^-$.
    \end{proof}
    Now, we present the proof of Lemma \ref{lemma3.2} which involves estimating the conditional expectation of $|L_{k+1}'|$ and $\kappa^{(\rho)}(L_{k+1}')$ given $\mathcal L'_k$.
    \begin{proof}[Proof of Lemma \ref{lemma3.2}]
        The proofs for cases $d\geq5$ and $d=4$ are identical except for the details related to the normalization of expected volume. Here we only present the proof of the latter. By Lemma \ref{local},
        \begin{equation} 
            \begin{aligned}
                &\mathbb{E}[|L'_{k+1}|\big|L'_k]\leq\displaystyle\sum_{x\in L'_k}\displaystyle\sum_{m,l\in\mathbb{N}}\frac{u\rho(m+l)}{\mu_1(\rho)+1}\cdot P_x[\widetilde{H}_{L'_k}>m]\cdot(m+l)\\
                =\;&\frac{u}{\mu_1(\rho)+1}\displaystyle\sum_{x\in L'_k}\displaystyle\sum_{m\in\mathbb{N}}\mu_1^{(m)}(\rho)\cdot P_x[\widetilde{H}_{L'_k}>m]
                \leq u\cdot\frac{\mu_2(\rho)}{\mu_1(\rho)}\cdot\kappa^{(\rho)}(L'_k).
            \end{aligned}
        \end{equation}
        Plugging in $(u,\rho)=(u_n^-,\rho_n)$ and taking the expectation, we get $V_{k+1}\leq \frac{8}{\pi^2}\cdot C_k$, which implies (\ref{eq:ineq_V}). Now we compute the conditional expectation of $\kappa^{(\rho)}(L_{k+1}')$ similarly:
        \begin{equation}\label{eq:ineq_kappa}
            \begin{aligned}
                \mathbb{E}[\kappa^{(\rho)}(L'_{k+1})|L'_k]\leq\;&\displaystyle\sum_{x\in L'_k}\displaystyle\sum_{m,l\in\mathbb{N}}\frac{u\rho(m+l)}{\mu_1(\rho)+1}\cdot P^x[\widetilde H_{L_k'}>m]\cdot\mu(x,m,l;L'_k)[\kappa^{(\rho)}(\cdot)]\\
                \leq\;&\displaystyle\sum_{x\in L'_k}\displaystyle\sum_{m,l\in\mathbb{N}}\frac{u\rho(m+l)}{\mu_1(\rho)+1}\cdot P_x[\widetilde{H}_{L'_k}>m]\cdot(1+\epsilon/2)\frac{\pi^2}{8}\cdot(m+l)\log^{-1}\mu_1(\rho)\\
                &+\displaystyle\sum_{x\in L'_k}\displaystyle\sum_{m,l\in\mathbb{N}}\frac{u\rho(m+l)}{\mu_1(\rho)+1}\cdot \text{err}^{(\rho)}(m+l)\\
                \leq\;&(1+\epsilon/2)\frac{\pi^2}{8}\cdot\frac{u\cdot \mu_2(\rho)\cdot\kappa^{(\rho)}(L'_k)}{\mu_1(\rho)\log \mu_1(\rho)}\\
                &+|\mathcal{L}_k|\cdot\frac{u}{\mu_1(\rho)+1}\displaystyle\sum_{m=0}^\infty(m+1)\rho(m)\cdot \text{err}^{(\rho)}(m),
            \end{aligned}
        \end{equation}
        where $(u,\rho)=(u_n^-,\rho_n)$, and
        $$\text{err}^{(\rho)}(m):=E^{\bm0}\left[\kappa^{(\rho)}(X[0,m])\mathbf 1_{\{\kappa^{(\rho)}(X[0,m])>(1+\epsilon/2)\frac{\pi^2}{8}m\log^{-1}\mu_1(\rho)\}}\right].$$
        Here, in the second inequality, we have used
        $$\begin{aligned}
            {\rm err}^{(\rho)}(m+l)\ge\;&E^{\bm0}[\kappa^{(\rho)}(X[-m,l])\mathbf 1_{\kappa^{(\rho)}(X[-m,l])>M}\mathbf1_{X[-m,0)\cap L_k'=\emptyset}]\\
            \ge\;& P^x[\widetilde H_{L_k'}>m]\mu(x,m,l;L_k')[\kappa^{(\rho)}(\cdot)\mathbf1_{\kappa^{(\rho)}(\cdot)>M}],
        \end{aligned}$$
        where $M=(1+\epsilon/2)\frac{\pi^2}8\cdot(m+l)\log^{-1}\mu_1(\rho)$.
        Taking $(u,\rho)=(u_n^-,\rho_n)$, we have
        \begin{equation}\label{eq:un-}             (1+\epsilon/2)\varepsilon_d\cdot\frac{u_n^-\mu_2(\rho_n)}{\mu_1(\rho_n)}<1-\epsilon/2.
        \end{equation}
        Moreover, we claim that
        \begin{equation}\label{eq:err}
            \lim_{n\to\infty}\frac{\log \mu_1(\rho_n)}{\mu_2(\rho_n)}\sum_{m=1}^\infty m\rho_n(m)\cdot\text{err}^{(\rho_n)}(m)=0.
        \end{equation}
        Inequality (\ref{eq:ineq_C}) then follows by combining (\ref{eq:ineq_kappa}) with (\ref{eq:un-}) and (\ref{eq:err}).
        
        We now verify the claim (\ref{eq:err}). To control $\text{err}^{(\rho_n)}(m)$, we compare $\kappa^{(\rho_n)}(X[0,m])$ with ${\rm cap}(X[0,m],\mu_1(\rho_n))$ via (\ref{eq:rho_bound4}), and use the concentration result on ${\rm cap}(X[0,m],\mu_1(\rho_n))$.
        
        Fix $\delta>0$, then by the definition of the appropriate family, there exists some $C'(\delta)$ such that for all $n$, 
        $$\frac{1}{\mu_2(\rho_n)}\sum_{m>C'\mu_1(\rho_n)}m^2\rho_n(m)<\delta.$$
        Similar to (\ref{eq:rho_bound4}),
        \begin{equation}\label{eq:kappa_bound}\begin{aligned}\kappa^{(\rho_n)}(X[0,m])\leq\;& {\rm cap}(X[0,m],\mu_1(\rho_n))+m\cdot\frac{\mu_1(\rho_n)}{\mu_2(\rho_n)}\mu_1(\rho_n)\log^{-10}\mu_1(\rho_n)\\\leq\;& {\rm cap}(X[0,m],\mu_1(\rho_n))+m\log^{-10}\mu_1(\rho_n),\end{aligned}\end{equation}
        so we have
        \begin{equation}\label{eq:err_bound}
            \text{err}^{(\rho_n)}(m)\leq \text{err}^{[\rho_n]}(m)+m\log^{-10} \mu_1(\rho_n),
        \end{equation}
        where $$\text{err}^{[\rho_n]}(m)=E^{\bm0}\left[{\rm cap}(X[0,m],\mu_1(\rho_n))\mathbf1_{\{{\rm cap}(X[0,m],\mu_1(\rho_n))\geq(1+\epsilon/4)\frac{\pi^2}{8}m\log^{-1}\mu_1(\rho_n)\}}\right].$$
        To bound the second term on the right-hand side of (\ref{eq:err_bound}), note that
        $$\frac{\log \mu_1(\rho_n)}{\mu_2(\rho_n)}\sum_{m=1}^\infty m\rho_n(m)\cdot m\log^{-10}\mu_1(\rho_n)=\log^{1-10}\mu_1(\rho_n)\longrightarrow0.$$
        Now we only need to verify that
        \begin{equation}\label{eq:err'}
        \frac{\log \mu_1(\rho_n)}{\mu_2(\rho_n)}\sum_{m=1}^\infty m\rho_n(m)\cdot {\rm err}^{[\rho_n]}(m)\longrightarrow0.
        \end{equation}
        By Proposition (\ref{eq:Tcap_ex}), for $m<\delta \mu_1(\rho_n)$,
        \begin{equation}\label{eq:small_m}
            {\rm err}^{[\rho_n]}(m)\leq E^{\bm0}[{\rm cap}(X[0,m],m)]<Cm\log^{-1}m;
        \end{equation} while for $m>C'\mu_1(\rho_n)$,
        \begin{equation}\label{eq:large_m}\begin{aligned}
            {\rm err}^{[\rho_n]}(m)\leq\; &E^{\bm0}[{\rm cap}(X[0,m],\mu_1(\rho_n))]\leq\lceil m/\mu_1(\rho_n) \rceil\cdot E^{\bm0}[{\rm cap}(X[0,\mu_1(\rho_n)],\mu_1(\rho_n))]\\\leq \;&Cm\log^{-1}\mu_1(\rho_n);
        \end{aligned}\end{equation}
        while for $\delta \mu_1(\rho_n)\leq m\leq C'\mu_1(\rho_n)$, we claim that, for large $n$,
        \begin{equation}\label{eq:medium_m}
            {\rm err}^{[\rho_n]}(m)\leq\delta m\log^{-1}\mu_1(\rho_n).
        \end{equation}
        We defer the proof of (\ref{eq:medium_m}) for a few lines and finish the proof of (\ref{eq:err'}). By (\ref{eq:small_m}), (\ref{eq:large_m}) and (\ref{eq:medium_m}), we have
        \begin{align*}&\sum_{m=1}^\infty m\rho_n(m)\cdot {\rm err}^{[\rho_n]}(m)\\\leq\;&\sum_{m=1}^{\lfloor\delta \mu_1(\rho_n)\rfloor}m^2\rho_n(m)[(\log^{-1}m)\wedge1]+\delta\sum_{m=\lfloor\delta \mu_1(\rho_n)\rfloor+1}^{\lfloor C'\mu_1(\rho_n)\rfloor}m^2\rho_n(m)\log^{-1}\mu_1(\rho_n)\\&+C\sum_{m=\lfloor C' \mu_1(\rho_n)\rfloor+1}^{\infty}m^2\rho_n(m)\log^{-1}\mu_1(\rho_n)\\
        \leq\;&2(\delta \mu_1(\rho_n))^2\log^{-1}\mu_1(\rho_n)+\delta \mu_2(\rho_n)\log^{-1}\mu_1(\rho_n)+C\delta \mu_2(\rho_n)\log^{-1}\mu_1(\rho_n)\\
        \leq\;&2C\delta \mu_2(\rho_n)\log^{-1}\mu_1(\rho_n),
        \end{align*}
        which implies (\ref{eq:err'}) since $\delta$ is arbitrary.
        
        We now prove (\ref{eq:medium_m}). For all $m\in[\delta \mu_1(\rho_n),C'\mu_1(\rho_n)]$, write $\xi={\rm cap}(X[0,m],\mu_1(\rho_n))$, $\xi'={\rm cap}(X[0,m])$, $a=E^{\bm0}[\xi]$, $a'=E^{\bm0}[\xi']$ and $\widetilde a=\frac{\pi^2}{8}m\log^{-1}\mu_1(\rho_n)$. By Proposition \ref{trace_cap_origin_new} and the Markov inequality,
        \begin{align*}
            E^{\bm0}[\xi';\xi'>(1+\epsilon/8)\widetilde a]\le\;&a' P^{\bm0}[\xi'-a'>(1+\epsilon/8)\widetilde a-a']+E^{\bm0}[\xi'-a';\xi'-a'>(1+\epsilon/8)\widetilde a-a']\\
            \leq\;& (1+\epsilon/16)\widetilde aP^{\bm0}[\xi'- a'>\epsilon \widetilde a/16]+E^{\bm0}[(\xi'-a')\mathbf1_{\{\xi'-a'>\epsilon\widetilde a/16\}}]\\\leq\;&(1+\epsilon/16)\widetilde a(16/\epsilon\widetilde a)^2\cdot {\rm var}(\xi')+16/\epsilon\widetilde a\cdot {\rm var}(\xi')\leq C(\widetilde a)^{-1}{\rm var}(\xi)\\
            \leq\;& C\widetilde a\log^{-2}\mu_1(\rho_n),
        \end{align*}
        where in the second inequality we have used $a'<(1+\epsilon/16)\widetilde a$ for all large $n$, which is deduced from (\ref{eq:cap_ex}).
        Moreover, by (\ref{eq:dexpectation}) and (\ref{eq:dexpectation''}), 
        $$P^{\bm0}[\xi-\xi'>\epsilon a/8]\le\frac8{\epsilon a}E^{\bm0}[\xi-\xi']<C\log^{-1/2}\mu_1(\rho_n).$$ Thus, we have
        $$\begin{aligned}
            {\rm err}^{[\rho_n]}(m)=&\;E^{\bm0}[\xi\mathbbm1_{\xi\ge(1+\epsilon/4)\widetilde a}]\le E^{\bm0}[\xi'\mathbbm1_{\xi\ge(1+\epsilon/4)\widetilde a}]+E^{\bm0}[\xi-\xi']\\
            \le&\;E^{\bm0}[\xi'\mathbbm1_{\xi'\le(1+\epsilon/8)\widetilde a}\mathbbm1_{\xi\ge(1+\epsilon/4)\widetilde a}]+E^{\bm0}[\xi'\mathbbm1_{\xi'>(1+\epsilon/8)\widetilde a}]+E^{\bm0}[\xi-\xi']\\
            \le&\;(1+\epsilon/8)\widetilde a\cdot P^{\bm0}[\xi-\xi'\ge\epsilon \widetilde a/8]+C\widetilde a\log^{-2}\mu_1(\rho_n)+C\widetilde a\log^{-1/2}\mu_1(\rho_n)\\
            \le&\;C\widetilde a\log^{-1/2}\mu_1(\rho_n).
        \end{aligned}$$
        (\ref{eq:medium_m}) then follows by taking $n\in\mathbb N$ large.
    \end{proof}

    \begin{remark}
        Recall the definition of the $\rho$-capacity for all $\rho\in\mathcal M$. It is then not hard to see that for all finite $A\subset\mathbb Z^d$,
        $$\kappa^{(\rho)}(A)={\rm cap}^{(\hat\rho)}(A),$$
        where $\hat\rho=\sum_{k=0}^\infty\frac{k\rho(k)}{\mu_1(\rho)}\delta_k$.
    \end{remark}
    \begin{remark}
        If we only deal with the lower bound for the original finitary random interlacement with geometric killing, the argument can be greatly simplified by the following observation on random walks conditioned on not returning to a given set.
        \begin{observation}For any $A\subset \mathbb{Z}^d$, $x\in A$ and $T\in\mathbb{N}$, let $X^1$ be a simple random walk from $x$ and $X^2$ be a simple random walk from $x$ conditioned not to return to $A$. Let $N^i\sim\text{Geo}(\frac{1}{T+1})$ be independent of $X^i$, $i=1,2$. Then there is a coupling between $(X_t^1)_{t=0}^{N^1}$ and $(X_t^2)_{t=0}^{N^2}$ such that 
            \begin{enumerate}
		      \item almost surely, $N^1\ge N^2$; 
                \item almost surely, there exists a random point $y_{\dagger}\in \mathbb{Z}^d$ such that
		      \begin{equation}
                    X_{N^1-N^2+j}^{1}+y_{\dagger}= X_{j}^{2}, \ \ \forall 0\le j\le N^2. 
		      \end{equation}
            \end{enumerate}
            As a direct consequence, the capacity of $(X_n^{1})_{n=0}^{N^1}$ stochastically dominates the capacity of $(X_n^{2})_{n=0}^{N^2}$.
        \end{observation}
    In fact, we can recursively construct a sequence of trajectories from $X^1$ as follows: let $\tau_0$; for $k\ge1$, let $\bar X^k=\phi^{\tau_{k-1}}X^1-X_{\tau_{k-1}}^1$, and let $\tau_k=(\tau_{k-1}+\widetilde H_A(\bar X^k))\wedge N^1$, where $\phi^k X^1$ is defined by
    $$\phi^k X_j^1=X^1_{j+k},\ \forall j\in\mathbb N$$
    and $\widetilde H_A(\bar X^k)$ is the first time $\bar X^k$ returns to $A$. Define $k^*\coloneq\inf\{k\ge0:\tau_k=N^1\}$. By the memoryless property of the geometric distribution, $(\bar X^{k^*}_t)_{0\le t\le\tau_{k^*}-\tau_{k^*-1}}$ is distributed as $(X^2_t)_{0\le t\le N^2}$, which yields the observation.
            
    Note that by Lemma \ref{local}, the number of trajectories in $\mathcal X^{u,{\rm Geo}(\frac 1{T+1})}$ that hits $A$ at $x$ follows ${\rm Poi}(ue_A^{(T)}(x))$, where $e_A^{(T)}(x)=\sum_{i=0}^\infty\frac{1}{T+1}(\frac{T}{T+1})^iP^x[\widetilde{H}_A>i]$, and the conditional law of trajectories given their total number is identical to that of $\zeta_1*\zeta_2$, where $\zeta_1\sim X^1[0,N^1]$ and $\zeta_2\sim\mathbb{P}[X^1[0,N^1]=\cdot\big|X(0,N^1]\cap A=\emptyset]$. By the above observation, the expected capacity of $\zeta_1*\zeta_2$ is at most twice that of $\zeta_1$. When $d\ge5$, it is not hard to check that as $T\to\infty$,
    $$\mathbb{E}[{\rm cap}^{(T)}(X^1[0,N^1])]=(1+o(1))\varepsilon_dT.$$
    Therefore,
    $$\mathbb{E}[{\rm cap}^{(T)}(V(\mathcal I^{u,T}[A]))]\leq\sum_{x\in A}ue_A^{(T)}(x)\cdot(2+o(1))\varepsilon_dT=(2+o(1))u\varepsilon_dT\cdot{\rm cap}^{(T)}(A).$$
    Taking $u=u_T^-:=\frac{1-\epsilon}{2\varepsilon_d}\cdot\frac{1}{T}$, we get
    $$\mathbb{E}[{\rm cap}^{(T)}(V(\mathcal I^{u,T}[A]))]\leq(1-\epsilon/2){\rm cap}^{(T)}(A)$$ for large $T$, which yields that $\mathcal I^{u_T^-,T}$ is subcritical. The value of $u_T^-$ matches $u_n^-$ defined below (\ref{eq:def_un}). When $d=4$, the same arguments hold by taking $u_T^-=\frac{4(1-\epsilon)}{\pi^2}\cdot T^{-1}\log T$.
    \end{remark}
\section{Proof of the upper bound of (\ref{new1.4})}
In this section, we will prove the upper bound of (\ref{new1.4}) for any appropriate family of distributions $(\rho_n)_{n\in\mathbb N}$.

Recalling $u_n^+$ from (\ref{eq:upm}), we will show that $\mathcal I^{u_n^+,\rho_n}$ is supercritical through an exploration algorithm, as outlined in last but two paragraph of Section 1. To be more specific, we first restrict to a truncated version of $\mathcal X^{u_n^+,\rho_n}$ (denoted by $\overline{\mathcal X}^{u_n^+,\rho_n}$) which consists of ``typical'' trajectories (rigorously defined later). Next, we divide the Euclidean space into boxes of side length $R_n\coloneq\lfloor\sqrt{\mu_1(\rho_n)}\rfloor$, with one-to-one correspondence with vertices of the coarse-grained lattice $\mathbb L_n\coloneq R_n\cdot\mathbb Z^d$, and start the exploration by revealing a cluster in the box corresponding to $\bm 0$ consisted of a fixed number of typical trajectories that satisfy a certain spatial restriction (we call such a collection of trajectories a ``seed''), and then start the exploration algorithm by declaring $\bm 0$ to be active. In each round of the subsequent exploration, we pick an active vertex and try to extend the revealed cluster inside the corresponding box to neighboring ones by revealing a certain number of layers of trajectories, one layer at a time. After that, we determine whether this new round of exploration is good. If so, and the last layer in the aforementioned extension contains a seed in each of the neighboring boxes, we activate the corresponding neighboring vertices and declare the current vertex to be surviving; otherwise, we say this round of exploration at the current vertex fails, and we declare the vertices within a neighborhood of range proportional to $R_n$ to be ruined (including those that has become surviving or active). By bounding from above the conditional probability of failure in each round,  which again involves comparison between the round of exploration process and a Galton--Watson branching process, we show that the vertices surviving in the end dominates a finitely dependent percolation model on the coarse-grained lattice $\mathbb L_n$, and prove its supercriticality with the Liggett--Schonmann--Stacey domination \cite{10.1214/aop/1024404279}.

We implement the aforementioned strategy in the following two subsections. In Subsection 4.1, we introduce some further notation necessary for the introduction of the exploration algorithm (henceforth referred to as the Algorithm) and gather some related preliminary results. In Subsection 4.2, we rigorously describe the Algorithm and prove the supercriticality assuming the estimate on the conditional probability of failure in each step (c.f. Proposition \ref{uniform_low}).

\subsection{Preparation for the Algorithm}
Let $\mathcal U_d:=\{e_1,...,e_d\}$ be the collection of canonical unit vectors on $\mathbb Z^d$, where $e_j$ is the unit vector with $j$-th component $1$, and write $\vec{1}=e_1+e_2+...+e_d$. Let 
\begin{equation}\label{eq:intro_Rn}
    R_n\coloneq\lfloor\sqrt{\mu_1(\rho_n)}\rfloor
\end{equation}
and define the corresponding coarse-grained lattice by $\mathbb L_n\coloneq R_n\cdot\mathbb Z^d$. Let 
\begin{equation}\label{eq:intro_zn}
    z_n=\lfloor R_n/2\rfloor\cdot\vec 1.
\end{equation}
For $x\in \mathbb L_n$, let $B_x^n=x+[0,R_n)^d$ and $\widetilde B_x^n=B(x+z_n,R_n/4)$ be two concentric boxes. Fix a constant
\begin{equation}\label{def_c}
    c=0.01.
\end{equation}
We fix the dimension $d\ge4$, the perturbative constant $\epsilon\in(0,1/2)$ introduced in (\ref{eq:upm}) and the appropriate family $(\rho_n)_{n\in\mathbb N}$, and drop the dependence on them thereinafter. We now introduce some extra symbols and terminology.

\hspace{\fill}

\textbf{Typical trajectories}: To facilitate further construction, we need to ignore certain atypical configurations of random walk trajectories in terms of the duration, capacity, volume and diameter.

Define a mesoscopic scale
$$L_n:=\begin{cases}(R_n)^\frac{2+c}{d-2},&d\geq5;\\R_n\log^{-c}R_n,&d=4\end{cases}$$
and a mildly increasing function with respect to $R_n$
$$I_n:=\begin{cases}(R_n)^c,&d\geq5;\\\log^cR_n,&d=4.\end{cases}$$
For $\eta\in W^{[0,\infty)}$, write $T=T(\eta)$. We introduce four parameters whose values will be determined solely depending on $d$, $\epsilon$ and $(\rho_n)_{n\in\mathbb N}$: parameters
\begin{equation}\label{eq:intro_kK}
    k,\,K\in\mathbb R_+
\end{equation}
put a restriction on the duration of typical trajectories, which will be fixed right after the statement of Lemma \ref{M_apriori}; the parameter
\begin{equation}\label{eq:intro_M}
    M\in\mathbb R_+
\end{equation}
controls the diameter and will be fixed in (\ref{eq:fix_Mtheta1}) and the parameter
\begin{equation}\label{eq:intro_theta2}
    \theta_1=\theta_1(q)\in(0,1),
\end{equation}
which depends on another parameter $q$ formally introduced in Lemma \ref{M_apriori}, controls the fluctuation of the capacity of typical trajectories, and will be fixed in (\ref{eq:fix_betatheta2k1}). We assume in following results in this subsection that these parameters are arbitrarily chosen, unless stated otherwise.

We say $\eta$ is {\it typical} if it satisfies all requirements below formulated by $(\mathcal E_i)_{i=1,2,3,4}$:
\begin{itemize}
    \item $\mathcal E_1=\mathcal E_1(k,K,M)\coloneq\{T\in[k(R_n)^2,K(R_n)^2]\mbox{ and }\eta\subset B(\eta(0), MR_n)\}.$
    \item If $d\geq5$, $$\mathcal E_2=\mathcal E_2(\theta_1)\coloneq\{(1-\theta_1^2)\varepsilon_dT<
    {\rm cap}({\rm range}(\eta))<(1+\theta_1^2)\varepsilon_dT\};$$
    if $d=4$,
    \begin{align*}
        \mathcal E_2=\mathcal E_2(\theta_1)\coloneq\{&{\rm cap}({\rm range}(\eta))>(1-\theta_1^2)\frac{\pi^2}{8}T\log^{-1}\mu_1(\rho_n),\\&{\rm cap}({\rm range}(\eta),R_n^2)<(1+\theta_1^2)\frac{\pi^2}{8}T\log^{-1}\mu_1(\rho_n)\}.
    \end{align*}
    \item $\mathcal E_3\coloneq\{|B(\eta,L_n)|<R_n^2L_n^{d-2}I_n^{1/3}\}.$
    \item  Let 
    \begin{equation}\label{eq:T'}
        T_n'=\begin{cases}\lfloor (R_n)^2\log^{-1}\mu_1(\rho_n)\rfloor,&d=4;\\\lfloor (R_n)^{2-c/4}\rfloor,&d\ge5.\end{cases}
    \end{equation}
    If $d=4$, 
    \begin{align*}
        \mathcal E_4=\mathcal E_4(\theta_1)\coloneq\{&{\rm cap}(\eta[jT_n',(j+1)T_n'\wedge T])\leq (1+\theta_1^2)\frac{\pi^2}{8}T_n'\log^{-1}\mu_1(\rho_n),\\
        &\text{diam}(\eta[jT_n',(j+1)T_n'\wedge T])<\frac{1}{2}T^{1/2}\log^{-c}\mu_1(\rho_n),\ \forall j\in\mathbb{N}, jT_n'\leq T\}.
    \end{align*}
    If $d\geq5$, 
    $$\mathcal E_4=\mathcal E_4(\theta_1)\coloneq\{{\rm cap}(\eta[jT_n',(j+1)T_n'\wedge T])\leq (1+\theta_1^2)\varepsilon_d T_n',\ \forall j\in\mathbb{N},jT_n'\leq T\}.$$
\end{itemize}

Let $\mathscr T_n\coloneq\{\eta\in W^{[0,\infty)}:\eta\mbox{ satisfies }\mathcal E_i,i=1,2,3,4\}$ denote the collection of typical trajectories. The definition of the typicality of a trajectory is dependent on $k,K,M,\theta_1$ and $c$, but we omit them in the notation for simplicity.

As stated in the following lemma, random walk trajectories of properly fixed length are typical with high probability. Recall the law $\mu(0,m,l)$ for any $m,l\in\mathbb N$ from below (\ref{eq:def_muxml}).
\begin{lemma}\label{le:good_trivial}
    There exists a decreasing function $\psi:\mathbb R_+\to(0,1)$ satisfying $\lim_{x\to\infty}\psi(x)=0$, such that
    $$\limsup_{n\to\infty}\sup_{k(R_n)^2\leq m+l\leq K(R_n)^2}\mu(0,m,l)(\mathscr T_n^c)<\psi(M/K^{1/2}).$$
\end{lemma}
\begin{proof} By standard Gaussian estimates, there exists a function $\psi(x)$ tending to $0$ as $x\to\infty$, such that
\begin{equation}\label{eq:cE1}
    \mu(0,m,l)(\mathcal E_1^c)\le \frac12\psi(M/K^{1/2})
\end{equation}
for all $m+l\le KR_n^2$ and all large $n$.
By Propositions \ref{trace_cap_origin_new} and \ref{trace_cap_truncated_new}, for all $kR_n^2\le m+l\le KR_n^2$,
\begin{equation}\label{eq:cE2}
    \mu(0,m,l)(\mathcal E_2^c)<C\log^{-1/2}R_n,
\end{equation}
and
\begin{equation}\label{eq:cE4}
    \mu(0,m,l)(\mathcal E_4^c)<C\log^{-1}R_n.
\end{equation}
Combining (\ref{eq:cE1})-(\ref{eq:cE4}), we get
$$\limsup_{n\to\infty}\sup_{k(R_n)^2\le m+l\le K(R_n)^2}\mu(0,m,l)(\mathcal E_1^c\cup \mathcal E_2^c\cup \mathcal E_4^c)<\psi(M/K^{1/2}).$$

To control the probability of $\mathcal E_3^c$, first note that for all $m+l\le KR_n^2$,
$$\mu(0,m,l)(\mathcal E_3^c)\leq P^{\bm0}\left[|B(X[0,K(R_n)^2],L_n)|\geq R_n^2L_n^{d-2}I_n\right].$$
Then we divide $X[0,K(R_n)^2]$ into pieces of length $(L_n')^2$, where $L_n'=\lfloor L_n/I_n^{1/3}\rfloor$. For all $0\leq j\leq k\coloneq\lceil KR_n^2/(L_n')^2\rceil$, let $\xi_j=X[j(L_n')^2,(j+1)(L_n')^2]$, then we have, for all large $n$,
\begin{align*}
&P^{\bm0}\left[|B(X[0,K(R_n)^2],L_n)|\geq R_n^2L_n^{d-2}I_n\right]\leq P^{\bm0}[\exists\ 0\leq j\leq k, {\rm diam}(\xi_j)> L_n]\\\leq\;& kP^{\bm0}[{\rm diam}(X[0,(L_n')^2])>L_n]\leq CR_n^2/{(L_n')^2}\text{exp}(-I_n^{2/3}).\end{align*}
Here, the last inequality follows from a standard estimate of the transition probability of the simple random walk. Calculating the last term gives
\begin{equation}\label{eq:cE3}
    \mu(0,m,l)(\mathcal E_3^c)\le\begin{cases}
        CI_n^{8/3}\exp(-I_n^{2/3}),&d=4;\\
        \exp(-R_n^{c/2}),&d\ge5\end{cases}
\end{equation}
for all large $n$.
Therefore, we have
$$\limsup_{n\to\infty}\sup_{k(R_n)^2\le m+l\le K(R_n)^2}\mu(0,m,l)(\mathcal E_3^c)=0,$$
concluding the lemma.
\end{proof}

\hspace{\fill}

\textbf{Restriction within trajectories:} When revealing a new layer in the Algorithm, we need to avoid exploring trajectories hitting certain vertices (those with unusually small equilibrium measure, for instance) within trajectories in the last layer explored, which facilitates the characterization of the conditional law of the new layer. To this end, we introduce a further parameter
\begin{equation}\label{eq:intro_theta1}
    \theta_2\in(0,1)
\end{equation}
that will be fixed in (\ref{eq:fix_Mtheta1}).
For any $\eta\in W^{[0,\infty)}$, recall that $e_\eta$ is the equilibrium measure of ${\rm range}(\eta)$, and define the *-proper part of $\eta$ by
\begin{equation}\label{eq:defetahat}
    \widehat{\eta}=\begin{cases}\{\eta(s):s\in[0,T],e_\eta(s)\geq\theta_2\},&d\geq5;\\\{\eta(s):s\in[0,T],e_\eta(s)\geq\theta_2\log^{-1}\mu_1(\rho_n)\\\mbox{ and }P^x[\widetilde{H}_\eta>(R_n)^2\log^{-10}R_n]\leq(1+\theta_1)e_\eta(x)\},&d=4.\end{cases}
\end{equation}
We will need the following property which states that a random walk starting from $x\in\widehat\eta$ conditioned on not returning to ${\rm range}(\eta)$ forgets about the conditioning soon after departure. Recall the symbol ``$\simeq$'' for spatial translation of trajectories from Section 2.
\begin{lemma}\label{lem:forget}
Let $\eta\in \mathscr T_n$, $x\in\widehat\eta$. Let $$l_n=\begin{cases}\lfloor R_n^{1-c/4}\rfloor,&d\geq5;\\\lfloor R_n\log^{-2c}R_n\rfloor,&d=4\end{cases}$$ and
$$t_n=\begin{cases}\lfloor R_n^{2-c}\rfloor,&d\geq5;\\\lfloor R_n^2\log^{-10}R_n\rfloor,&d=4.\end{cases}$$
Then, for all large $n$, there is a coupling $\mathbb{Q}$ of $X^1\sim P^x[X\in\cdot|\widetilde{H}_\eta=\infty]$ and $X^2\sim P^x[X\in\cdot]$ such that 
$$\mathbb{Q}[E]\geq\begin{cases}1-R_n^{-c/2},&d\geq5;
\\1-2\theta_1,&d=4,\end{cases}$$
where
$$E:=\{|X^1_j-X^2_j|\leq l_n,\forall j\geq0\}\cap\{X^1[t_n,\infty)\simeq X^2[t_n,\infty)\}.$$
\end{lemma}
\begin{proof}
    Consider independent random walk paths constructed in the same probability space $(\Omega,\mathscr F,\mathbb Q)$ with different laws: $X^3\sim P^x[X\in\cdot|\widetilde H_\eta>t_n]$, $X^4\sim P^x[X\in\cdot]$, $X^5\sim P^{\bm0}[X\in\cdot]$, $X^6\sim P^x[X\in\cdot|\widetilde H_\eta=\infty]$.
    Let $$X^1[0,\infty)=\begin{cases}X^3[0,t_n]*X^5[0,\infty),&\mbox{if }X^3[0,t_n]*X^5[0,\infty)\mbox{ avoids }\eta\mbox{ after time }0;
    \\X^6[0,\infty), &\mbox{otherwise}\end{cases}$$
    and let $X^2[0,\infty)=X^4[0,t_n]*X^5[0,\infty)$.
    Then $X^1$ and $X^2$ constructed in this fashion has desired marginal distributions. Define two events 
    $$E_1=\{X^1[0,\infty)=X^3[0,t_n]*X^5[0,\infty)\},$$ $$E_2=\{X^3[0,t_n],X^4[0,t_n]\subset B_x(l_n/2)\}.$$ Note that $E\supset E_1\cap E_2$, so it suffices to control the probability of $E_j^c,j=1,2$. For $d=4$, by the fact that $x\in \widehat\eta$ (recall the definition (\ref{eq:defetahat})), 
    $$e_\eta(x)\ge\frac{1}{1+\theta_1}P^x[\widetilde H_\eta>t_n],
    $$
    so it can be deduced easily that
    $$\mathbb Q[E_1^c]=\frac{P^x[t_n<\widetilde H_\eta<\infty]}{P^x[\widetilde H_\eta>t_n]}=1-\frac{e_\eta(x)}{P^x[\widetilde H_\eta>t_n]}<\theta_1;$$
    while for $d\geq5$,
    \begin{align*}&\mathbb Q[E_1^c]\leq \theta_2^{-1}P^x[X[t_n,\infty)\cap{\rm range}(\eta)\neq\emptyset]\\
    \leq\;&\theta_2^{-1}\left(P^x[X_{t_n}\in B(\eta,L_n)]+CL_n^{2-d}{\rm cap}({\rm range}(\eta))\right)\leq \frac12R_n^{-\frac c2}
    \end{align*}for large $n$, where in the first inequality we have used (\ref{eq:cap_ball}) and (\ref{eq:hitting}).
    By estimating the transition probability of the simple random walk and the definition of $\hat\eta$,
     $$\mathbb{Q}[X^4[0,t_n]\not\subset B(x,l_n/2)]=P^{\bm0}[X[0,t_n]\not\subset B(0,l_n)]\leq C\exp(-l_n/t_n^{1/2}),$$
    $$\mathbb{Q}[X^3[0,t_n]\not\subset B(x,l_n/2)]=P^x[X[0,t_n]\not\subset B(x,l_n)|\widetilde{H}_\eta>t_n]\leq C\theta_2^{-1}\log R_n\exp(-l_n/t_n^{1/2}),$$
    where the second inequality uses the fact that $e_\eta(x)>\theta_2 \log^{-1}(R_n)^2$. As a result,
    $$\mathbb Q[E_2^c]\le \frac12R_n^{-c/2}$$ 
    for all large $n$. So far, the lemma has been proved.
\end{proof}

For any fixed $\eta\in W^{[0,\infty)}$, and any $x\in{\rm range}(\eta)$, we define a probability measure $\overline\mu_{x,\eta}^*$ on $W^{[0,\infty)}$ which characterizes the law of the trajectories in an FRI point process that hits $\eta$ at $x$ by
\begin{equation}\label{eq:overline_mu_*}
    \overline{\mu}_{x,\eta}^*:=C\cdot\sum_{k(R_n)^2\leq m+l\leq K(R_n)^2}\rho_n(m+l)\cdot\mu(x,m,l;{\rm range}(\eta)),
\end{equation}
where $C=\sum_{kR_n^2\le m\le KR_n^2}(m+1)\rho_n(m)$ is a normalizing constant. Then, we consider the probability measure obtained by conditioning on $\mathscr T_n$:
\begin{equation}\label{eq:overline_mu}
    \overline{\mu}_{x,\eta}:=\overline{\mu}_{x,\eta}^*(\cdot|\mathscr T_n)
\end{equation}
Similarly, for all $x\in\mathbb Z^d$, we define a simpler probability measure that resemble $\overline\mu_{x,\eta}$ by
\begin{equation}\label{eq:mu}
    \mu_x:=C\cdot\sum_{k(R_n)^2\leq m+l\leq K(R_n)^2}\rho_n(m+l)\mu(x,m,l),
\end{equation}
where $C$ is again a normalizing constant.

Recall that we will compare a round of exploration with of Galton-Watson branching process with agents in the space of trajectories $W^{[0,\infty)}$. To guarantee the supercriticality of that branching process (namely the process $\overline Y$ defined below the proof of Proposition \ref{produce_seed}), we introduce the following lemma which focuses on the expected capacity of $\zeta\sim\overline{\mu}_{x,\eta}$ and can be derived an analogue of Lemma \ref{le:good_trivial} (c.f. Lemma \ref{le:good_trivial2}). Recall the parameters $k,K,M$ and $\theta_1$ introduced for the definition of typical trajectories from (\ref{eq:intro_kK})-(\ref{eq:intro_theta2}), and $\theta_2$ introduced for the definition of *-proper part from (\ref{eq:intro_theta1}).
\begin{lemma}\label{M_apriori}
    We can properly fix $k$ and $K$ such that the following holds. There exists $C_1$ and $C_2=C_2(K)\in(0,\infty)$ such that for all $M>C_2$ and $\theta_1,\theta_2<C_1$, the following holds for large $n$: for all $\eta\in W^{[0,\infty)}$ and $x\in\widehat\eta$, 
    \begin{equation}\label{eq:mu(Tn)_lower}
        \mu_x(\mathscr T_n)\wedge\overline{\mu}_{x,\eta}^*(\mathscr T_n)>1-\epsilon/8
    \end{equation}
    and for a random trajectory $\zeta\sim\overline\mu_{x,\eta}$,
    \begin{equation}\label{eq:equi_lower}
        \mathbb{E}[e_\zeta(\widehat\zeta)]>(1-\epsilon/4)\varepsilon_d\frac{\mu_2(\rho_n)}{\mu_1(\rho_n)}(1+\log\mu_1(\rho_n)\mathbbm1_{d=4})^{-1}.
    \end{equation}
\end{lemma}

The proof is a combination of proved estimates and standard computation of equilibrium measures, so we put it in the appendix. From now on, we fix $k$ and $K$ and assume that $\theta_1,\theta_2<C_1$ and $M>C_2$, so that Lemma \ref{M_apriori} holds.

\hspace{\fill}

\textbf{Good sequence of trajectory sets}: In the Algorithm, we will need a criterion for the goodness of a round of exploration at a certain vertex which is defined through the notion of good sequence of trajectory sets presented below. To make our expositions below more concise, we start with general definitions.

For finite $A,B\subset\mathbb{Z}^d$ and $\rho\in\mathcal{M}$, define
\begin{equation}\label{eq:defvarphi}
    \varphi^{(\rho)}(A,B):=\sum_{x\in A,y\in B}e^{(\rho)}_A(x)e^{(\rho)}_B(y)g(x,y),
\end{equation} 
which measures the likelihood of a trajectory from FRI with length distribution $\rho$ hitting both $A$ and $B$. For $A,D\subset\mathbb Z^d$ and $\eta\in W^{[0,\infty)}$, we define the proper part of $\eta$ with respect to $(A,D)$ by
$$\text{pp}(\eta;A,D)\coloneq\widehat\eta\cap\left(\eta[\tau_A^\eta-R_n^2/I_n,\tau_A^\eta+R_n^2/I_n]\cup B(D,L_n)\right)^c,$$
see Figure \ref{fig:construct_pp}. In the sequel, we will take $A$ to be the vertex set induced by the latest layer we have revealed, and $D$ that induced by all the previous layers.
\begin{figure}
\centering
\includegraphics[width=0.5\linewidth]{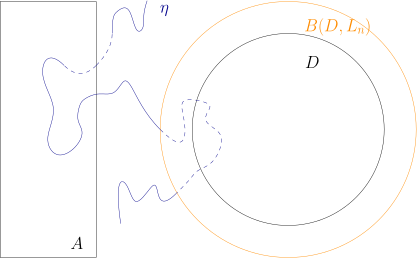}
\caption{\label{fig:construct_pp}This figure illustrates the construction of ${\rm pp}(\eta;A,D)$. The boundaries of sets $A$ and $D$ are shown in black, and that of the neighborhood $B(D,L_n)$ is shown in orange. The blue curve indicates the trajectory $\eta$, where the union of the solid segments represent the proper part ${\rm pp}(\eta;A,D)$.}
\end{figure}

For finite collections of trajectories $\mathcal A_1,...,\mathcal A_t,\mathcal D_0,...,\mathcal D_s\subset W^{[0,\infty)}$, write $\mathcal A_0=\mathcal D_s$. For all $\eta\in\mathcal A_i$($1\leq i\leq t$), we write $${\overline\eta}^i={\rm pp}\left(\eta;V(\mathcal A_{i-1}),V\left(\bigcup_{j=0}^s\mathcal D_j\cup\bigcup_{j=0}^{i-1}\mathcal A_j\right)\right).$$ We introduce a parameter
\begin{equation}\label{eq:intro_k1}
    k_1=k_1(\epsilon,q)\in\mathbb N
\end{equation}
which is dependent on $q$ introduced in Proposition \ref{uniform_low}, whose value will be fixed in (\ref{eq:fix_betatheta2k1}). Let $\widetilde C_n$ be the typical capacity of a random walk running $(R_n)^2$ steps:
\begin{equation}\label{eq:deftildeC}
    \widetilde{C}_n=\begin{cases}R_n^2,&d\geq5;\\R_n^2\log^{-1}R_n^2,&d=4\end{cases}
\end{equation}

\begin{definition}\label{def:good_exploration}
    We say $(\mathcal A_i)_{1\leq i\leq t}$ is {\it good} with respect to $(\mathcal D_i)_{0\leq i\leq s}$ if the following holds for all $i\in\{1,...,t\}$.
    \begin{itemize}
        \item For all $\eta_1,\eta_2\in\mathcal A_i$, ${\rm cap}^{(\rho_n)}(\widehat \eta _j\backslash{\overline\eta_j^i})<4\widetilde C_n/I_n^{1/2}$, $j=1,2$ and $\varphi({\overline\eta_1^i},{\overline\eta_2^i})<\widetilde{C}_n/I_n$.
        \item For all $\eta\in\mathcal A_i$, $\eta$ hits at most one trajectory $\zeta\in\mathcal A_{i-1}$, and $\eta(\tau_{\zeta}^\eta)\in{\overline{\zeta}^{i-1}}$.
        \item The number of trajectories in $\mathcal A_i$ hitting any certain trajectory in $\mathcal A_{i-1}$ is bounded from above by $k_1$.
    \end{itemize}
\end{definition}
In the sequel, we will let $(\mathcal A_i)_{1\le i\le t}$ be the first $t$ layers in the round of exploration at the current vertex, and take $(\mathcal D_i)_{0\le i\le s}$ to be the sequence of all layers explored before the current round.
\begin{remark}
    The above definition guarantees that for $0<t'<t$, $(\mathcal A_i)_{1\le i\le t'}$ is good with respect to $(\mathcal D_i)_{0\le i\le s}$ if $(\mathcal A_i)_{1\le i\le t}$ is.
\end{remark}

\subsection{The Algorithm}
We now introduce the exploration algorithm. Recall the general scheme described at the beginning of this section, where we start each round of exploration with a seed and sampling a certain number of layers. We use 
\begin{equation}\label{eq:intro_alpha}
    \alpha\in\mathbb N
\end{equation}
to denote the total number of layers in each round,
which will be fixed in (\ref{eq:alpha_condition}),
and introduce a parameter $\gamma$ that determines the range of boxes ruined due to an unsuccessful round of exploration:
\begin{equation}\label{eq:intro_gamma}
    \gamma=2\alpha M+2,
\end{equation}
where $M$ defined in (\ref{eq:intro_M}) corresponds to the maximum diameter of a typical trajectory. Let
\begin{equation}\label{eq:intro_beta}
    \beta=\beta(q)
\end{equation}
denote the number of trajectories in a seed, which will be fixed in (\ref{eq:fix_betatheta2k1}). Here, $q$ is the same as in (\ref{eq:intro_theta2}) and (\ref{eq:intro_k1}). Heretofore, we have introduced all the parameters necessary for the definition of Algorithm. Among them, note that $k,K,M,\theta_2$ and $\alpha$ solely depend on $d,\epsilon$ and $(\rho_n)_{n\in\mathbb N}$ and are fixed throughout the proof, although the choice of the latter three will not be covered until Subsection 5.2. Hence, in the remainder of the paper, we omit these parameters in statements and proofs if not necessary to avoid redundancy.

Recall that $\mathscr T_n$ is the collection of all typical trajectories, and define the restriction of $\mathcal X^{u_n^+,\rho_n}$ to $\mathscr T_n$ by $$\overline{\mathcal X}^{u_n^+,\rho_n}=\sum_{\eta\in\mathcal X^{u_n^+,\rho_n}\cap\mathscr T_n}\delta_{\eta}.$$ The algorithm reveals a series of trajectory sets (layers): $\mathcal{L}_{-1}$ will be the collection of trajectories in $\overline{\mathcal X}^{u_n^+,\rho_n}$ hitting a fixed point in the box associated with $\bm0$, which initiates our exploration; $\mathcal{L}_{x,i}$ ($x\in \mathbb L_n,1\leq i\leq\alpha$) denotes the $i$-th layer of trajectories explored in the round of exploration at $x$. We write $\mathcal L_x\coloneq\bigcup_{i=1}^\alpha\mathcal L_{x,i}$ if $x$ is ever explored. A sequence $\aleph_m=(\aleph_m^1,\aleph_m^2)\in(\mathbb L_n\times\{0,1\})\cup\{\Delta\},m=0,1,...$ ($\Delta$ is the cemetery state) will be used to record the result of the exploration: $\aleph_m^1$ is for the position of the $m$-th round of exploration and $\aleph_m^2$ is the indicator of failure in exploration. The state at each element of $\mathbb L_n$ will also be updated along with the exploration of $\overline{\mathcal I}^{u_n^+,\rho_n}$. All possible states of vertices are ``unexplored'',``active'', ``surviving'' and ``ruined''.  Before the implementation of the algorithm, we set all the vertices as unexplored.

Later in the description of the algorithm, we say the $m$-th round of exploration/the round of exploration at $x=\aleph_m^1$ is good until layer $i$($1\leq i\leq\alpha$) if $(\mathcal L_{x,j})_{1\leq j\leq i}$ is good with respect to $(\mathcal L_{-1},\mathcal L_{\aleph_0^1,1},...,\mathcal L_{\aleph_{m-1}^1,\alpha})$ (c.f. Definition \ref{def:good_exploration}), and write ``good'' as a shorthand for ``good until layer $\alpha$''. Here, we adopt the following notation for proper parts of trajectories: $${\overline\eta^{(j,x)}}\coloneq{\rm pp}\left(\eta;V(\mathcal L_{x,j-1}),V(\mathcal L_{-1}\cup\bigcup_{l=0}^{m-1}\mathcal L_{\aleph _l^1}\cup\bigcup_{l=1}^{j-1}\mathcal L_{x,l})\right)$$ for $\eta\in\mathcal L_{x,j}$, $x=\aleph_m^1$ and 
$${\overline\eta^{-1}}\coloneq{\rm pp}(\eta;\{z_n\},\emptyset)$$ for $\eta\in\mathcal L_{-1}$  (recall from the beginning of the section that $z_n$ is approximately the center of the box $B_x^n$). To make our exposition more concise, we will also use the notation $\overline\eta$ when no confusion arises. ${\rm range}(\eta)\backslash\overline\eta$ will be referred to as the improper part of $\eta$. 

We now introduce the notion of ``seed''.
\begin{definition}
    A seed $\mathcal{S}$ at $x\in \mathbb L_n$ is a subset of $\mathscr T_n$ which consists of $\beta$ typical trajectories staying inside the box $B_x^n$ and satisfy: for all $\eta_1,\eta_2\in\mathcal{S}$, $\varphi^{(\rho_n)}({\overline{\eta}_1},{\overline{\eta}_2})<\widetilde C_n/I_n$.
\end{definition}

Note that $W^{[0,\infty)}$ is countable. We choose an arbitrary ordering on $W^{[0,\infty)}$, one on the set $\mathbb L_n$ and one on the collection of all finite subsets $W^{[0,\infty)}$. The ordinals ``first'' and ``next'' in the following description of the algorithm all refer to these prescribed orders.

The algorithm is now described as follows.
\begin{itemize}
    \item Step 1. If there exists a seed at $\bm0$ that is a subset of $\mathcal{L}_{-1}:=\overline{\mathcal X}^{u_n^+,\rho_n}[\{\lfloor R_n/2\rfloor\cdot{\rm \vec 1}\}]$, set the origin as active, and denote the seed by $\mathcal{S}_{\bm0}$ (choose the first seed if there are multiple options). Meanwhile, let $\mathcal{J}_{{\bm0},0}=\{\lfloor R_n/2\rfloor\cdot\vec 1\}$, $J_{{\bm0},0}=\emptyset$, $x={\bm0}$ and $m=0$. If there is no seed, end the algorithm and set $\aleph_0=\Delta$.
    \item Step 2. The $m$-th round of exploration (the round of exploration at $x$):
        \begin{itemize}
            \item Substep 2A. Set $\aleph_m^1=x$. Let $\mathcal{L}_{x,0}=\mathcal{S}_x$. Let $i=1$.
            \item Substep 2B. Let $L_{x,i-1}:=V(\mathcal{L}_{x,i-1})$, $L_{x,i-1}':=\bigcup_{\eta\in\mathcal{L}_{x,i-1}}\widehat{\eta}$, $\mathcal{J}_{x,i}:=\mathcal{J}_{x,i-1}\cup \mathcal{L}_{x,i-1}$, $J_{x,i}:=J_{x,i-1}\cup L_{x,i-1}$ and $\mathcal{L}_{x,i}:=\overline{\mathcal X}^{u_n^+,\rho_n}[L_{x,i-1};J_{x,i-1};L_{x,i-1}']\backslash\mathcal{L}_{x,i-1}$. (Recall the definition of $\overline{\mathcal X}^{u_n^+,\rho_n}[\cdot]$ from (\ref{eq:defXADB}).)
            \item Substep 2C. If $i<\alpha$, increase $i$ by $1$ and go back to Substep 2B. Otherwise, proceed to Step 3.
        \end{itemize}
    \item Step 3. If the following two events
    \begin{equation}\label{eq:def_good}
        \text{Good}(x):=\{\mbox{the exploration at }x\mbox{ is good}\}
    \end{equation}
    and 
    \begin{equation}\label{eq:def_seed}
        \text{Seed}(x):=\{\forall e\in\mathcal U_d\cup(-\mathcal U_d)\mbox,\ \mathcal{L}_{x,\alpha}\mbox{ contains a seed at }x+R_n\cdot e\}
    \end{equation}
    both happen, then we declare $x$ to be surviving, set $y=x+R_n\cdot e$ as active if it is not yet explored or ruined, and let $\mathcal S_y$ be the first seed at $y$ contained in $\mathcal L_{x,\alpha}$, and set $x$ as explored and $\aleph_m^2=0$. Otherwise, we set all the vertices in $B(x,2\gamma\cdot R_n)\cap \mathbb L_n$ as ruined and set $\aleph_m^2=1$.
    \item Step 4. Increase $m$ by $1$. If there is no further active vertex, set $\aleph_m=\Delta$ and end the algorithm; otherwise, let $x'$ be the first active vertex, let $\mathcal{J}_{x',0}=\mathcal{J}_{x,\alpha}$, $J_{x',0}=J_{x,\alpha}$, and assign the value of $x'$ to $x$, and then go back to Step 2.
\end{itemize}

For $x=\aleph_m^1\in\mathbb L_n$ for some $m\in\mathbb N$, let $\mathcal L_x=\bigcup_{i=1}^\alpha\mathcal L_{x,i}$. Let $\tau=\inf\{m\geq0:\aleph_m=\Delta\}$, $\mathcal J=\mathcal L_{-1}\cup\bigcup_{i=0}^{\tau-1}\mathcal L_{\aleph_i^1}$. By our construction, $G(\mathcal J)$ is always connected. Let 
$$H:=\{\mathcal{L}_{-1}\mbox{ contains a seed at }\bm0\}$$ be the event that the exploration at $\bm0$ ever starts, whose probability is non-zero. For $m\in\mathbb{N},i\in\{0,1,...,\alpha-1\}$, define $$\mathcal{F}_{m,i}:=\sigma(\mathcal{L}_{-1},(\aleph_j^1)_{0\leq j\leq m},(\mathcal{L}_{\aleph_j^1,s})_{0\leq j<m,1\leq s\leq\alpha},(\mathcal{L}_{\aleph_j^1,s})_{1\leq s\leq i}).$$

Note that $\tau,\aleph_m$ and $H$ are all measurable with respect to $\overline{\mathcal X}^{u_n^+,\rho_n}$ and that $\mathcal I^{u_n^+,\rho_n}$ percolates as long as $\tau=\infty$. To this end, we define for all $0\le m<\tau$
$${\rm Fail}_m\coloneq\{\aleph_m^2=1\},$$
and say that the exploration at the vertex $\aleph_m^1$ is successful (resp. unsuccessful) if ${\rm Fail}_m^c$ (resp. ${\rm Fail}_m$) occurs.
We bound the probability of ${\rm Fail}_m$ from above conditionally on $\mathcal F_{m,0}$:
\begin{proposition}\label{uniform_low}
    For all $q>0$, we can choose $\beta,k_1$ and $\theta_1$ introduced in (\ref{eq:intro_beta}), (\ref{eq:intro_k1}) and (\ref{eq:intro_theta2}) respectively such that
    \begin{equation}\label{eq:uniform_low}
        \mathbb{P}[{\rm Fail}_m|\mathcal{F}_{m,0}](\omega)<q.
    \end{equation}
    for all large $n\in\mathbb{N}$, $m\geq 0$ and configuration $\omega$ of $\overline{\mathcal X}^{u_n^+,\rho_n}$ satisfying $\aleph_m(\omega)\neq\Delta$.
\end{proposition}

Note that here we assume that $k,K,M$ introduced for the definition of typical trajectories, $\theta_2$ introduced for the definition of $*$-proper part and the number of layers per round $\alpha$ are all properly fixed. The proof of Proposition \ref{uniform_low} is deferred to Section 5. Now, to obtain the upper bound of (\ref{new1.4}), we introduce an auxiliary finitely dependent percolation model $\omega^{q}$, and prove that its cluster containing $\bm 0$, after proper rescaling, is dominated by the set of (coarse-grained) vertices that is surviving in the end which we denote by $V_a\subset\mathbb L_n$; then, we prove that $\omega^{q}$ is supercritical using the Liggett--Schonmann--Stacey domination \cite{10.1214/aop/1024404279}.

Recall $\gamma$ from (\ref{eq:intro_gamma}). Let $\xi_x\sim{\rm Ber}(q)$, $x\in \mathbb Z^d$ be i.i.d.\ random variables. For all $x\in \mathbb Z^d$, let
\begin{equation}\label{eq:def_omega}
    \omega_x^{q}=\begin{cases}0,\;&\exists y\in B(x,2\gamma),\ \xi_y=1;\\1,\;&\mbox{otherwise}.\end{cases}
\end{equation}
We say that $\omega^{q}$ percolates if there is an infinite self-avoiding nearest-neighbor path on $\mathbb Z^d$, such that for each vertex $x$ on the path, $\omega_x^{q}=1$. To put it another way, for every $x\in \mathbb Z^d$, all vertices in its neighborhood $B(x,2\gamma)$ is ruined simultaneously with probability $q$, and we consider the percolation on the  vertices that are not ruined. 

Let $\mathcal{C}_0^{q}\subset \mathbb L_n$ be the cluster of the origin in $\omega^{q}$, and we have the coupling below.
\begin{proposition}\label{coupling_bernoulli}
    For all $q\in(0,1)$, we can choose parameters $\beta,k_1$ and $\theta_1$ depending on $q$ such that for all large $n\in\mathbb{N}$, there is a coupling $\mathbb{Q}$ of $V_a$ with conditional law on $H$ and $\mathcal{C}_0^{q}$ such that $V_a\supset R_n\cdot\mathcal{C}_0^q$ on the event $\{\tau<\infty\}$.
\end{proposition}
\begin{proof}[Proof assuming Proposition \ref{uniform_low}.]
    We use a sequence of configurations $\sigma^m\in\{{\rm U},{\rm S},{\rm A},{\rm R}\}^{\mathbb L_n}$ to record the change of states of vertices in $\mathbb L_n$: $\sigma_x^m$ stands for the state of $x$ right after the algorithm assigns some value to $\aleph_m^2$, and ${\rm U},{\rm S},{\rm A},{\rm R}$ correspond to unexplored, surviving, active and ruined respectively. Define the eventual state of $x$ by $\sigma_x\coloneq\lim_{m\to\infty}\sigma_x^{m\wedge{(\tau-1)}}$. We write $\mathbb{P}^H=\mathbb{P}[\cdot|H]$. Recall that for $x\in\mathbb L_n$ and $m\in\mathbb N$, $\sigma_x$ is the eventual state of $x$ and $\sigma_x^m$ is the state of $x$ after round $m$.  Since given $(\aleph_j)_{0\le j\le m-1}$, the choice of $\aleph_m^1$ is deterministic (according to the pre-fixed order on $\mathbb L_n$), one can argue by induction that the 0-1 variables $\aleph_j^2,j=0,...,m-1$ determine the location of the $m$-th round of exploration. Therefore, we have $\aleph_m^1,\{\aleph_m=\Delta\}\in\sigma(\aleph_j^2,0\leq j<m)$. 
    Assume that $\aleph_m\neq\Delta$, then by Proposition \ref{uniform_low},
    $$\mathbb{P}^H[\aleph_m^2=1|\aleph_j^2,0\leq j<m]=\mathbb{E}^H[\mathbb{P}[\text{Fail}_m|\mathcal{F}_{m,0}]\big|\aleph_j^2,0\leq j<m]\leq q,$$ where $\mathbb{E}^H=\mathbb{E}[\cdot|H]$ is the corresponding expectation. Thus, we can construct under $\mathbb{P}^H$ random variables $\xi_m,m\in\mathbb{N}$ that are i.i.d.\ with distribution ${\rm Ber}(q)$ (Bernoulli distribution with expectation $q$), such that $\aleph_m^2\leq\xi_m$ a.s. under $\mathbb{P}^H$. Moreover, let $\chi_x\sim {\rm Ber}(q),x\in\mathbb L_n$ be mutually independent and independent of $(\xi_m)_{m\in\mathbb{N}}$ and $(\aleph_m)_{m\in\mathbb{N}}$. Let
    $$\chi_x'=\begin{cases}\xi_m,&\aleph_m^1=x\mbox{ for some }m\in\mathbb{N};\\\chi_x,&\mbox{ otherwise.}\end{cases}$$
    Then we have $\chi_x'$, $x\in \mathbb L_n$ are i.i.d.\ with distribution ${\rm Ber}(q)$. Let $A=\bigcup_{\chi_x'=1}B(x,2\gamma R_n)$, and $\Omega_x=\mathbf{1}_{x\in A^c},x\in \mathbb L_n$, then $\omega:=(\Omega_{R_n\cdot x})_{x\in \mathbb Z^d}$ has the same distribution as $\omega^{q}$. Let $\mathcal{C}$ denote the cluster in $\omega$ containing $\bm0$. 
    
    We claim that for all $x\in\mathcal{C}$, $\sigma_{R_n\cdot x}=1$ ($R_n\cdot x$ is surviving eventually), which leads to our desired conclusion. Otherwise, let $x\in\mathcal{C}$ satisfy $\sigma_{R_n\cdot x}\neq1$. Let $\pi=(\pi(j))_{0\leq j\leq m}$ be a nearest-neighbor path on $\mathbb L_n$ in $\mathcal C $ that connects $\bm0$ to $R_n\cdot x$. For all $y=\pi(j),0\leq j\leq m$, $\chi_z'=0$ for all $z\in B(y,2\gamma R_n)\cap \mathbb L_n$. Thus, by our construction, $\aleph_j^2=\xi_j=0$ for all $j$ satisfying $\aleph_j^1\in B(y,2\gamma R_n)$, indicating that $\sigma_y\neq R$. By the definition of Algorithm, if $y\in\mathbb L_n$ ranks $j$-th according to the deterministic order we have fixed, $|\{m\in\mathbb N:\sigma_y^m=A\}|\leq j$, so $\sigma_y\neq 2$, which indicates the final state of any site cannot be active. By our assumption, one may only have $\sigma_y=U$. However, that would imply the existence of $0\leq j<m$ such that $\sigma_{\pi(j)}=1$ and $\sigma_{\pi(j+1)}=0$, which is impossible because $\pi(j+1)$ would have become active once the status of $\pi(j)$ turned ``surviving''.
\end{proof}

\begin{proof}[Proof of the upper bound of (\ref{new1.4}) assuming Proposition \ref{coupling_bernoulli}.]
    By \cite[Theorem 1.3]{10.1214/aop/1024404279}, we can take $q=q(\gamma)\in(0,1)$ sufficiently close to $0$ such that $\omega^{q}$ stochastically dominates a supercritical Bernoulli site percolation on $\mathbb Z^d$. Hence, $\mathcal{C}_0^{q}$ is infinite with non-zero probability. Take $\beta,k_1$ and $\theta_1$ depending on $q$ and satisfying the conditions of Proposition \ref{coupling_bernoulli}, then for large $n$,
    $$\mathbb{Q}[\tau<\infty]=\mathbb{Q}[\tau<\infty,V_a\mbox{ is finite}]\leq\mathbb{Q}[\mathcal{C}_0^{q}\mbox{ is finite}]<1.$$
    Therefore, for large $n$, $\{\tau=\infty\}$ happens with non-zero probability, so $\mathcal I^{u_n^+,\rho_n}$ is supercritical, which yields $u_*(\rho_n)\le u_n^+$.
\end{proof}

\section{Proof of Proposition \ref{uniform_low}.}\label{sec:5}
In this section, we prove Proposition \ref{uniform_low}. We divide the proof of Proposition \ref{uniform_low} into two parts, deriving lower bounds for two quantities: the conditional probability of the exploration at a vertex being good, and that of a good exploration at a vertex revealing new seeds for its neighbors, which respectively correspond to Propositions \ref{good_exploration} and \ref{produce_seed} below. Our strategy is to ignore the previously explored trajectories when revealing new ones in the Algorithm and view a layer of trajectories in the exploration as a generation of a branching process with agents in the space $W^{[0,\infty)}$, which is achieved through couplings. Then, we analyze the behavior of the branching process using the LCLT.

For all $\mathcal S\subset W^{[0,\infty)}$, let $ P_Y^\mathcal S$ be the law of the branching process with agents in $W^{[0,\infty)}$ with offspring distribution $\left(\overline{\mathcal X}^{u_n^+,\rho_n}[{\rm range}(\eta);\emptyset;\widehat\eta]\right)_{\eta\in W^{[0,\infty)}}$ starting from $\mathcal S$. Recall the definition of ${\rm Good}(x)$ and ${\rm Seed}(x)$ for $x\in\mathbb L_n$ from (\ref{eq:def_good}) and (\ref{eq:def_seed}) respectively.

\begin{proposition}\label{good_exploration}
    Arbitrarily choose $\beta\in\mathbb{N}$ in (\ref{eq:intro_beta}) and $r\in(0,1)$. Then there exists $C_3=C_3(\beta,r)\in\mathbb{N}$ and $C_4=C_4(\beta,k_1,r)>0$ such that for all $k_1>C_3$ and $\theta_1<C_4$,
    \begin{equation}
        \mathbb{P}[{\rm Good}(\aleph_m^1)|\mathcal{F}_{m,0}]>1-r
    \end{equation}
    whenever $\aleph_m^1\neq\Delta$ for large $n$.
\end{proposition}

Adapting part of the proof of the above proposition, we obtain the following coupling between the sequence of layers $(\mathcal L_{x,i})_{1\le i\le\alpha}$ and the branching process $(Y_i)_{1\le i\le\alpha}$.
\begin{proposition}\label{coupling1}
    Arbitrarily choose $\beta\in\mathbb{N}$ in (\ref{eq:intro_beta}) and $r\in(0,1)$. Let $m\in\mathbb{N}$ and assume that $\aleph_m^1=x\in \mathbb L_n$. Then, there exists some $C_4^*=C_4^*(\beta,r)$ such that for all $\theta_1>C_4^*$ and all large $n$, we can find a coupling $\mathbb{Q}$ of $(\mathcal{L}_{x,i})_{1\leq i\leq\alpha}$ with law conditioned on $\mathcal{F}_{m,0}$ and $(Y_i)_{i\geq0}$ with law $P_Y^{\mathcal{S}_x}$ satisfying
    \begin{equation}\label{eq:l=y}\mathbb{Q}[\mathcal L_{x,i}=Y_i,i=1,2,...,\alpha]>1-r.
    \end{equation}
\end{proposition}
\begin{proposition}\label{produce_seed}
    For all $r>0$, there exists $C_5=C_5(r)$ such that the following holds. For all $\beta>C_5$ and $k_1>C_3(\beta,r/2)$ in Proposition \ref{good_exploration}, there is $C_6=C_6(\beta,r)$ such that for $\theta_1<C_6$,
    \begin{equation}
        \mathbb{P}[{\rm Seed}(\aleph_m^1)|\mathcal{F}_{m,0}]>1-r
    \end{equation}
    whenever $\aleph_m^1\neq \Delta$ for large $n$.
\end{proposition}
Clearly, Proposition \ref{uniform_low} follows from combining Propositions \ref{good_exploration} and \ref{produce_seed}.
\begin{proof}[Proof of Proposition \ref{uniform_low} assuming Propositions \ref{good_exploration} and \ref{produce_seed}]
    We fix
    \begin{equation}\label{eq:fix_betatheta2k1}
        \beta>C_5(\epsilon,q/2),\ k_1>C_3(\beta,q/2)\mbox{ and }\theta_1<C_6(\beta,q/2)\wedge C_4(\beta,k_1,q/2).
    \end{equation}
    Thus, by Proposition \ref{good_exploration},
    $$\mathbb{P}[{\rm Good}(\aleph_m^1)|\mathcal{F}_{m,0}]>1-q/2;$$
    by Proposition \ref{produce_seed},
    $$ \mathbb{P}[\text{Seed}(\aleph_m^1)|\mathcal{F}_{m,0}]>1-q/2.$$
    The claim \eqref{eq:uniform_low} then follows by noting that ${\rm Fail}_m={\rm Good}(\aleph_m^1)^c\cup\text{Seed}(\aleph_m^1)^c$. 
\end{proof}

The remaining part of this section is organized as below. In Subsection 5.1, we prove Proposition \ref{good_exploration} by induction, and show that techniques involved in the proof lead to Proposition \ref{coupling1}. In Subsection 5.2, we prove Proposition \ref{produce_seed} by proving a similar result for the branching process $(Y)_{i\ge0}$ using the LCLT and then applying Proposition \ref{coupling1}.

\subsection{On the goodness of exploration} In this subsection, we prove Proposition \ref{good_exploration} by iteration: we define $${\rm Good}(x,i)\coloneq\{\mbox{the exploration at }x\mbox{ is good until layer }i\},$$$i=0,1,...,\alpha$, and compute $\mathbb{P}[{\rm Good}(x,i+1)|{\rm Good}(x,i)]$ for $i\le\alpha-1$. To facilitate the computation, we condition on $\mathcal{F}_{m,i}$ and try to couple the conditional law of $\mathcal{L}_{x,i+1}$ with that of the union of independent point processes with distributions $\overline{\mathcal X}^{u_n^+,\rho_n}[{\rm range}(\eta)]$ for all $\eta\in\mathcal L_{x,i}$. Therefore, we get as a byproduct a coupling between $(\mathcal{L}_{x,i})_{1\leq i\leq\alpha}$ conditioned on $\mathcal{F}_{m,0}$ and a branching process with agents in the space $W^{[0,\infty)}$, which will also be crucial for the proof of Proposition \ref{produce_seed}.

We reduce the proof of Proposition \ref{good_exploration} to the iteration of the following result.
\begin{proposition}\label{recursive_good_exploration}
    Arbitrarily fix $\beta$ in (\ref{eq:intro_beta}), then for all $r>0$, then there exists $C_7(\beta,r)$, such that for all $k_1>C_7$, there is some $C_8=C_8(\beta,k_1,r)$ satisfying for all $\theta_1<C_8$ and all large $n$, on the event ${\rm Good}(\aleph_m^1,i)$,
    $$\mathbb{P}[{\rm Good}(\aleph_m^1,i+1)|\mathcal F_{m,i}]>1-r.$$
\end{proposition}
\begin{proof}[Proof of Proposition \ref{good_exploration} assuming Proposition \ref{recursive_good_exploration}.]
    Take $r'=r'(r)>0$ small so that $(1-r')^\alpha>1-r$, and let $k_1>C_7(\beta,r')$ and $\theta_1< C_8(\beta,k_1,r')$.
    By iterating the inequality in Proposition \ref{recursive_good_exploration}, we have
    $$\mathbb{P}[{\rm Good}(\aleph_m^1)|\mathcal{F}_{m,0}]>(1-r')^\alpha>1-r.$$
    Thus, we conclude the proof by taking $C_3(\beta,r)=C_7(\beta,r')$ and $C_4(\beta,k_1,r)=C_8(\beta,k_1,r')$.
\end{proof}

We gather an ingredient for the proof of Proposition \ref{recursive_good_exploration}. In order to compute the conditional probability of ${\rm Good}(\aleph_m^1,i+1)$, we first simplify the conditional law of the $(i+1)$'th layer. We write $\mathcal L_{\aleph_m^1,i}=\{\eta_1,...,\eta_s\}$. As mentioned at the beginning of this subsection, we construct independent point processes with distributions $\overline{\mathcal X}^{u_n^+,\rho_n}[{\rm range}(\eta_j);\emptyset;\hat\eta_j],j=1,...,s$ and couple $\mathcal{L}_{x,i+1}$ with their union. Formally, let $(\mathcal{J}_{x,i,l})_{x\in \mathbb L_n,1\leq i\leq\alpha,l\in\mathbb{N}}$ be a collection of i.i.d.\ PPPs that are independent of $\mathcal{X}^{u_n^+,\rho_n}$ but share the same distribution, and let $\overline{\mathcal J}_{x,i,l}$ denote the restriction of $\mathcal J_{x,i,l}$ to $\mathscr T_n$. Then, we have the following lemma.
\begin{lemma}\label{recursive_coupling1}
    Arbitrarily fix $\beta,k_1$ and $\theta_1$. Conditioning on $\mathcal{F}_{m,i}$ and with restriction on ${\rm Good}(\aleph_m^1,i)$, write $x=\aleph_m^1$, $\mathcal{L}_{x,i}=\{\eta_0,\eta_1,...,\eta_s\}$, and let $$\mathcal{K}_{x,i+1}=\bigcup_{j=0}^s\overline{\mathcal{J}}_{x,i+1,j}[{\rm range}(\eta_j);\emptyset;\widehat{\eta}_j].$$ Then for all $r>0$, there exists a coupling $\mathbb Q$ between $\mathcal{L}_{x,i+1}$ with law conditioned on $\mathcal{F}_{m,i}$ and $\mathcal{K}_{x,i+1}$ such that
    $$\mathbb Q[\mathcal L_{x,i+1}\neq\mathcal K_{x,i+1}]<r$$
    for large $n$ and all realizations of $\overline{\mathcal X}^{u_n^+,\rho_n}$ satisfying ${\rm Good}(\aleph_m^1,i)$.
\end{lemma}
\begin{proof}
    Recall that for $\eta_j\in\mathcal L_{x,i}$, $\overline{\eta_j}={\overline{\eta_j}^{(i,x)}}$ is the proper part of $\eta_j$. Recall the definition of $\mathcal{J}_{x,i}$ from Algorithm. Since we have arbitrarily fixed an $x$ and $i$, we write $V=V(\mathcal{J}_{x,i})$ throughout this proof.
    Note that conditionally on $\mathcal{F}_{m,i}$, the law of $\mathcal{L}_{x,i+1}$ is identical to that of an independent copy of $\overline{\mathcal X}^{u_n^+,\rho_n}[L_{x,i},V;L_{x,i}']$. Now we sample independent point processes $\mathcal T=\mathcal X^{u_n^+,\rho_n}$, and $$\mathcal{J}_j=_d\overline{\mathcal X}^{u_n^+,\rho_n}[{\rm range}(\eta_j);\emptyset;\widehat{\eta}_j]\backslash\overline{\mathcal X}^{u_n^+,\rho_n}[{\rm range}(\eta_j);V\cup\bigcup_{l\neq j}\overline{\eta}_l;\overline{\eta}_j],\ j=0,1,...,s.$$
    Recall that $\mathcal T(\eta)$ is the mass of $\mathcal T$  assigned to $\eta$. Define the restriction of $\mathcal T$ to $T_n$, $\overline{\mathcal T}$ by
    $$\overline{\mathcal T}\coloneq\sum_{\eta\in T_n}\mathcal T(\eta)\cdot\delta_\eta.$$
    Let $\mathcal{X}=\overline{\mathcal{T}}[L_{x,i};V;L_{x,i}']$ and
    $$\mathcal{X}_j=\overline{\mathcal{T}}[{\rm range}(\eta_j);V\cup\bigcup_{l\neq j}\overline{\eta_l};\overline{\eta}_j]\cup \mathcal{J}_j,j=0,1,...,s.$$
    Let $\mathcal{X}'=\sum_{j=0}^s\mathcal{X}_j$, and note that we have $\mathcal{X}\sim\nu$ and $\mathcal{X}'\sim\nu'$, where $\nu$ is the law of $\mathcal L_{x,i+1}$ given $\mathcal F_{m,i}$, and $\nu'$ is the law of $\mathcal K_{x,i+1}$. Therefore, it suffices to show that $\mathbb{P}[\mathcal{X}=\mathcal{X}']>1-r$ when $n$ is large. 
    
    We define the so called bad part of $\mathcal{X}$ by \begin{equation*}
        \begin{aligned}
            \mathcal{X}_{\text{bad}}:=\{&\eta\in\mathcal{X}:\eta\mbox{ first hits }L_{x,i}\mbox{ in the improper part of some trajectory,}\\&\mbox{or hits both }\overline{\eta}_j\mbox{ and }\overline{\eta}_l\mbox{ for some }j\neq l,0\leq j,l\leq s\}
        \end{aligned}
    \end{equation*}
    In fact, $\mathcal{J}_j,0\leq j\leq s$ and $\mathcal{X}_\text{bad}$ are all empty with probability tending to $1$ when $n$ tends to infinity, which directly leads to $\mathcal{X}=\mathcal{X}'$. Let $L_{x,i}''=\bigcup_{\eta\in\mathcal L_{x,i}}\hat\eta\backslash \overline\eta$. Since the exploration at $x$ is good until the $i$'th layer, $s\leq\beta k_1^i$, by computing the intensity of PPPs $\mathcal{X}_\text{bad}$ and $\mathcal{J}_j,0\leq j\leq s$, we get
    \begin{equation}\label{eq:intensity_control_bad}
        \begin{aligned}
            \mathbb{P}[\mathcal{X}_\text{bad}\neq\emptyset]\le\;&u_n^+{\rm cap}^{(\rho_n)}(L_{x,i}'')+\sum_{0\le j< l\le s}\mathbb P[\exists\zeta\in\mathcal T,\zeta\mbox{ hits both }\overline\eta_j\mbox{ and }\overline\eta_l]\\
            \leq\;&u_n^+{\rm cap}^{(\rho_n)}(L_{x,i}'')+\sum_{0\le j< l\le s}\left(\sum_{x\in\overline\eta_j}u_n^+e_{\overline\eta_j}^{(\rho_n)}(x)h(x,\overline\eta_l)+\sum_{x\in\overline\eta_l}u_n^+e_{\overline\eta_l}^{(\rho_n)}(x)h(x,\overline\eta_j)\right)\\
            \leq\;& u_n^+{\rm cap}^{(\rho_n)}(L_{x,i}'')+\sum_{j\neq l}u_n^+\varphi^{(\rho_n)}(\overline\eta_j,\overline\eta_l),
        \end{aligned}
    \end{equation}
    where we have used (\ref{eq:hit_decompose}) in the last line. Likewise, we can show
    \begin{equation}\label{eq:intensity_control_j}
        \sum_{j=0}^s\mathbb{P}[\mathcal{J}_j\neq\emptyset]\leq u_n^+\sum_{j=0}^s{\rm cap}^{(\rho_n)}(\widehat\eta_j\backslash\overline\eta_j)+\sum_{0\leq j,l\leq s}u_n^+\varphi^{(\rho_n)}(\overline\eta_j,\overline\eta_l)+u_n^+\sum_{j=0}^s\varphi^{(\rho_n)}(\overline\eta_j,V\cap B(x,\gamma R_n)).
    \end{equation}
    In the last term, we have $V\cap B(x,\gamma R_n)$ instead of $V$ for the following reason: by definition, the diameter of a typical trajectory is always bounded by $2MR_n$, so it is impossible for any trajectory in $\mathcal J_j$ to hit $B(x,\gamma R_n)^c$ by our choice of $\gamma$ in (\ref{eq:intro_gamma}). By the first item in Definition \ref{def:good_exploration},
    $${\rm cap}^{(\rho_n)}(L_{x,i}'')\leq\sum_{j=0}^s{\rm cap}^{(\rho_n)}(\widehat\eta_j\backslash\overline\eta_j)\leq(s+1)\cdot 4\widetilde C_n/ I_n^{1/2};$$
    and for all $0\leq j<l\leq s$,
    $$\varphi^{(\rho_n)}(\overline\eta_j,\overline\eta_l)\leq\widetilde C_n/I_n^{1/2}$$
    Moreover, we claim that
    \begin{equation}\label{eq:phi_far}
        \varphi^{(\rho_n)}(\overline\eta_j,V\cap B(x,\gamma R_n))\leq C(\widetilde C_n)^2 L_n^{2-d}.
    \end{equation} 
    Finally, noting that by (\ref{eq:def_un}) and (\ref{eq:upm}),
    \begin{equation}\label{eq:u_n+order}
        u_n^+\lesssim\frac{(1+\log\mu_1(\rho_n)\mathbbm1_{d=4})\mu_1(\rho_n)}{\mu_2(\rho_n)}\le\frac{1+\log\mu_1(\rho_n)\mathbbm1_{d=4}}{\mu_1(\rho_n)}\lesssim(\widetilde C_n)^{-1},
    \end{equation}
    where $\widetilde C_n$ was defined in (\ref{eq:deftildeC}), we have
    \begin{equation}\label{5.3}\begin{aligned}
        \mathbb P[\mathcal X\neq\mathcal X']\le\mathbb{P}[\mathcal{X}_{\text{bad}}\neq\emptyset]+\sum_{j=0}^s\mathbb{P}[\mathcal{J}_j\neq\emptyset]\le C(sI_n^{-1/2}+\widetilde C_nL_n^{2-d})
    \end{aligned}\end{equation}
    which tends to $0$ uniformly in the realization of $\overline{\mathcal X}^{u_n^+,\rho_n}$.
    This concludes the proof of the lemma.
    
   It remains to prove (\ref{eq:phi_far}). Note that the vertices comprising $V\cap B(x,\gamma R_n)$ can only come from trajectories in $\mathcal L_y$ where $y\in B(x,2\gamma R_n)\cap\mathbb L_n$. Since $x=\aleph_m^1$ has not been ruined until the end of the $(m-1)$-th round of exploration, the round of exploration at any $y\in\{\aleph_0^1,...,\aleph_{m-1}^1\}\cap B(x,2\gamma R_n)$ must have succeeded, so $$|\mathcal L_y|\le\beta(1+k_1+...+k_1^{\alpha-1})\le\beta k_1^\alpha.$$ Thus, the vertices comprising $V\cap B(x,\gamma R_n)$ come from at most $$|B(x,2\gamma R_n)\cap\mathbb L_n|\cdot\beta k_1^{\alpha+1}=(4\gamma+1)^d\beta k_1^\alpha$$ typical trajectories, so ${\rm cap}^{(\rho_n)}(V\cap B(x,2\gamma R_n))\leq C\cdot\widetilde C_n$, where $C=C(\beta,\gamma,k_1)$ and $\widetilde C_n$ defined in (\ref{eq:deftildeC}) stands for the order of the capacity of a typical random walk path of length $(R_n)^2$. In addition, ${\rm cap}^{(\rho_n)}(\overline\eta_j)\leq C\cdot\widetilde{C}_n$, and $d(\overline\eta_j,V)>L_n$ by the definition of $\overline\eta_j$, so we have
    \begin{equation}
        \varphi^{(\rho_n)}(\overline\eta_j,V)\leq{\rm cap}^{(\rho_n)}(\overline\eta_j){\rm cap}^{(\rho_n)}(V\cap B(x,2\gamma R_n))\cdot CL_n^{2-d}\leq C\cdot(\widetilde C_n)^2L_n^{2-d}.
    \end{equation}
This finishes the proof.
\end{proof}

Now we are in a good position to prove Proposition \ref{recursive_good_exploration}.
\begin{proof}[Proof of Proposition \ref{recursive_good_exploration}.]
    By Lemma \ref{recursive_coupling1}, it suffices to prove
    \begin{equation}\label{eq:good'}\mathbb{P}[{\rm Good}'(\aleph_m^1,i+1)|\mathcal{F}_{m,i}]>1-r/2,
    \end{equation}
    where $${\rm Good}'(\aleph_m^1,i+1)\coloneq\{(\mathcal L_{\aleph_m^1,1},...,\mathcal{L}_{\aleph_m^1,i},\mathcal{K}_{\aleph_m^1,i+1})\mbox{ is good w.r.t. }\mathcal{P}_m\}.$$ Recall the three restrictive items for goodness of exploration in Definition \ref{def:good_exploration}. We are going to verify them respectively.
    
    First, we deal with the third restriction that for all $\eta\in\mathcal L_{x,i}$, the number of trajectories in $\mathcal L_{x,i+1}$ intersecting $\hat\eta$ is no larger than $k_1$. We need to control the probability that $\overline{\mathcal J}_{x.i+1,j}[{\rm range}(\eta_j);\emptyset;\widehat\eta_j]$, a PPP with intensity at most $u_n^+{\rm cap}^{(\rho_n)}(\eta_j)$, contains more than $k_1$ trajectories. 
    Note that for $N\sim{\rm Poi}(u_n^+K\widetilde C_n)$, there is some $C=C(K)<\infty$, such that for all large $k_1$, 
    \begin{equation}\label{5.4}
        \mathbb{P}[N>k_1]<C2^{-k_1},
    \end{equation}
    because the intensity parameter $u_n^+K \widetilde C_n\lesssim K$ by (\ref{eq:u_n+order}). Henceforth, we always assume that $k_1$ satisfies (\ref{5.4}). Let
    $$E_1=\{\exists1\leq j\leq s\mbox{, s.t. }|\mathcal J_{x.i+1,j}[{\rm range}(\eta_j);\emptyset;\widehat\eta_j]|>k_1\},$$ then by (\ref{5.4}), we can take a large  constant $C_7(\beta,r)$ such that when $k_1>C_7$,
    \begin{equation}\label{eq:e1}
        \mathbb{P}[E_1|\mathcal F_{m,i}]\leq s\cdot C2^{-k_1}\leq\beta k_1^\alpha 2^{-k_1}<r/4.
    \end{equation}
    
    Next, we verify the second restriction that each $\zeta\in\mathcal L_{x,i+1}$ hits $\hat\eta$ for exactly one $\eta\in \mathcal L_{x,i}$ and first hits $\hat\eta$ in $\overline\eta$. We control the intersection events defined below. Define
    $$E_2=\{\exists 1\leq j\leq s,\zeta\in \mathcal J_{x.i+1,j}[{\rm range}(\eta_j);\emptyset;\widehat\eta_j]\mbox{ hits }\eta_j\mbox{ in }\widehat\eta_j\backslash\overline\eta_j\},$$
    $$E_3=\{\exists \zeta\in\mathcal K_{x,i+1}\mbox{ that hits }\overline\eta_j\mbox{ and }\overline\eta_l\mbox{, where }1\leq j<l\leq s\},$$
    Then, by the proof of Lemma \ref{recursive_coupling1}, for large $n$,
    \begin{equation}\label{eq:e234}
        \mathbb{P}[E_2\cup E_3|\mathcal F_{m,i}]<r/4.
    \end{equation}

    Finally, we deal with the first restriction which involves estimates on the $\rho_n$-capacity and $\varphi^{(\rho_n)}$ quantities (defined in (\ref{eq:cap^rho}) and (\ref{eq:defvarphi}) respectively) related to the new generation $\mathcal L_{x,i+1}$. For $\zeta\in\mathcal K_{\aleph_m^1,i+1}$, write $\overline\zeta={\rm pp}(\zeta;L_{m,i},V\cup L_{m,i})$, where $V$ is the same as in the proof of Lemma \ref{recursive_coupling1}. Let
    $$E_4=\{\exists\zeta\in\mathcal K_{\aleph_m^1,i+1}, {\rm cap}^{(\rho_n)}(\widehat\zeta\backslash\overline\zeta)>4\widetilde C_n/I_n^{1/2}\},$$
    $$E_5=\{\exists \zeta_1,\zeta_2\in\mathcal{K}_{\aleph_m^1,i+1}\mbox{, s.t. }\varphi^{(\rho_n)}(\overline\zeta_1,\overline\zeta_2)>\widetilde C_n/I_n\},$$ and we aim to bound the conditional probability of $E_4\cup E_5$ from above. Instead of conditioning on $\mathcal F_{m,i}$ alone, this time we additionally condition on the number, lengths and hitting points of trajectories within $\mathcal K_{\aleph_m^1,i+1}$. More precisely, let $N_j=\bigl|\mathcal J_{x.i+1,j}[{\rm range}(\eta_j);\emptyset;\widehat\eta_j]\bigr|$, $j=1,...,s$, and define
    $$\mathcal{G}=\sigma\{N_j,T(\zeta),t(\zeta,\eta_j),\zeta(t(\zeta,\eta_j)),\forall 1\leq j\leq s,\zeta\in\mathcal J_{x.i+1,j}[{\rm range}(\eta_j);\emptyset;\widehat\eta_j]\}.$$ We claim that there exists $C_8(\beta,k_1,r)$ such that for all $\theta_1<C_8$, on ${\rm Good}(\aleph_m^1,i)\cap E_1^c$,
    \begin{equation}\label{eq:e56}
        \mathbb{P}[E_4\cup E_5|\mathcal F_{m,i},\mathcal G]<r/4.
    \end{equation}
    for all large $n$. The proof of (\ref{eq:e56}) relies on the law conditioned on $\mathcal F_{\aleph_m^1,i}$ and $\mathcal G$ of trajectories within $\mathcal K_{\aleph_m^1,i+1}$ deduced from Lemma \ref{local}, and also on Lemma \ref{lem:forget} that helps deal with the law of backward paths of these trajectories that involves conditioning. We will attend to the details shortly.
    Coming back to the proof, by definition,
    $${\rm Good}(\aleph_m^1,i+1)\supset {\rm Good}(\aleph_{m}^1,i)\backslash\bigcup_{j=1}^5 E_j.$$
    Assuming (\ref{eq:e56}), the conclusion then follows from (\ref{eq:e1}). (\ref{eq:e234}) and (\ref{eq:e56}).
\end{proof}

In order to show (\ref{eq:e56}), we need the following two auxiliary lemmas, whose proofs are technical and deferred to the appendix.
\begin{lemma}\label{le:intersect_cap}
     There exists some $C<\infty$ such that the following holds for all large $n$. For all $A\subset \mathbb Z^d$ finite and nonempty, $y\in\mathbb Z^d$, and positive integer $T<KR_n^2$,
        \begin{align*}
        E^y[{\rm cap}(X[T,KR_n^2]\cap A,R_n^2)]&\leq C|B(A,L_n)|T^{-1}\log^{-1}R_n,&\mbox{if $d=4$};\\
    E^y\left[\left|X[T,KR_n^2]\cap A\right|\right]&\le C|A|T^{1-d/2},&\mbox{if $d\ge 5$}. 
        \end{align*}
    \end{lemma}

\begin{lemma}\label{le:free_varphi}
    There exists some $C<\infty$ such that for all large $n$ and $y,z\in\mathbb Z^d$,
    $$E^{y,z}[\varphi^{(\rho_n)}(X^1[0,K(R_n)^2],X^2[0,K(R_n)^2])]\leq \begin{cases}CR_n^2\log^{-2}R_n,&d=4;\\CR_n,&d=5;\\C\log R_n,&d=6;\\C,&d>6.\end{cases}$$
\end{lemma}

Now we prove (\ref{eq:e56}) with Lemmas \ref{le:intersect_cap} and \ref{le:free_varphi}.
\begin{proof}[Proof of (\ref{eq:e56}).]
    Write $\mathcal J_{x,i+1,j}[{\rm range}(\eta_j);\emptyset;\widehat\eta_j]=\sum_{l=1}^{N_j}\delta_{\zeta_{j,l}}$, $T_{j,l}=T(\zeta_{j,l})$, $t^1_{j,l}=t(\zeta_{j,l},\eta_j)$, $t^2_{j,l}=T_{j,l}-t^1_{j,l}$, $x_{j,l}=\zeta_{j,l}(t^1_{j,l})$. By a union bound, 
    \begin{align}
        \label{eq:e5}\mathbb P[E_4^c|\mathcal F_{m,i},\mathcal G]\leq&\sum_{0\leq j\leq s,1\leq l\leq N_j}\mathbb P[{\rm cap}^{(\rho_n)}(\widehat\zeta_{j,l}\backslash\overline\zeta_{j,l})>4\widetilde C_n/I_n^{1/2}|\mathcal F_{m,i},\mathcal G],\\
        \label{eq:e6}\mathbb P[E_5^c|\mathcal F_{m,i},\mathcal G]\leq&\sum_{(j_1,l_1)\neq(j_2,l_2)}\mathbb P[\varphi^{(\rho_n)}(\overline\zeta_{j_1,l_1},\overline\zeta_{j_2,l_2})>\widetilde C_n/I_n|\mathcal F_{m,i},\mathcal G].
    \end{align}
    (Here, the proper part $\overline\zeta_{j,l}$ for $\zeta_{j,l}$ that is atypical is similarly defined.) Now, it suffices to bound every term on the right-hand side of (\ref{eq:e5}) and (\ref{eq:e6}) from above. By Lemma \ref{local}, conditionally on $\mathcal F_{m,i}$ and $\mathcal G$, with the restriction ${\rm Good}(\aleph_m^1,i)\cap E_1^c$, the distribution of $\zeta_{j,l}$ is $\mu(x_{j,l},t^1_{j,l},t^2_{j,l};{\rm range}(\eta_j))$, for all $0\leq j\leq s, 1\leq l\leq N_j$, and they are mutually independent. Let $\zeta_{j,l}^f=\zeta_{j,l}[t_{j,l}^1+R_n^2/I_n,T_{j,l}]$, $\zeta_{j,l}^b=\zeta_{j,l}[0,t_{j,l}^1-(R_n)^2/I_n]$. Recall the constant $C_8$ from Proposition \ref{recursive_good_exploration}. By Lemma \ref{lem:forget}, we can take $C_8(\beta,k_1,r)$ sufficiently small such that for all $\theta_1<C_8$ and all large $n$, conditionally on $\mathcal F_{m,i}$ and $\mathcal G$, there is a coupling between $\zeta_{j,l}$ and $X^{j,l}$ ($j=0,1,...,s,\ l=1,...,N_j$), where $X^{j,l}\sim P^{x_{j,l}}$ are mutually independent, and satisfy
    $$X^{j,l}[-t_{j,l}^1,-t_n]\simeq\zeta_{j,l}[0,t_{j,l}^1-t_n],X^{j,l}[t_n,t_{j,l}^2]\simeq\zeta_{j,l}[t_{j,l}^1+t_n,T_{j,l}]$$
    and $|X^{j,l}(u-t_{j,l}^1)-\zeta_{j,l}(u)|<l_n$
    for all $0\leq u\leq T_{j,l}$ 
    with probability $1-r/N$, where $N=32(\beta k_1^\alpha)^2$.

    Now we consider the right-hand side of (\ref{eq:e5}). Note that by definition, $$\widehat\zeta_{j,l}\backslash\overline\zeta_{j,l}\subset\zeta_{j,l}[t^1_{j,l}-R_n^2/I_n,t_{j,l}^1+R_n^2/I_n]\cup(\zeta_{j,l}^f\cap V')\cup(\zeta_{j,l}^b\cap V'),$$
    where $V'=B(V(\mathcal J_{x,i}),L_n)\cap B(x,2\gamma R_n)$. Therefore,
    \begin{equation}\label{eq:cap^rho}\begin{aligned}
        &\mathbb P[{\rm cap}^{(\rho_n)}(\widehat\zeta_{j,l}\backslash\overline\zeta_{j,l})>4\widetilde C_n/I_n^{1/2}|\mathcal F_{m,i},\mathcal G]\\
        \leq\;&\mathbb P[{\rm cap}^{(\rho_n)}(\zeta_{j,l}[t_{j,l}^1-R_n^2/I_n,t_{j,l}^1-t_n])+{\rm cap}^{(\rho_n)}(\zeta_{j,l}[t_{j,l}^1+t_n,t_{j,l}^1+R_n^2/I_n])\mid\mathcal F_{m,i},\mathcal G]\\
        &+\mathbb P[{\rm cap}^{(\rho_n)}(\zeta_{j,l}^f\cap V')>\widetilde C_n/I_n^{1/2}|\mathcal F_{m,i},\mathcal G]+\mathbb P[{\rm cap}^{(\rho_n)}(\zeta_{j,l}^b\cap V')>\widetilde C_n/I_n^{1/2}|\mathcal F_{m,i},\mathcal G]\\
        \le\;&P^{x_{j,l}}[{\rm cap}^{(\rho_n)}(X[-(R_n)^2/I_n,-t_n])+{\rm cap}^{(\rho_n)}(X[t_n,(R_n)^2/I_n])>\widetilde C_n/I_n^{1/2}]\\
        &+P^{x_{j,l}}[{\rm cap}^{(\rho_n)}(X[R_n^2/I_n,t_{j,l}^2]\cap B(V',L_n))>\widetilde C_n/I_n^{1/2}]\\&+P^{x_{j,l}}[{\rm cap}^{(\rho_n)}(X[-t_{j,l}^1,-R_n^2/I_n]\cap B(V',L_n))>\widetilde C_n/I_n^{1/2}]+3r/N,
    \end{aligned}\end{equation}
    where the second inequality above follows from the coupling between $\zeta_{j,l}$ and $X^{j,l}$.
    
    We now estimate the probabilities appearing in the last three lines of (\ref{eq:cap^rho}) which we denote by $p_1,p_2$ and $p_3$ respectively. By (\ref{eq:rho_bound4}),
    \begin{equation*}\begin{aligned}
        {\rm cap}^{(\rho_n)}(X[0,(R_n)^2/I_n])\le\;& R_n^2/I_n\cdot \log^{-10}R_n+{\rm cap}(X[-R_n^2/I_n,R_n^2/I_n],R_n^2/I_n)\\
        \le\;&\widetilde C_n/4I_n^{1/2}+{\rm cap}(X[-R_n^2/I_n,R_n^2/I_n],R_n^2/I_n).\\
    \end{aligned}\end{equation*}
    Thus, we have
    \begin{align*}
            p_1\le\;&2P^{\bm0}[{\rm cap}^{(\rho_n)}(X[0,R_n^2/I_n])>\widetilde C_n/2I_n^{1/2}]\\
            \le\;& 2P^{\bm0}[{\rm cap}(X[0,(R_n)^2/I_n],R_n^2/I_n)>\widetilde C_n/4I_n^{1/2}].
    \end{align*}
    By Proposition \ref{eq:Tcap_ex} and the Markov inequality, 
    \begin{equation}\label{eq:cap_near}
        \begin{aligned}                   
            p_1\le C\mathbb E[{\rm cap}(X[0,(R_n)^2/I_n],R_n^2/I_n)]\cdot I_n^{1/2}/\widetilde C_n\le CI_n^{-1/2}.
        \end{aligned}
    \end{equation}
    Obviously, $p_2=p_3$,
    so we only bound the former from above. In the four-dimensional case, by (\ref{eq:rho_bound4}), we have
    \begin{equation*}p_2\le P^{x_{j,l}}[{\rm cap}(X[R_n^2/I_n,KR_n^2]\cap B(V',L_n),R_n^2)>\frac12\widetilde C_n/I_n^{1/2}];
    \end{equation*}
    while in the higher dimensional case,
    by the trivial upper bound of the $\rho_n$-capacity,
    \begin{equation*}
        p_2\le P^{x_{j,l}}\left[\left|X[R_n^2/I_n,KR_n^2]\cap B(V',L_n)\right|>\frac12\widetilde C_n/I_n^{1/2}\right].
    \end{equation*}
    In short, we have
    \begin{equation}\label{eq:captoY}
       p_2\le P^{x_{j,l}}[\xi>\frac12\widetilde C_n/I_n^{1/2}],
    \end{equation}
    where
    $$\xi=\begin{cases}{\rm cap}(X[R_n^2/I_n,KR_n^2]\cap B(V',L_n),R_n^2),&d=4;\\\left|X[R_n^2/I_n,KR_n^2]\cap B(V',L_n)\right|,&d\ge5.
    \end{cases}$$
    Note that since $B(V',L_n)$ consists of the $2L_n$-neighborhood of at most $(4\gamma+1)^d\cdot\beta k_1^\alpha$ good trajectories, $|B(V',2L_n)|\leq CR_n^2L_n^{d-2}I_n^{1/3}$. Thus, by (\ref{eq:captoY}) and Lemma \ref{le:intersect_cap},
    \begin{equation}\label{eq:cap_far}\begin{aligned}
        &p_2\le P^{x_{j,l}}[Y> \frac{1}{2}\widetilde C_n/I_n^{1/2}]\leq 2I_n^{1/2}/\widetilde C_n\cdot E^{x_{j,l}}[Y]\\
        \leq\;&CI_n^{1/2}/\widetilde C_n\cdot R_n^2L_n^{d-2}I_n^{1/3}\cdot(R_n^2/I_n)^{1-d/2}(1+\log R_n\cdot\mathbf 1_{d=4})^{-1}\leq CI_n^{-1/6}.
    \end{aligned}\end{equation}
    Combining (\ref{eq:cap^rho}), (\ref{eq:cap_near}) and (\ref{eq:cap_far}) gives
    \begin{equation}\label{eq:e5_final}
        \mathbb P[{\rm cap}^{(\rho_n)}(\widehat\zeta_{j,l}\backslash\overline\zeta_{j,l})>4\widetilde C_n/I_n^{1/2}|\mathcal F_{m,i},\mathcal G]\leq 4r/N
    \end{equation}
    for all large $n$.
    
    To estimate the terms on the right-hand side of (\ref{eq:e6}), we write for all $y,z\in\mathbb Z^d$
    $$\Phi(y,z)=E^{y,z}[\varphi^{(\rho_n)}(X^1[0,K(R_n)^2],X^2[0,K(R_n)^2])].$$
    By Lemma \ref{le:free_varphi},
    \begin{equation}\label{eq:Phi_bound}
        \sup_{y,z\in\mathbb Z^d}\Phi(y,z)\leq\begin{cases}CR_n,&d\geq5;\\CR_n^2\log^{-2}R_n,&d=4.\end{cases}
    \end{equation}
    Therefore, we have for all large $n$,
    \begin{equation}\label{eq:varphi_final}\begin{aligned}
        &\mathbb P[\varphi^{(\rho_n)}(\overline\zeta_{j_1,l_1},\overline\zeta_{j_2,l_2})>\widetilde C_n/I_n|\mathcal F_{m,i},\mathcal G]\\
        \leq\;\;&\mathbb P[\varphi^{(\rho_n)}(\zeta_{j_1,l_1}^f,\zeta_{j_2,l_2}^f)+\varphi^{(\rho_n)}(\zeta_{j_1,l_1}^f,\zeta_{j_2,l_2}^b)+\varphi^{(\rho_n)}(\zeta_{j_1,l_1}^b,\zeta_{j_2,l_2}^f)\\&+\varphi^{(\rho_n)}(\zeta_{j_1,l_1}^b,\zeta_{j_2,l_2}^b)>\widetilde C_n/I_n|\mathcal F_{m,i},\mathcal G]\\
        \leq\;\;&I_n/\widetilde C_n\cdot\mathbb E[\Phi(\zeta_{j_1,l_1}(t),\zeta_{j_2,l_2}(t))+\Phi(\zeta_{j_1,l_1}(t),\zeta_{j_2,l_2}(-t))+\Phi(\zeta_{j_1,l_1}(-t),\zeta_{j_2,l_2}(-t))\\&+\Phi(\zeta_{j_1,l_1}(t),\zeta_{j_2,l_2}(-t))|\mathcal F_{m,i},\mathcal G]+2r/N\\
        \overset{(\ref{eq:Phi_bound})}{\leq}&I_n^{-1}+2r/N<4r/N.
    \end{aligned}\end{equation}
    Here, $t=\lfloor R_n^2/I_n\rfloor$. 
    By (\ref{eq:e5}), (\ref{eq:e6}), (\ref{eq:e5_final}) and (\ref{eq:varphi_final}),
    $$\begin{aligned}
        \mathbb P[(E_4\cap E_5)^c|\mathcal F_{m,i},\mathcal G]\leq\;&|\{(j,l):0\leq j\leq s,1\leq l\leq N_j\}|^2\cdot 4r/N\\
        \leq\;&(\beta k_1^\alpha)^2\cdot 4r/N=r/8<r/4.
    \end{aligned}$$
    This finishes the proof.
\end{proof}

We now turn to Proposition \ref{coupling1}. Recall the symbol $P_Y^\mathcal S$ for the law of the branching process $Y$ from the beginning of this section. By simply modifying the definition of ${\rm Good}(x,i)$, we obtain the following analogue of Proposition \ref{recursive_good_exploration}.

\begin{proposition}\label{good*}
    Assume that $\aleph_m^1=x\neq\Delta$. Arbitrarily fix $\beta$, then there exists $k_1$ large and a constant $C_8^*=C_8^*(\beta,k_1,r)$, such that for all $\theta_1<C_8^*$ and all large $n$, there exists a coupling $\mathbb Q$ of $(\mathcal L_{x,i})_{1\leq i\leq\alpha}$ with conditional law on $\mathcal F_{m,0}$ and $(Y_i)_{i\geq0}$ with law $P^{S_x}$ satisfying the following. Let $\mathcal F^*_{x,i}=\sigma\{\mathcal F_{x,i},Y_j,1\leq j\leq i\}$, and ${\rm Good}^*(x,i)={\rm Good}(x,i)\cap\{\mathcal L_{x,j}=Y_j,j=1,...,i\}$, then on ${\rm Good}^*(x,i)$,
    $$\mathbb Q[{\rm Good}^*(x,i+1)|\mathcal F^*_{m,i}]>1-r$$
    for all $i\in\{0,1,...,\alpha-1\}$.
\end{proposition}
Repeating the proof of Proposition \ref{good_exploration}, but with Proposition \ref{good*} applied iteratively (instead of Proposition \ref{recursive_good_exploration}), we readily obtain Proposition \ref{coupling1}.

\subsection{Finding seeds for neighboring vertices}
This subsection is dedicated to the proof of Proposition \ref{produce_seed}. Recall $\beta$ from (\ref{eq:intro_beta}), which is the number of trajectories in a seed (c.f. Definition \ref{eq:def_seed}), and the definition of events (\ref{eq:def_good}) and (\ref{eq:def_seed}). Assuming goodness of the exploration at $x\in\mathbb L_n$, we only need to show the existence of at least $\beta$ trajectories in $\mathcal L_{x,\alpha}$ that stay inside $B_{x+R_n\cdot e}$ for all $e\in\mathcal U_d\cup(-\mathcal U_d)$. Moreover, by Proposition \ref{coupling1}, we can replace $\mathcal L_{x,\alpha}$ with $Y_\alpha$ in the last statement. Therefore, it suffices to show:

\begin{proposition}\label{Y_alpha}
    There exists $C_9>0$ such that for an arbitrary $\theta_1<C_9$, the following holds. For all $r>0$, there exists a large constant $C_{10}=C_{10}(r)$, such that for all $\beta>C_{10}$ and $\mathcal{S}\subset W^{[0,\infty)}$ composed of $\beta$ good trajectories staying inside $[0,R_n)^d$ and any $e\in\mathcal U_d\cup(-\mathcal U_d)$, the event
    $$E_e:=\{Y_\alpha\mbox{ contains }\beta\mbox{ good trajectories staying inside }B_{R_n\cdot e}^n\}$$ happens with probability at least $1-r$ if $(Y_i)_{i\geq0}\sim P_Y^\mathcal S$.
\end{proposition}
\begin{proof}[Proof of Proposition \ref{produce_seed} assuming Proposition \ref{Y_alpha}.]
    Recall $C_5,C_6$ from Proposition \ref{produce_seed} and $C_3,C_4$ from Proposition \ref{good_exploration}. Take $C_5(r)=C_{10}(r/8d)$ and $C_6(\beta,r)=C_4(\beta,r/2)\wedge C_9$. Let $$\text{Seed}'(\aleph_m^1)=\bigcap_{e\in\mathcal U_d\cup(-\mathcal U_d)}\{\exists\beta\mbox{ good trajectories}\in\mathcal L_{\aleph_m^1,\alpha}\mbox{ staying inside }\aleph_m^1+ R_n\cdot e+[0,R_n)^d\},$$
    and let $\text{Seed}''(\aleph_m^1)$ be the above event with $\mathcal L_{\aleph_m^1,\alpha}$ replaced by $Y_\alpha$.
    Take $k_1>C_3(\beta,r/2)$, $\beta>C_5$ and $\theta_1<C_6$, then by Proposition \ref{good_exploration},
    \begin{equation}\label{eq:pf_seed1}\mathbb{P}[\text{Seed}(\aleph_m^1)|\mathcal F_{m,0}]\geq \mathbb{P}[\text{Seed}'(\aleph_m^1)|\mathcal F_{m,0}]-\mathbb{P}[{\rm Good}(\aleph_m^1)^c|\mathcal F_{m,0}]>\mathbb{P}[\text{Seed}'(\aleph_m^1)|\mathcal F_{m,0}]-r/2.
    \end{equation}
    By Proposition \ref{coupling1}, for all large $n$.
    \begin{equation}\label{eq:pf_seed2}
        \mathbb{P}[\text{Seed}'(\aleph_m^1)|\mathcal{F}_{m,0}]\geq P^{S_{\aleph_m^1}}[\text{Seed}''(\aleph_m^1)]-r/4
    \end{equation}
    Moreover, by Proposition \ref{Y_alpha}, 
    \begin{equation}\label{eq:pf_seed3}
        P_Y^{S_{\aleph_m^1}}[\text{Seed}''(\aleph_m^1)]>1-2d\cdot r/8d=1-r/4.
    \end{equation}
    The proposition then follows from (\ref{eq:pf_seed1}), (\ref{eq:pf_seed2}) and (\ref{eq:pf_seed3}).
\end{proof}

The remainder of the subsection is devoted to proving Proposition \ref{Y_alpha}. For simplicity, we always assume $(Y_i)_{i\geq0}\sim P_Y^\mathcal S$ from now on. We aim to obtain the ``spatial distribution'' of trajectories in $Y_\alpha$ by studying the spatial distribution of locations where the offspring trajectories hit their parent. To this end, we define a process $(\overline Y_i)_{i\ge0}$ dominated by $(Y_i)_{i\geq0}$ whose $(i+1)$-th layer is generated from the $i$-th layer by first sampling the hitting locations whose collection will be denoted by $Z_i$, and then sampling paths from them. Noting that $(Z_i)_{i\ge0}$ form a branching process with agents in $\mathbb Z^d$, we bound from below first the number of hitting points in $Z_{\alpha-1}$ that are close enough to the center of each neighboring box $B_{x+R_n\cdot e}$, then the conditional probability that the trajectories stemming from those points stay inside $B_{x+R_n\cdot e}$, which yields Proposition \ref{Y_alpha}.

Now we make rigorous the aforementioned idea. For $\mathcal{S}\subset W^{[0,\infty)}$, we construct two processes $(\overline{Y}_i)_{i\geq0}\in\{0,1\}^{W^{[0,\infty)}}$ and 
$(Z_{i,\eta})_{i\geq0,\eta\in\overline{Y}_i}\in\{0,1\}^{\mathbb Z^d}$ simultaneously under a new probability measure $\overline{P}^\mathcal{S}$. Let $\overline{Y}_0=\mathcal{S}$. Assume that $\overline{Y}_i$ is sampled, then for $\eta\in \overline{Y}_i$, let $Z_{i,\eta}$ be an independent PPP supported on ${\rm range}(\eta)\subset\mathbb{Z}^d$ with density measure $$\sum_{x\in\widehat{\eta}}u_n^+(1-\epsilon/2)\cdot e_\eta(x)\cdot\delta_x,$$ sample one new trajectory $\zeta_x$ for each $x\in Z_{i,\eta}$ independently according to distribution $\overline{\mu}_{x,\eta}$ and set $\overline{Y}_{i+1}=\bigcup_{\eta\in\overline{Y}_i}\{\zeta_x:x\in{Z}_{i,\eta}\}$. We write $Z_i=\bigcup_{\eta\in\overline{Y}_i}Z_{i,\eta}$.

By construction, $(\overline Y_i)_{i\geq0}$ is also a branching process with agents in $W^{[0,\infty)}$. In general, $(\overline{Y}_i)_{i\geq0}$ ``trims off'' the excessive intensity of descendant trajectories introduced in $(Y_i)_{i\geq0}$ due to the finiteness of backward paths and highlights the ``branching'' nature of the process more explicitly without introducing PPPs on $W^{[0,\infty)}$. The following lemma reveals the relation between $(Y_i)_{i\geq0}$ and $(\overline{Y}_i)_{i\geq0}$, whose proof is straightforward given Lemma \ref{M_apriori}.
\begin{lemma}\label{coupling1.5}
    For large $n$, the following holds: for all $\mathcal{S}\subset\mathscr T_n$, there exists a coupling $\mathbb{Q}$ between $(Y_i)_{i\geq0}$ and $(\overline{Y}_i)_{i\geq0}$ such that $$\mathbb{Q}[Y_i\supset\overline{Y}_i,0\leq i\leq\alpha]=1.$$
\end{lemma}
\begin{proof}[Proof of Lemma \ref{coupling1.5}.]
    To prove the domination for branching processs, it suffices to prove domination for their respective offspring distributions. Obviously, the offspring distribution of the branching process $(\overline Y_i)_{i\geq0}$ at $\eta\in W^{[0,\infty)}$ is given by a PPP on $W^{[0,\infty)}$ with intensity measure
    $$\pi_1=u_n^+\cdot\sum_{x\in\widehat\eta}(1-\epsilon/2)e_\eta(x)\overline\mu_{x,\eta}.$$
    By the definition of $(Y_i)_{i\geq0}$, its offspring distribution at $\eta$ is given by a PPP with intensity
    $$\pi_2=u_n^+\cdot\sum_{x\in\widehat\eta}\sum_{m,l\in\mathbb N}\frac{\rho_n(m+l)}{\mu_1(\rho_n)+1}P^x[\widetilde H_\eta>m]\mu(x,m,l;{\rm range}(\eta))[\cdot;T_n]\geq u_n^+\cdot\sum_{x\in\widehat\eta}e_\eta(x)\overline\mu_{x,\eta}^*[\cdot;\mathscr T_n].$$
    By Lemma \ref{M_apriori}, $\overline\mu_{x,\eta}^*[\cdot;T_n]\geq(1-\epsilon/8)\overline\mu_{x,\eta}$, so $\pi_1\geq\pi_2$, which leads to the conclusion.
\end{proof}

For the aforementioned estimate regarding $Z_{\alpha-1}$ we are to establish, We define two correlated new processes $(K_i)_{i\geq0}\in\mathbb Z^d$ and $(\zeta_i)_{i\ge0}\in W^{[0,\infty)}$, and prove a LCLT-type result for $(K_i)_{i\ge0}$. The connection between these processes and $(Z_i)_{i\geq0}$ will be elaborated shortly afterwards.

For $\eta\in W^{[0,\infty)}$, let 
\begin{equation}\label{eq:def_overline_e}
    \overline e_{\eta}=\sum_{x\in\widehat\eta}e_\eta(x)\cdot\delta_x
\end{equation}
be the equilibrium measure on ${\rm range}(\eta)$ restricted to $\widehat\eta$. Recall that we denote by $\overline e_\eta^0$ the probability measure obtained by normalizing $\overline e_\eta$. Fix $\eta\in\mathscr T_n$. Let $\zeta_0=\eta$. Assume that $\zeta_i$ is sampled, we then choose $K_i$ from ${\rm range}(\zeta_i)$ according to $\overline e_{\zeta_i}^0$, and then sample $\zeta_{i+1}$ according to $\overline{\mu}_{K_i,\zeta_i}$. This way, $(K_i)_{i\ge0}$ and $(\zeta_i)_{i\ge0}$ are recursively defined. We denote the probability measure related to $(\zeta_i,K_i)$ by $\overline Q^\eta$. Using the LCLT, we are able to prove the following result:
\begin{proposition}\label{lclt}
    For all $\alpha\in\mathbb N$ in (\ref{eq:intro_alpha}), there are constants $C_{11}=C_{11}(\alpha),C_{12}=C_{12}(\alpha)$ such that for all choices $M>C_{11}$ and $\theta_1,\theta_2<C_{12}$, the following holds. There exists $C_{13}=C_{13}(k,K)$ such that for all large $n$ and all $\eta\in \mathscr T_n$ staying inside $[0,R_n)^d$,
    $$\overline Q^{\eta}[K_{\alpha-1}\in \widetilde B_{R_n\cdot e}]\geq C_{13}\alpha^{-d/2}.$$
\end{proposition}
To see the connection between $((K_i)_{i\ge0},(\zeta_i)_{i\ge0})$ and $(Z_i)_{i\ge0}$, note the following fact: if we condition on the number of descendants of every trajectory in the process $(\overline Y_i)_{i\geq0}$, then the distribution of an arbitrary element of $Z_{\alpha-1}$ is the same as $K_{\alpha-1}$ under $\overline Q^\eta$ if that element can be traced back to $\eta$. In particular, for any $A\subset\mathbb Z^d$, we have
\begin{equation}\label{eq:stat_lclt}
    \overline E^\mathcal S[|Z_{\alpha-1}\cap A|]=\overline E^\eta[|Z_{\alpha-1}|]\cdot\sum_{\eta\in\mathcal S}\overline Q^\eta[K_{\alpha-1}\in A].
\end{equation}
With the above observation and Proposition \ref{lclt}, the proof of Proposition \ref{Y_alpha} is straightforward, which can be found at the end of this subsection. 

We now aim to prove Proposition \ref{lclt}. Note that it is impossible to apply the LCLT to $(K_i)_{i\ge0}$ directly, because the distribution of $K_{i+1}$ depends on not only $K_i$ but also information offered by $\zeta_i$. To overcome this obstacle, we couple $((\zeta_i)_{i\geq0},(K_i)_{i\geq0})$ with its simplified version $((\zeta_i^\square)_{i\geq0},(K_i^\square)_{i\geq0})$ such that $(K_i^\square)_{i\geq0}$ can be viewed as a sequence of points on a simple random walk on $\mathbb Z^d$ to which we apply the LCLT.

Recall the definition of $\mu_x,x\in\mathbb Z^d$. We define the processes $((\zeta_i^\square)_{i\geq0},(K_i^\square)_{i\geq0})$ as follows. For $\eta\in W^{[0,\infty)}$, let $\zeta_0^\square=\eta$. Choose $K_0^\square$ from ${\rm range}(\zeta_0^\square)$ according to $e_{\zeta_0^\square}^0$. Assume that $K_i^\square$ is sampled, we sample $\zeta_{i+1}^\square$ according to $\mu_{K_i^\square}$, uniformly choose $j$ from $\{0,1,...,T(\zeta_{i+1}^\square)\}$, and let $K_{i+1}^\square=\zeta_{i+1}^\square(j)$ . We denote the probability measure related to $((\zeta_i^\square)_{i\geq0},(K_i^\square)_{i\geq0})$ by $\widetilde Q_\eta$.
    
\begin{remark}\label{re:Ksquare}
    $(K_i^\square)_{i\ge0}$ can be equivalently constructed in the following way. Note that the increment $K_{i+1}^\square-K_i^\square,i\ge0$ are i.i.d.\ and their distribution can be seen as that of $X_\tau$, where $\tau$ is a random time and $X\sim P^{\bm0}$ is independent of $\tau$. For $m,l\in\mathbb N$, let $\sigma_m^*$ be the uniform distribution on $\{0,1,...,m\}$, and let $\sigma_{m,l}$ be the distribution of $|\xi-m|$ where $\xi\sim\sigma_{m+l}^*$. Consider the probability measure
    $$\sigma=C\cdot\sum_{kR_n^2\le m+l\le KR_n^2}\rho_n(m+l)\cdot\sigma_{m,l}$$
    where $C=\sum_{kR_n^2\le m\le KR_n^2}(m+1)\rho_n(m)$ is a normalizing constant, then is not hard to verify that $\sigma$ is the distribution of $\tau$ mentioned above. Define the law of a simple random walk with starting point chosen from ${\rm range}(\eta)$ according to its normalized equilibrium measure by $$P^\eta\coloneq\sum_{x\in{\rm range}(\eta)}e_\eta^0(x)\cdot P^x.$$ Let $X\sim P^\eta$ and $\tau_i\sim\sigma,i\ge1$ be mutually independent, and write $S_i=\sum_{j=1}^i\tau_j$. Then, by previous discussion, $(X_{S_i})_{i\ge0}$ has the same distribution as $(K_i^\square)_{i\ge0}$ under $\widetilde Q^\eta$.
\end{remark}

Apart from the simplified processes $((\zeta_i^\square)_{i\geq0},(K_i^\square)_{i\geq0})$, we further introduce a pair of auxiliary processes $((\zeta_i^\diamondsuit)_{i\geq0},(K_i^\diamondsuit)_{i\geq0})$. For $\eta\in W^{[0,\infty)}$, let $\zeta_0^\diamondsuit=\eta$. Assume that $\zeta_i^\diamondsuit$ is sampled, we then choose $K_i^\diamondsuit$ from ${\rm range}(\zeta_i^\diamondsuit)$ according to $e_{\zeta_i^\diamondsuit}^0$, and then sample $\zeta_{i+1}^\diamondsuit$ according to $\mu_{K_i^\diamondsuit}$. We denote the probability measure related to $(\zeta_i^\diamondsuit,K_i^\diamondsuit)$ by $Q^\eta$.
\begin{remark}
    Similar to $(\zeta_i^\square)_{i\ge0}$, $(\zeta_i^\diamondsuit)_{i\ge0}$ becomes a random walk on $\mathbb Z^d$ once $\zeta_0^\diamondsuit$ is fixed.
\end{remark}

Recall $l_n$ from Lemma \ref{lem:forget}. Recall parameters $M,\theta_1$ introduced in (\ref{eq:intro_M}) and (\ref{eq:intro_theta2}) respectively for the definition of typical trajectories and $\theta_2$ introduced in (\ref{eq:intro_theta1}) for the definition of *-proper part. The major ingredient in the proof of Proposition \ref{lclt} is the following two coupling results. 
\begin{lemma}\label{branch_coupling2}
    For all $\alpha\in\mathbb N$ and $r\in(0,1)$, there exist constants $C_{14}=C_{14}(\alpha,r)$ and $C_{15}=C_{15}(\alpha,r)$, such that for all $M>C_{14}$ and $\theta_1,\theta_2<C_{15}$, all large $n$ and all $\eta\in\mathscr T_n$, we can couple $(K_i)$ with $(K^\diamondsuit_i)$ such that $d(K_i,K_i^\diamondsuit)<(i+1)\cdot 2l_n$, $i=0,1,...,\alpha-1$ with probability at least $1-r$, where $((\zeta_i)_{i\geq0},(K_i)_{i\geq0})$ (resp. $((\zeta_i^\diamondsuit)_{i\geq0},(K_i^\diamondsuit)_{i\geq0})$) follows the law under $\overline Q^\eta$ (resp. $Q^\eta$).
\end{lemma}
\begin{proof}
     We recursively construct a coupling $\mathbb Q$ satisfying
     \begin{equation}\label{eq:K_0}
        \mathbb Q[K_0=K_0^\diamondsuit]>1-r/\alpha
    \end{equation}
    and
    \begin{equation}\label{eq:K_i}
        \mathbb Q[d(K_{i+1}-K_i,K_{i+1}^\diamondsuit-K_i^\diamondsuit)<2l_n\mid\mathcal F_i]
    \end{equation}
    for all $0\le i<\alpha-1$,
    where 
    $$\mathcal F_i=\sigma\{(\zeta_j)_{0\le j\le i},(\zeta_j^\diamondsuit)_{0\le j\le i},(K_j)_{0\le j\le i},(K_j^\diamondsuit)_{0\le j\le i}\},$$
    and obtain $d(K_i,K_i^\diamondsuit)<i\cdot 2l_n$ by the triangle inequality. To start with, for $i_0<\alpha-1$, we assume that we have coupled $\left((\zeta_j)_{0\le j<i_0},(K_j)_{0\le j<i_0}\right)$ with $\left((\zeta_j^\diamondsuit)_{0\le j<i_0},(K_j^\diamondsuit)_{0\le j<i_0}\right)$ such that we have (\ref{eq:K_0}) and (\ref{eq:K_i}) for $0\le i< i_0$, and construct $\zeta_{i_0+1},K_{i_0+1},\zeta_{i_0+1}^\diamondsuit,K_{i_0+1}^\diamondsuit$ satisfying (\ref{eq:K_i}) for $i=i_0$. As alluded to in the above remark, the distribution of $K_{i_0+1}^\diamondsuit-K_{i_0}^\diamondsuit$ is irrelevant to $K_{i_0}^\diamondsuit$. Therefore, we can assume without loss of generality that $K_{i_0}^\diamondsuit=K_{i_0}$, and it suffices to prove the following result in a simpler setup: there exist constants $C_{14}(\alpha,r)$ and $C_{15}(\alpha,r)$ such that for all $M>C_{14}$, $\theta_1,\theta_2<C_{15}$ and all $\eta\in\mathscr T_n$, $x\in\widehat{\eta}$, there exists a coupling $\mathbb{Q}$ of random trajectories $\zeta^\diamondsuit\sim\mu_x,\zeta\sim\overline{\mu}_{x,\eta}$, and there are $K^\diamondsuit\sim e_{\zeta^\diamondsuit}^0$ under $\mathbb{Q}[\cdot|\zeta^\diamondsuit]$ and $K^*\sim\overline e_{\zeta}^0$ under $\mathbb{Q}[\cdot|\zeta]$ such that
    \begin{equation}\label{5.5}
        \mathbb{Q}[d(K,K^\diamondsuit)>2l_n]<r/\alpha.
    \end{equation}
    
    By the definition of $\overline\mu_{x,\eta}^*$ and $\mu_x$ and Lemma \ref{lem:forget}, for sufficiently small $\theta_2$ and $\theta_1$, we can couple $\zeta_1\sim\overline\mu_{x,\eta}^*$ with $\zeta^\diamondsuit\sim\mu_x$ such that with probability at least $1-r/8\alpha$, \begin{equation}\label{eq:same_length}
        T(\zeta_1)=T(\zeta^\diamondsuit)\mbox{ and }\zeta_1(t^*)=\zeta^\diamondsuit(t^*)=x
    \end{equation}
    almost surely, where $t^*=\tau_\eta^{\zeta_1}$ (c.f. (\ref{eq:def_hitting})), and \begin{equation}\label{eq:close}
        \begin{aligned}
            &\zeta_1[0,(t^*-t_n)\vee0]\simeq\zeta^\diamondsuit[0,(t^*-t_n)\vee0],\zeta_1[t^*,T(\zeta_1)]\simeq\zeta^\diamondsuit[t^*,T(\zeta^\diamondsuit)]\\&\mbox{and }|\zeta_1(s)-\zeta^\diamondsuit(s)|\leq l_n,\forall 0\le s\le T(\zeta_1).
        \end{aligned}
    \end{equation}
     Denote the event that (\ref{eq:same_length}) and (\ref{eq:close}) both hold by $E(\zeta_1,\zeta^\diamondsuit)$.

    Similar to (\ref{eq:mu(Tn)_lower}), there exist constants $C_1^*(\alpha,r)$ and $C_{14}(\alpha,r)$ such that for all $\theta_1,\theta_2<C_1^*$ and $M>C_{14}$, a trajectory following $\mu_x$ or $\overline\mu_{x,\eta}^*$ is typical with probability at least $1-r/8\alpha$ as $n\to\infty$, which implies that $d_\text{TV}(\overline\mu_{x,\eta},\overline\mu_{x,\eta}^*)<r/8\alpha$. In that case, the coupling mentioned in the previous paragraph still works if we replace $\zeta_1\sim\overline\mu_{x,\eta}^*$ by $\zeta\sim\overline\mu_{x,\eta}$, and the similarly defined event $E(\zeta,\zeta^\diamondsuit)$ happens with probability at least $1-r/4\alpha$.

    Define an event
    $$E'\coloneq E(\zeta,\zeta^\diamondsuit)\cap\{{\rm cap}(\zeta^\diamondsuit)>(1-\theta_1^2)h_nT(\zeta^\diamondsuit)\},$$
    where 
    $$h_n=\varepsilon_d(1+\log\mu_1(\rho_n)\mathbbm1_{d=4})^{-1}.$$
    Then, by (\ref{eq:cap_con}), $\mathbb Q[E']\ge 1-r/2\alpha$ for all large $n$.
    Now, on $E'$, we aim to couple $K$ and $K^\diamondsuit$ conditionally on $\zeta$ and $\zeta^\diamondsuit$. Let $T=T(\zeta)=T(\zeta^\diamondsuit)$. Write $t^*=t(\zeta,\eta)$, $\zeta_-=\zeta[0,(t^*-t_n)\vee0]$, $\zeta_{\rm m}=\zeta[(t^*-t_n)\vee0,t^*]$ and $\zeta_+=\zeta[t^*,T]$. Define $e_\zeta^*=e_{\zeta_-}+e_{\zeta_+}$, and let $e_\zeta^{**}$ be the restriction of $e_\zeta$ to ${\rm range}(\zeta_-)\cup{\rm range}(\zeta_+)$. Define $\zeta_\pm^\diamondsuit,\zeta_{\rm m}^\diamondsuit,e_{\zeta^\diamondsuit}^*$ and $e_{\zeta^\diamondsuit}^{**}$ similarly. Note that $\zeta$ is always typical, so by definition, $\zeta\in\mathcal E_4$ can be divided into pieces of length $T_n'$(see \ref{eq:T'} for definition) with capacity no larger than $(1+\theta_1^2)h_nT_n'$, while at the same time we have ${\rm cap}(\zeta)\ge(1-\theta_1^2)h_nT$ by $\zeta\in\mathcal E_2$. If we divide $\zeta_-$ and $\zeta_+$ into the aforementioned pieces of length $T_n'$, at most one piece may be used twice, so the total length of all the pieces used in both coverings is at most $T+T_n'$. Hence, we have
    $${\rm cap}(\zeta_-)+{\rm cap}(\zeta_+)\le(1+\theta_1^2)h_n(T+T_n'),$$
    implying
    \begin{equation}\label{eq:c+c-c}
        {\rm cap}(\zeta_-)+{\rm cap}(\zeta_+)-{\rm cap}(\zeta)\le (1+\theta_1^2)h_n(T+T_n')-(1-\theta_1^2)h_nT<C\theta_1^2\widetilde C_n.
    \end{equation}
    Now, we claim that conditionally on $\zeta,\zeta^\diamondsuit$ with restriction on $E'$, we can couple $K^*$ and $K^\diamondsuit$ such that
    \begin{equation}\label{eq:TV_KKdiamond}
        \mathbb Q[K^*=K^\diamondsuit\in{\rm range}(\zeta_-)\mbox{ or }K^*-K^\diamondsuit=\zeta(t^*)-\zeta^\diamondsuit(t^*)\mid\zeta,\zeta^\diamondsuit]>1-C(\theta_1+\theta_2)>1-r/2\alpha,
    \end{equation}
    where the second inequality holds for $\theta_1,\theta_2$ small. Then, (\ref{5.5}) follows since $\mathbb Q[E']>1-r/2\alpha$. 
    
    Now we aim to prove (\ref{eq:TV_KKdiamond}). Recall the restriction of equilibrium measure $\overline e_\zeta$ from (\ref{eq:def_overline_e}). Note the following inequalities:
    \begin{equation}\label{eq:align}
        \begin{aligned}
        d_{\rm TV}(\overline e_\zeta,e_\zeta)\le\;& C(\theta_1+\theta_2)\widetilde C_n;\\
        d_{TV}(e_\zeta,e_\zeta^{**})\le\;& t_n;\\
        d_{\rm TV}(e_\zeta^{**},e_\zeta^*)=\;&{\rm cap}(\zeta_-)+{\rm cap}(\zeta_+)-e_\zeta(\zeta_-\cup\zeta_+)\\
        \le\;&\left({\rm{\rm cap}(\zeta_-)+{\rm cap}(\zeta_+)-{\rm cap}(\zeta)}\right)+{\rm cap}(\zeta_{\rm m})\\
        \le\;&C\theta_1^2\widetilde C_n+t_n.
        \end{aligned}
    \end{equation}
    Here, the first inequality holds, since by the definition of $\widehat\zeta$,
    \begin{equation*}
        {\rm d}_{\rm TV}(\overline e_\zeta,e_\zeta)=e_\zeta({\rm range}(\zeta)\setminus\widehat\zeta)\le\begin{cases}\theta_2 T(\zeta),&d\ge5;\\\theta_2T(\zeta)\log^{-1}\mu_1(\rho_n)+\theta_1^{-1}\left({\rm cap}(\zeta,R_n^2)-{\rm cap}(\zeta)\right),&d=4\end{cases}
    \end{equation*}
    which, combined with the requirements of $\mathcal E_1$ and $\mathcal E_2$ in the definition of typical trajectories, yields
    \begin{equation}\label{eq:equi_setminus}
            {\rm d}_{\rm TV}(\overline e_\zeta,e_\zeta)=e_\zeta({\rm range}(\zeta)\setminus\widehat\zeta)\le C(\theta_1+\theta_2)\widetilde C_n
    \end{equation}
    Summing up the inequalities aligned in (\ref{eq:align}), we get
    $$d_{\rm TV}(\overline e_\zeta,e_\zeta^*)<C(\theta_1+\theta_2)\widetilde C_n$$
    since $\theta_1<1$ and $t_n\ll\widetilde C_n$. Note that 
    $$e_\zeta^*({\rm range}(\zeta))={\rm cap}(\zeta_-)+{\rm cap }(\zeta_+)\ge{\rm cap}(\zeta)-t_n\ge Ck\widetilde C_n$$ 
    for some constant $C$ and all large $n$. Hence,
    \begin{equation}\label{eq:add0}
        d_{\rm TV}(\overline e_\zeta^0,(e_\zeta^*)^0)=d_{\rm TV}(\overline e_\zeta^0,\frac1{e_\zeta^*(\zeta)}\cdot\overline e_\zeta)+\frac1{e_\zeta^*(\zeta)}\cdot d_{\rm TV}(\overline e_\zeta,e_\zeta^*)<C(\theta_1+\theta_2),
    \end{equation}
    where the constant $C=C(d,k)$.
     Similarly,
    \begin{equation}
        d_{\rm TV}(e_{\zeta^{\diamondsuit}}^0,(e^*_{\zeta^\diamondsuit})^0)<C(\theta_1+\theta_2).
    \end{equation}
    Thus, we can couple $K^*$ with $K_a\sim(e_\zeta^*)^0$ and $K^\diamondsuit$ with $K_b\sim(e^*_{\zeta^\diamondsuit})^0$ such that the probability of $\{K^*= K_a\}$ and $\{K^\diamondsuit=K_b\}$ given $\zeta,\zeta^\diamondsuit$ are both at least $1-C(\theta_1+\theta_2)$. 
    (\ref{eq:TV_KKdiamond}) then follows since $\zeta_-\simeq\zeta_-^\diamondsuit$ and $\zeta_+=\zeta_+^\diamondsuit$, which yields that $K_a$ and $K_b$ can be coupled so that $$K_a=K_b\in{\rm range}(\zeta_-)\mbox{ or }K_a-K_b=\zeta(t^*)-\zeta^\diamondsuit(t^*)$$ almost surely.

    It remains to construct $K_0$ and $K_0^\diamondsuit$ satisfying (\ref{eq:K_0}). To see this, note that $\eta\in \mathscr T_n$, and by a similar argument as above we have
    \begin{equation}\label{eq:TV_initial}
        d_{\rm TV}(e_\eta^0,\overline e_\eta^0)<C(\theta_1+\theta_2)<r/\alpha
    \end{equation}
    for all small $\theta_1,\theta_2$. To conclude the proof, note that we can pick some $C_{13}(\alpha,r)\in(0,C_1^*(\alpha,r))$ such that whenever $\theta_1,\theta_2<C_{13}$, (\ref{eq:TV_KKdiamond}) and (\ref{eq:TV_initial}) both hold.
\end{proof}

\begin{lemma}\label{branch_coupling3}
    Arbitrarily choose $\alpha\in\mathbb N$ and $r\in(0,1)$. Then, for all large $n$ and all $\eta\in W^{[0,\infty)}$,
    we can couple $(K_i^\diamondsuit)_{i\ge0}$ with $(K_i^\square)_{i\ge0}$ such that $d(K_{\alpha-1}^\square,K_{\alpha-1}^\diamondsuit)<R_n/16$, $i=0,1,...,\alpha-1$ with probability at least $1-r$, where $((\zeta_i^\square)_{i\geq0},(K_i^\square)_{i\geq0})$ (respectively, $((\zeta_i^\diamondsuit)_{i\geq0},(K_i^\diamondsuit)_{i\geq0})$) follows the law under $\widetilde Q^\eta$ (respectively, $Q^\eta$).
\end{lemma}
\begin{proof}
    We recursively construct a coupling $\mathbb Q$ such that
    \begin{equation}\label{eq:K_0ds}
        \mathbb Q[K_0^\diamondsuit=K_0^\square]=1,
    \end{equation}
    and
    \begin{equation}\label{eq:K_ids}
        \mathbb Q[d(K_{i+1}^\diamondsuit-K_i^\square,K_{i+1}^\diamondsuit-K_i^\diamondsuit)< R_n/16\alpha\mid\mathcal F_i]>1-r/\alpha
    \end{equation}
    for all $i\ge0$, where 
    $$\mathcal F_i=\sigma\{(\zeta_j^\square)_{0\le j\le i},(\zeta_j^\diamondsuit)_{0\le j\le i},(K_j^\square)_{0\le j\le i},(K_j^\diamondsuit)_{0\le j\le i}\}.$$
    (\ref{eq:K_0ds}) is trivial because $K_0^\diamondsuit$ and $K_0^\square$ are identically distributed. Similar to Lemma \ref{branch_coupling2}, we work in a simpler setup to obtain (\ref{eq:K_ids}). Let $\zeta^\diamondsuit,\zeta^\square\sim\mu_0$. Conditionally on $\zeta^\diamondsuit$, choose $K^\diamondsuit$ according to $e_{\zeta^{\diamondsuit}}^0$; conditionally on $\zeta^\square$, let $K^\square=\zeta^\square(j)$, where $j$ is uniformly chosen from $\{0,1,...,T(\zeta^\square)\}$. In order to obtain (\ref{eq:K_ids}), it suffices to construct a coupling $\mathbb Q$ such that
    \begin{equation}\label{eq:Kds}
        \mathbb Q[d(K^\diamondsuit,K^\square)\ge R_n/16\alpha]<r/\alpha
    \end{equation}
    for all large $n$.
    First, we take $\zeta^\diamondsuit=\zeta^\square=\zeta\sim\mu_0$, and let 
    $$\xi_j=\zeta[jT'',((j+1)T''-1)\wedge T(\zeta)],0\le j\le N_1=\lfloor T(\zeta)/T''\rfloor,$$ 
    where $T''=\lfloor \delta R_n^2\rfloor$ and $\delta>0$ is a small but fixed number to be specified later.
    Recall $h_n$ from the proof of Lemma \ref{branch_coupling2}, and let $E=E_1\cap E_2$ where
    \begin{align*}
        E_1=\{&{\rm cap}(\xi_j)\le (1+r/16\alpha N_1)h_nT'',\forall 0\le j\le N_1\\
        &\mbox{and }{\rm cap}(\zeta[0,N_1T''])\ge(1-r/16\alpha N_1)h_nN_1T''\},\\
        E_2=\{&{\rm diam}(\xi_j)<R_n/16\alpha,\forall 0\le j\le N_1\}.
    \end{align*}
    By Proposition (\ref{eq:cap_con}), $\mathbb Q[E_1^c]<r/4\alpha$ for all large $n$. By a standard Gaussian estimate and union bound, there are constants $C=C(d,K)>0$ and $C'=C'(d,\alpha)>0$ such that
    $$\limsup_{n\to\infty}\mathbb Q[E_2^c]\le C\delta^{-1}\exp(-C'\delta^{-1/2}).$$
    Thus, we can take $\delta$ small such that $\mathbb Q[E]>1-r/2\alpha$ for all large $n$.

    On $E$, we construct a coupling between $K^\diamondsuit$ and $K^\square$ conditionally on $\zeta$. Let $K'\sim e_\zeta'\coloneq\left(\sum_{i=0}^{N_1-1}e_{\xi_i}\right)^0$, then we can sample $K'$ as follows. First, randomly choose $j'\in\{0,1,...,N_1-1\}$ according to the probability measure proportional to $\sum_{i=0}^{N_1-1}{\rm cap}(\xi_i)\cdot\delta_i$, then choose $K'$ from ${\rm range}(\xi_{j'})$ according to $e_{\xi_{j'}}^0$. We can similarly sample $K^\square$ by taking $j^\square\in\{0,1,...,N_1\}$ according to $\sum_{i=0}^{N_1}\frac {T(\xi_i)}{T(\zeta)}\cdot\delta_i$ and then letting $K^\square=\xi_{j^\square}(t)$, where $t$ is uniformly chosen from $\{0,1,...,T(\xi_{j^\square})\}$.
    
    Note that on $E$, 
    $$e_\zeta(\xi_i)\le{\rm cap}(\xi_j)\le(1+r/16\alpha N_1)h_nT''$$
    for all $0\le i\le N_1$ and 
    $$e_\zeta(\xi_i)\ge{\rm cap}(\zeta[0,N_1T''])-\sum_{j\neq i,j<N_1}{\rm cap}(\xi_j)\ge(1-r/8\alpha) h_nT''$$
    for all $0\le i<N_1$. Therefore, for all small $\delta$,
    \begin{equation}\label{eq:zetasum}
        d_{\rm TV}(e_\zeta,\sum_{i=0}^{N_1-1}e_{\xi_i})\le\sum_{i=0}^{N_1}{\rm cap}(\xi_i)-{\rm cap}(\zeta[0,(N-1)T''])\le\frac r{8\alpha N_1}h_nT(\zeta)+2h_nT''<\frac{1}{8}r\alpha^{-1}h_nT(\zeta)
    \end{equation}
    and 
    \begin{equation}\label{eq:sumuni}
       \sum_{j=0}^{N_1-1}|{\rm cap}(\xi_j)-h_nT''|+h_nT''\le\frac18r\alpha^{-1}h_nT(\zeta)+h_nT''\le \frac3{16}r\alpha^{-1}h_nT(\zeta).
    \end{equation}
    The total mass of $e_\zeta$ is ${\rm cap}(\zeta)>(1-r/16\alpha N_1)h_n N_1T''>\frac78h_nT(\zeta)$ when $\delta$ is small so that $N_1>1$. Thus, by an argument similar to (\ref{eq:add0}), we obtain
    \begin{equation}\label{eq:zetasum'}
        d_{\rm TV}\left(e_\zeta^0,e_\zeta'\right)<r/4\alpha
    \end{equation}
    and 
    \begin{equation}\label{eq:sumuni'}
       \sum_{j=0}^{N_1}\left|s_j-\frac{T(\xi_j)}{T(\zeta)}\right|<r/4\alpha.
    \end{equation}
    Here 
    $$s_j=\begin{cases}
        \frac{{\rm cap}(\xi_j)}{\sum_{i=0}^{N_1-1}{\rm cap}(\xi_i)},&0\le j<N_1;\\
        0,&j=N_1
    \end{cases}$$
    is the probability that $j'=j$ while we sample $K'$ in the aforementioned way, and $\frac{T(\xi_j)}{T(\zeta)}$ is the probability that $j^\square=j$. This indicates that we can couple $j',j^\square$ such that $j'=j^\square$ with probability at least $1-r/4\alpha$.
    Thus, by (\ref{eq:zetasum}) and (\ref{eq:sumuni}), there exists a coupling between $K^\diamondsuit$ and $K'$ such that $K^\diamondsuit=K'$ with probability at least $1-r/4\alpha$, and one between $K'$ and $K^\square$ such that they are both in the range of some $\xi_j$ with probability larger than $1-r/4\alpha$. Hence, we can couple $K^\diamondsuit$ and $K^\square$ such that
    $$\mathbb Q[\exists 0\le j\le N_1\mbox{ s.t. }K^\diamondsuit,K^\square\in{\rm range}(\xi_j)\mid\zeta]>1-r/2\alpha.$$ (\ref{eq:Kds}) then follows because $\mathbb Q[E]>1-r/2\alpha$ and ${\rm diam}(\xi_j)<R_n/16\alpha$ for all $0\le j\le N$ on $E$.
\end{proof}
We now present the proof of Proposition \ref{lclt}.

\begin{proof}[Proof of Proposition \ref{lclt}.]
For all $\mu\in W^{[0,\infty)}$ staying inside $B_0^n$, let $X$, $\tau_i,S_i,i\ge1$ be as in Remark \ref{re:Ksquare}. Apparently, we have
$$\mathbb P\left[\tau_i\ge \frac14kR_n^2\right]\ge\frac14$$
and $\tau_i\le KR_n^2$ almost surely, $i\ge1$.
Therefore, $\frac14kR_n^2\le S_{\alpha-1}\le\alpha KR_n^2$ happens with probability at least $\frac14$.
Let $\overline B_1:=B(\lfloor R_n/2\rfloor\cdot e,R_n/16)$ and $\overline B_2:=B(\lfloor R_n/2\rfloor\cdot e,R_n/8)$ be two smaller concentric boxes inside $\widetilde{B}_{R_n\cdot e}$. By \cite[Theorem 1.2.1]{trove.nla.gov.au/work/15980712}, there exists $C_{13}=C_{13}(k,K)$ such that
$$\widetilde Q^\eta[K_{\alpha-1}^\square\in\overline B_1]=\mathbb P[X_{S_{\alpha-1}}\in\overline B_1]\ge\frac14\cdot\inf_{kR_n^2/4\le j\le\alpha KR_n^2}\mathbb P[X_j\in\overline B_1]\geq 2C_{13}\alpha^{-d/2}$$
for all large $n$. Recall $C_{14}$, $C_{15}$ from Lemma \ref{branch_coupling2}, and take $C_{11}(\alpha)=C_{14}(\alpha,r)$ and $C_{12}(\alpha)=C_{15}(\alpha,r)$, where $r=\frac12C_{13}\alpha^{-d/2}$. Thus, for any $M>C_{11}$, $\theta_1,\theta_2<C_{12}$ and all large $n$, by Lemma \ref{branch_coupling2},
$${\rm d}(K_{\alpha-1},K_{\alpha-1}^\diamondsuit)\le\alpha\cdot 2l_n<R_n/8.$$
Moreover, by Lemma \ref{branch_coupling3},
${\rm d}(K_{\alpha-1}^\diamondsuit,K_{\alpha-1}^\square)<R_n/16$ for all large $n$. Therefore,
$$\overline Q^\eta[K_{\alpha-1}\in \widetilde B_{R_n\cdot e}]\ge Q^\eta[K_{\alpha-1}^\diamondsuit\in \overline B_2]-C_{13}\alpha^{-d/2}/2\geq \widetilde Q^\eta[K_{\alpha-1}^\square\in \overline B_1]-C_{13}\alpha^{-d/2}\geq C_{13}\alpha^{-d/2},$$
finishing the proof.
\end{proof}

Finally, we prove Proposition \ref{Y_alpha} using Proposition \ref{lclt}.
\begin{proof}[Proof of Proposition \ref{Y_alpha}.]
    Recall the function $f$ from Lemma \ref{lemma2.1}, and let $\alpha$ be large such that
    \begin{equation}\label{eq:alpha_condition}
        \frac{1}{2}f(256K)\cdot\frac{k}{2\varepsilon_d}C_{13}\cdot(1+\epsilon/8)^\alpha\alpha^{-d/2}>10.
    \end{equation}
    Recall the constants $C_{11},C_{12}$ from Proposition \ref{lclt}, $C_9$ from Proposition \ref{Y_alpha} and $C_1,C_2$ from Lemma \ref{M_apriori}. Having fixed $\alpha$, we now take $C_9=C_{12}(\alpha)\wedge C_1$, and fix
    \begin{equation}\label{eq:fix_Mtheta1}
     M>C_{11}(\alpha)\vee C_2\mbox{ and }\theta_2<C_1\wedge C_{12}(\alpha).
    \end{equation} Then, by Proposition \ref{lclt}, for all large $n$ and $\eta\in\mathcal{S}$,
    \begin{equation}\label{5.9}
        \overline Q^\eta[K_{\alpha-1}\in \widetilde{B}_{R_n\cdot e}^n]\geq C_{13}\alpha^{-d/2}.
    \end{equation} For $L>0$, recall that $W_L$ is the subset of $W^{[0,\infty)}$ with diameter at most $L$. By Lemma \ref{lemma2.1}, the following holds: for all $x\in\mathbb{Z}^d$,
    \begin{equation}\label{5.10}
        \mu_x(W_{R_n/16})>f(256K).
    \end{equation}
    This lower bound, together with Lemma \ref{le:good_trivial}, implies
    \begin{equation}\label{eq:sprout}
        \overline\mu_{x,\eta}(W_{R_n/8})\geq\frac{1}{2}f(256K).
    \end{equation}
    Let $\xi=|\{\eta\in\overline Y_\alpha:\eta\mbox{ stays within }e+[0,R_n)^d\}|$, we show that $\xi>\beta$ with high probability. We first estimate the first moment. Applying (\ref{eq:stat_lclt}) and (\ref{eq:sprout}) in turn gives
    \begin{equation}\label{5.11}
    \begin{aligned}
        \overline E^\mathcal{S}[\xi]
        \geq\;&\overline P^{\mathcal S}[|Z_{\alpha-1}\cap\widetilde B_{R_n\cdot e}|]\cdot\frac{1}{2}f(256K)\\
        \geq\;&\overline P^{\mathcal S}[|Z_{\alpha-1}|]\cdot C_{13}\alpha^{-d/2}\cdot\frac{1}{2}f(256K)\\
        \geq\;&\frac{k\beta}{2\varepsilon_d}(1+\epsilon/8)^\alpha\cdot C_{13}\alpha^{-d/2}\cdot \frac{1}{2}f(256K)>10\beta.
    \end{aligned}
    \end{equation}
    The third inequality is a consequence of Lemma \ref{M_apriori}, while the last one is due to (\ref{eq:alpha_condition}). Moreover, we have the following control over the variance of $\xi$. For $W\subset \mathscr T_n$, let $\text{var}^W$ denote the variance under $\overline P^W$, and abbreviate $\text{var}^\eta$ when $W=\{\eta\}$. For $i\ge0$, since $\overline Y_i$ only contains typical trajectories whose capacities are well controlled, it is easy to verify that conditioning on $\overline Y_i$, the number of offspring of each element of $\overline Y_i$ follows a Poisson distribution with parameter smaller than $\lambda=2K$. Thus, by induction, we have for all $\eta\in\mathscr T_n$,
    $$\text{var}^\eta(\xi)\leq \overline E^\eta[\xi^2]\leq \overline E^\eta[|\overline Y_\alpha|^2]\le\bigl[2(\lambda^2+\lambda)\bigr]^\alpha.$$ 
    Therefore,
    \begin{equation}\label{5.12}
        \text{var}^\mathcal{S}(\xi)=\sum_{\eta\in\mathcal{S}}\text{var}^\mathcal{\eta}(\xi)\leq\beta\bigl[2(\lambda^2+\lambda)\bigr]^\alpha.
    \end{equation}
    By (\ref{5.11}), (\ref{5.12}) and the Markov inequality,
    $$\overline P^\mathcal{S}[\xi<\beta]\leq\overline P^\mathcal S\left[\left|\xi-\overline E^\mathcal S[\xi]\right|>9\beta\right]\le9^{-2}\beta^{-2}{\rm var}^\mathcal S(\xi)\le9^{-2}\beta^{-1}[2(\lambda^2+\lambda)]^\alpha.$$
    The conclusion then follows by taking $\beta$ large.
\end{proof}

\section*{Acknowledgments}
We warmly thank Eviatar Procaccia for fruitful discussions. YB and XL are supported by National Key R\&D Program of China (No.\ 2020YFA0712900 and No.\ 2021YFA1002700). BR is partially supported by the grant NKFI-FK-142124 of NKFI (National Research, Development and Innovation Office), and the ERC Synergy Grant No.\ 810115 - DYNASNET. YZ is supported by NSFC-12271010, National Key R\&D Program of China (No.\ 2020YFA0712902) and the Fundamental Research Funds for the Central Universities and the Research Funds of Renmin University of China 24XNKJ06.

\appendix
\section{Proofs of technical lemmas}
In this section, we present the proofs of Lemmas \ref{M_apriori}, \ref{le:intersect_cap} and \ref{le:free_varphi}. For the proof of Lemma \ref{M_apriori}, we introduce the following analog of Lemma \ref{le:good_trivial} for random walk paths whose law involve conditioning.
\begin{lemma}\label{le:good_trivial2}
    Let $\psi$ be the same as in Lemma \ref{le:good_trivial} and take $\eta\in \mathscr T_n,x \in\widehat\eta$ and $m,l\in\mathbb N$ satisfying $kR_n^2\le m+l\le KR_n^2$. Then, for all large $n$,
    $$\mu(x,m,l;{\rm range}(\eta))[\mathscr T_n^c]<\psi\left((M-1)/K^{1/2}\right)+4\theta_1.$$
\end{lemma}
\begin{proof}
    By Lemma \ref{lem:forget}, we can construct a coupling $\mathbb Q$ between $\zeta\sim\mu(x,m,l)$ and \\
    $\zeta^*\sim\mu(x,m,l;{\rm range}(\eta))$ such that the event
    $$E=\{\zeta[0,m-t_n]\simeq\zeta^*[0,m-t_n],\zeta[m,m+l]\simeq\zeta^*[m,m+l],|\zeta(s)-\zeta^*(s)|\le l_n,\forall0\le s\le m+l\}$$
    happens with probability at least $1-2\theta_1-R_n^{-c/2}$.
    Taking advantage of the coupling, we have
    \begin{align*}
        \mu(x,m,l;{\rm range}(\eta))[\mathcal E_1^c]=\;&\mathbb Q[{\rm range}(\zeta^*)\not\subset B(x,MR_n)]\\
    \le\;&\mathbb Q[E\cap\{{\rm range}(\zeta^*)\not\subset B(x,MR_n)\}]+2\theta_1+R_n^{-c/2}\\
    \le\;&\mathbb Q[{\rm range}(\zeta)\not\subset B(x,MR_n-l_n)]+2\theta_1+R_n^{-c/2}\\
    \le\;&\mu(x,m,l)[\mathcal E_1(k,K,M-1)^c]+2\theta_1+R_n^{-c/2}.
    \end{align*}
    It then follows from (\ref{eq:cE1}) that
    \begin{equation}\label{eq:cE1'}
        \mu(x,m,l;{\rm range}(\eta))[\mathcal E_1^c]\le\frac34\psi\left((M-1)/K^{1/2}\right)+2\theta_1
    \end{equation}
    for all large $n$.
    Likewise, when $d=4$, since on the event $E$, $$\begin{aligned}
        {\rm cap}(\zeta^*[jT_n',(j+1)T_n'\wedge T(\zeta^*)])\le\;&{\rm cap}(\zeta[jT_n',(j+1)T_n'\wedge T(\zeta)])+t_n\\
        \le\;&{\rm cap}(\zeta[jT_n',(j+1)T_n'\wedge T(\zeta)])+\frac12\theta_1^2\cdot\frac{\pi^2}{8}T_n'\log^{-1}\mu_1(\rho_n),
    \end{aligned}$$ 
    and
    $$\begin{aligned}
        {\rm diam}(\zeta^*[jT_n',(j+1)T_n'\wedge T(\zeta^*)])\le\;&{\rm diam}(\zeta[jT_n',(j+1)T_n'\wedge T(\zeta)])+l_n\\
        \le\;&{\rm diam}(\zeta[jT_n',(j+1)T_n'\wedge T(\zeta)])+\frac14T^{1/2}\log^{-c}\mu_1(\rho_n)
    \end{aligned}$$ for all $j\in\mathbb N$ and all large $n$, we obtain
\begin{equation}\label{eq:cE4'}\begin{aligned}
        &\mu(x,m,l;{\rm range}(\eta))[\mathcal E_4^c]\\
        \le\;&\mathbb Q\Big[\exists j\in\mathbb N,{\rm cap}\left(\zeta[jT_n',(j+1)T_n'\wedge T(\zeta)]\right)>(1+\frac12\theta_1^2)\frac{\pi^2}8T_n'\log^{-1}\mu_1(\rho_n)\\
        &\mbox{or }{\rm diam}\left(\zeta[jT_n',(j+1)T_n'\wedge T(\zeta)]\right)>\frac14T^{1/2}\log^{-c}\mu_1(\rho_n)\Big]+2\theta_1+R_n^{-c/2}\\
\le\;&C\log^{-1}R_n+2\theta_1+R_n^{-c/2},
    \end{aligned}\end{equation}
    where we have estimated the probability appearing in the second and third line similarly as in (\ref{eq:cE4}). For $d\geq 5$, we can get the same upper bound in a similar fashion.
    Finally, since $x\in\widehat\eta$, $$\mu(x,m,l)[\mathcal E]\ge \theta_2\log^{-1}\mu_1(\rho_n)$$ for all large $n$, where 
    $$\mathcal E=\{\eta'\in W^{[0,\infty)}:\eta'[0,m)\cap{\rm range}(\eta)=\emptyset\}.$$ Thus, by (\ref{eq:cE2}) and (\ref{eq:cE3}),
    \begin{equation}\label{eq:cE23'}\begin{aligned}
        &\mu(x,m,l;{\rm range}(\eta))[\mathcal E_2^c\cup\mathcal E_3^c]\le \frac{\mu(x,m,l)[\mathcal E_2^c\cup\mathcal E_3^c]}{\mu(x,m,l)[\mathcal E]}\\
        \le\;&\frac{C\log^{-2}R_n+C\log^{8c/3}R_n\exp(-\log^{2c/3}R_n)}{\log^{-1}R_n}\le C\log^{-1}R_n
    \end{aligned}\end{equation}
    for all large $n$. The conclusion follows from (\ref{eq:cE1'}), (\ref{eq:cE4'}) and (\ref{eq:cE23'}).
\end{proof}
\begin{proof}[Proof of Lemma \ref{M_apriori}]
    Recall the notation $\mu_2^{(m)}(\rho_n),m>0$ which is the tail of the second moment of $\rho_n$ from Section 2. Since $(\rho_n)_{n\ge 0}$ is an appropriate family, by (\ref{eq:appropriate}), we can choose $K>1$ large such that
    \begin{equation}\label{eq:propK}
        \mu_2^{(KR_n^2)}(\rho_n)\le\frac 1{32}\epsilon\mu_2(\rho)
    \end{equation}
    for all $n\in\mathbb N$.
    Fix a $K$ such that the above estimate holds, and arbitrarily fix $k$ satisfying $k^2<\epsilon/32$.
    For all $\theta_1<\epsilon/64K$ and large $M$  so that $\psi((M-1)/K^{1/2})<\epsilon/16K$, we obtain by Lemmas \ref{le:good_trivial} and \ref{le:good_trivial2} that 
    $$\mu_x(\mathscr T_n)\wedge\overline{\mu}_{x,\eta}^*(\mathscr T_n)>1-\epsilon/8K>1-\epsilon/8.$$
    Now we aim at estimating $e_\zeta(\widehat\zeta)$. By the fact that $\zeta\in\mathscr T_n\subset\mathcal E_2$,
    \begin{equation}\label{eq:cap_zeta1}
        {\rm cap}(\zeta)\ge
        (1-\theta_1^2)\varepsilon_dT(\zeta)(1+\log\mu_1(\rho_n)\mathbbm1_{d=4})^{-1}.
    \end{equation}
    Let $\zeta^*\sim\overline\mu_{x,\eta}^*$, then by (\ref{eq:propK}) and $k^2<\epsilon/32$,
    $$\begin{aligned}
        \mathbb E[T(\zeta^*)]=\;&\frac{\sum_{kR_n^2\le m\le KR_n^2}m^2\rho_n(m)}{\sum_{kR_n^2\le m\le KR_n^2}(m+1)\rho_n(m)}\ge\frac{\mu_2(\rho_n)-\mu_2^{(KR_n^2)}(\rho_n)-(kR_n^2)^2}{\mu_1(\rho_n)+1}\\
        \ge\;&\frac{\mu_2(\rho_n)-\frac1{32}\epsilon\mu_2(\rho_n)-k^2\mu_2(\rho_n)}{\mu_1(\rho_n)+1}\ge(1-\epsilon/16)\frac{\mu_2(\rho_n)}{\mu_1(\rho_n)}.
    \end{aligned}$$
    Thus, we obtain 
    $$\begin{aligned}
        \mathbb E[T(\zeta)]\ge\;& \mathbb E[T(\zeta^*)1_{\zeta^*\in T_n}]\ge\mathbb E[T(\zeta^*)]-\frac\epsilon{8K}\cdot KR_n^2\\
        \ge\;&(1-\epsilon/16)\frac{\mu_2(\rho_n)}{\mu_1(\rho_n)}-\frac18\epsilon\frac{\mu_2(\rho_n)}{\mu_1(\rho_n)}=(1-\frac3{16}\epsilon)\frac{\mu_2(\rho_n)}{\mu_1(\rho_n)},
    \end{aligned}$$
    where in the second inequality we have used $\overline\mu_{x,\eta}^*(\mathscr T_n^c)<\epsilon/8K$ and the fact that $T(\zeta^*)\le KR_n^2$ almost surely.
    The above inequality and (\ref{eq:cap_zeta1}) together imply
    \begin{equation}\label{eq:cap_zeta2}
        \mathbb E[{\rm cap}(\zeta)]\ge\begin{cases}
        (1-\theta_1^2)(1-\frac3{16}\epsilon)\frac{\pi^2}8\frac{\mu_2(\rho_n)}{\mu_1(\rho_n)}\log^{-1}\mu_1(\rho_n),&d=4;\\
        (1-\theta_1^2)(1-\frac3{16}\epsilon)\varepsilon_d\frac{\mu_2(\rho_n)}{\mu_1(\rho_n)},&d\ge5.
        \end{cases}
    \end{equation}
    Noting that $e_\zeta(\widehat\zeta)={\rm cap}(\zeta)-e_\zeta({\rm range}(\zeta)\setminus\widehat\zeta)$, the estimate (\ref{eq:equi_lower}) holds true for sufficiently small $\theta_2$ and $\theta_1$ due to (\ref{eq:cap_zeta2}) and (\ref{eq:equi_setminus}).
\end{proof}

Now we turn to the lemmas of Section \ref{sec:5}.
\begin{proof}[Proof of Lemma \ref{le:intersect_cap}]
    We begin with the easier high-dimensional case. By the LCLT \cite[Theorem 1.2.1]{trove.nla.gov.au/work/15980712}, there is some $C=C(d)>0$ such that
    $$P^y[X_j=z]<Cj^{-d/2}$$
    for all $j>0$, $z\in\mathbb Z^d$. Hence, we have
    $$E^y\left[\left|X[T,KR_n^2]\cap A\right|\right]\le E^y[\sum_{j=T}^\infty 1_{X_j\in A}]\le C\sum_{j=T}^\infty|A|j^{-d/2}\le C|A|T^{1-d/2}.$$
    
    Now, we assume $d=4$. We divide the trajectory $X[T,KR_n^2]$ into pieces of length $T_n'=\lfloor R_n^2\log^{-1} R_n\rfloor$, and write $\xi_j=X[T+(j-1)T_n',T+jT_n']$, $x_j=\xi_j(0)$ for all $j\le\lceil KR_n^2/T_n'\rceil$. Note that $x_j$ and ${\rm cap}({\rm range}(\xi_j)\cap A,R_n^2)$ are independent, so we have
    \begin{align*}
        &E^y[{\rm cap}({\rm range}(\xi_j)\cap A,R_n^2)]\\
        =\;& E^y[{\rm cap}({\rm range}(\xi_j)\cap A,R_n^2)1_{x_j\in B(A,L_n)}]+E^y[{\rm cap}({\rm range}(\xi_j)\cap A,R_n^2)1_{x_j\notin B(A,L_n)}]\\
        \le\;&E^y[{\rm cap}({\rm range}(\xi_j),R_n^2\log^{-1}R_n)]\cdot P^y[x_j\in B(A,L_n)]+E^y[T_n'1_{\xi_j\not\subset B(x_j,L_n)}1_{x_j\notin B(A,L_n)}]\\
        \le\;&CR_n^2\log^{-2}R_n\cdot P^y[x_j\in B(A,L_n)]+T_n'\cdot P^{\bm0}[X[0,T_n']\not\subset B(0,L_n)]\\
        \le\;&CR_n^2\log^{-2}R_n|B(A,L_n)|(T+jT_n')^{-2}+T_n'\exp(-{L_n^2}/{2T_n'}),
    \end{align*}
    where in the last inequality we have used the LCLT and a standard estimate on the transition probability of the simple random walk.
    Take a summation over $j\in\{1,...,\lceil KR_n^2/T_n'\rceil\}$, and we obtain
    \begin{align*}
        &E^y[{\rm cap}(X[T,KR_n^2]\cap A,R_n^2)]\\
        \le\;& CR_n^2\log^{-2}R_n|B(A,L_n)|(T_n')^{-2}\sum_{j=1}^\infty(T/T_n'+j)^{-2}+CR_n^2\exp(-\frac12\log^{1-2c}R_n)\\
        \le\;&C|B(A,L_n)|T^{-1}\log^{-1} R_n+CR_n^2\exp(-\frac12\log^{1-2c}R_n).
    \end{align*}
    It remains to show that the second term in the last line is much smaller than the first term. This follows from the trivial bound
    $$|B(A,L_n)|\ge L_n^4=R_n^4\log^{-4c}R_n.$$
    Now, we have proved the four-dimensional case.
\end{proof}
\begin{proof}[Proof of Lemma \ref{le:free_varphi}]
    When $d\ge 5$, we have
    \begin{align*}
        &E^{y,z}[\varphi^{(\rho_n)}(X^1[0,K(R_n)^2],X^2[0,K(R_n)^2])]\\
        \le\;& E^{y,z}\left[\sum_{j,l=0}^{KR_n^2}g(X_j^1,X_l^2)\right]
        =\sum_{y_1,z_1\in\mathbb Z^2}\left(\sum_{j=0}^{KR_n^2}p_j(y,y_1)\right)g(y_1,z_1)\left(\sum_{j=0}^{KR_n^2}p_j(z_1,z)\right)\\
        =\;&\sum_{\substack{0\le t_1,t_2\le KR_n^2\\s\ge0}}\sum_{y_1,z_1\in\mathbb Z^d}p_{t_1}(y,y_1)p_s(y_1,z_1)p_{t_2}(z_1,z)        =\sum_{\substack{0\le t_1,t_2\le KR_n^2\\s\ge0}}p_{t_1+t_2+s}(y,z)\\
        =\;&\sum_{j=0}^{KR_n^2}j^2p_j(y,z)+K^2R_n^4\sum_{j>KR_n^2}p_j(y,z)        \le C\sum_{j=0}^{KR_n^2}j^{2-d/2}+CR_n^4\sum_{j>KR_n^2}j^{-d/2}.
    \end{align*}
    The conclusion then follows from direct computation of the series.

    When $d=4$, we proceed as in the last lemma, and let $\xi_j^i=X^i[(j-1)L_n,jL_n]$, $x_J^i=\xi_j^i(0)$ for $i=1,2,j>0$. For $U,V\subset\mathbb Z^4$, write $$\widetilde g(U,V)=\sup_{u'\in B(U,L_n^{1/2+c}),v'\in B(V,L_n^{1/2+c})}g(u',v'),$$
    and abbreviate $\widetilde g(u,v)=\widetilde g(\{u\},\{v\})$.
    Note that by (\ref{eq:Tcap_ex}) and (\ref{eq:rho_bound4}),
    \begin{align*}
        &E^{\bm0}[{\rm cap}^{(\rho_n)}(X[0,L_n])]\le L_n\log^{-10} R_n+E^{\bm0}[{\rm cap}(X[0,L_n],R_n^2)]\\
        \le\; &L_n\log^{-10}R_n+E^{\bm0}[{\rm cap}(X[0,L_n],L_n)]\le CL_n\log^{-1}R_n
    \end{align*}
    for some $C>0$.
    Thus, conditioning on $x_j^1,x_l^2$, such that $|x_j^1-x_l^2|>3L_n^{1/2+c}$, we have
    \begin{equation}\label{eq:varphi_bound1}\begin{aligned}
        &E^{y,z}[\varphi^{(\rho_n)}(\xi_j^1,\xi_l^2)|x_j^1,x_l^2]\\
        \le\;&E^{y,z}[{\rm cap}^{(\rho_n)}(\xi_j^1){\rm cap}^{(\rho_n)}(\xi_l^2)]\widetilde g(x_j^1,x_l^2)+2P^{\bm0}[X[0,L_n]\not\subset B(0,L_n^{1/2+c})]\cdot L_n^2\\
        \le\;&E^{\bm0}[{\rm cap}^{(\rho_n)}(X[0,L_n])]^2\widetilde g(x_j^1,x_l^2)+CL_n^2\exp\left(-\frac12L_n^{2c}\right)\\
        \le\;&CL_n^2\log^{-2}R_n\cdot g(x_j^1,x_l^2)+CL_n^2\exp\left(-\frac12L_n^{2c}\right).
    \end{aligned}\end{equation}

    Moreover, by the same argument as in the high-dimensional case, we obtain
    \begin{equation}\label{eq:varphi_bound2}\begin{aligned}
        E^{y,z}[\varphi^{(\rho_n)}(\xi_j^1,\xi_l^2)|x_j^1,x_l^2]\le CL_n.
    \end{aligned}\end{equation}
    Combining (\ref{eq:varphi_bound1}) and (\ref{eq:varphi_bound2}),
    \begin{equation}\label{eq:varphi_bound4}\begin{aligned}
        E^{y,z}[\varphi^{(\rho_n)}(\xi_j^1,\xi_l^2)]\le\;& CL_n^2\log^{-2}R_n \cdot E^{y,z}[g(x_j^1,x_l^2)]+CL_n^2\exp\left(-\frac12L_n^{2c}\right)\\&+CL_n P^{y,z}[|x_j^1-x_l^2|\le 3L_n^{1/2+c}].
    \end{aligned}\end{equation}
    Now, we bound the terms on the right-hand side of (\ref{eq:varphi_bound4}) from above. On the one hand, note that the following inequalities hold if $j,l\le KR_n^2/L_n$, $j+l>R_n\log^{-2+c}R_n$:
    \begin{equation*}
        P^{y,z}[|x_j^1-x_l^2|\le 3L_n^{1/2+c}]\le\frac{CL_n^{2+4c}}{(R_n^2\log^{-2}R_n)^2}\le CR_n^{-2+4c}\log^{4}R_n.
    \end{equation*}
    Thus, we have
    \begin{equation}\label{eq:var_bound5}\begin{aligned}
        \sum_{\substack{j,l\le KR_n^2/L_n\\j+l>R_n\log^{-2+c}R_n}}P^{y,z}[|x_j^1-x_l^2|\le 3L_n^{1/2+c}]\le\;& CR_n^{-2+4c}\log^4 R_n\cdot R_n^4L_n^{-2}\\\le\;& CR_n^{4c}\log^{4+2c}R_n.
    \end{aligned}\end{equation}
    On the other hand, applying twice the techniques in the proof of the case where $d\ge5$, we get
    \begin{equation}\label{eq:var_bound6}
        E^{y,z}\left[\varphi^{(\rho_n)}(X^1[0,R_n^2\log^{-2}R_n],X^2[0,R_n^2\log^{-2}R_n])\right]\le CR_n^2\log^{-2}R_n;
    \end{equation}
    \begin{equation}\label{eq:var_bound7}
        E^{y,z}\left[\sum_{j,l=1}^{\lceil KR_n^2/L_n\rceil}g(x_j^1,x_l^2)\right]\le CR_n^2L_n^{-2}.
    \end{equation}
    Therefore, by (\ref{eq:varphi_bound4})-(\ref{eq:var_bound7}),
    \begin{align*}
        &E^{y,z}\left[\varphi^{(\rho_n)}(X^1[0,KR_n^2],X^2[0,KR_n^2])\right]\\
        \le\;&E^{y,z}\left[\varphi^{(\rho_n)}(X^1[0,R_n^2\log^{-2}R_n],X^2[0,R_n^2\log^{-2}R_n])\right]
        +\sum_{\substack{j,l\le KR_n^2/L_n\\j+l>R_n\log^{-2+c}R_n}}E^{y,z}[\varphi^{(\rho_n)}(\xi_j^1,\xi_l^2)]\\
        \le\;&CR_n^2\log^{-2}R_n+CL_n^2\log^{-2} R_n\cdot E^{y,z}\left[\sum_{j,l=1}^{\lceil KR_n^2/L_n\rceil}g(x_j^1,x_l^2)\right]
        \\&+CL_n^2\exp\left(-\frac12L_n^{2c}\right)\cdot(KR_n^2/L_n)^2+CL_n\sum_{\substack{j,l\le KR_n^2/L_n\\j+l>R_n\log^{-2+c}R_n}}P^{y,z}[|x_j^1-x_l^2|\le 3L_n^{1/2+c}]\\
        \le\;&CR_n^2\log^{-2}R_n+CL_n^2\log^{-2}R_n\cdot R_n^2L_n^{-2}+C+CL_n\cdot R_n^{4c}\log^{4+2c}R_n\\
        \le\;& CR_n^2\log^{-2}R_n,
    \end{align*}
    proving the lemma.
\end{proof}

	\bibliographystyle{plain}
	\bibliography{references}

\end{document}